\documentclass[11pt]{article}
\usepackage[a4paper, total={6in, 10in}]{geometry}

\usepackage[english]{babel} 
\usepackage{authblk}

\usepackage{ourstyle} 
\usepackage{orcidlink,bbm,comment}
\usepackage{newtxtext} 
\usepackage[T1]{fontenc} 
\usepackage[disable]{todonotes}
\newcommand{\toyomu}[2][]{\todo[color=red!20!white, size=\footnotesize, #1]{TM: #2}}
\newcommand{\purba}[2][]{\todo[color=yellow!60!white, size=\footnotesize, #1]{PD: #2}}
\newcommand{\rafal}[2][]{\todo[color=green!60!white, size=\footnotesize, #1]{RL: #2}}

\begin{document}

\title{Level crossings of fractional Brownian motion}
\author[1]{Purba Das}
\author[2]{Rafał Łochowski}
\author[3]{Toyomu Matsuda \orcidlink{0000-0002-2422-0863}}
\author[3]{Nicolas Perkowski}
\affil[1]{King's College London, England. \protect\\ \href{mailto:purbadas@umich.edu}{purbadas@umich.edu}}
\affil[2]{Warsaw School of Economics, Poland. \protect\\
\href{mailto:rlocho@sgh.waw.pl}{rlocho@sgh.waw.pl}}
\affil[3]{Institut für Mathematik, Freie Universität Berlin, Germany. \protect\\
\href{mailto:toyomu.matsuda@fu-berlin.de}{toyomu.matsuda@fu-berlin.de}  and
\href{mailto:perkowski@math.fu-berlin.de}{perkowski@math.fu-berlin.de}}
\date{}

\maketitle

\begin{abstract}
 
Since the classical work of Lévy, it is known that the local time of Brownian motion can be characterized through the limit of level crossings. 
While subsequent extensions of this characterization have primarily focused on Markovian or martingale settings, this work presents a highly anticipated extension to fractional Brownian motion --- a prominent non-Markovian and non-martingale process.
Our result is viewed as a fractional analogue of Chacon et al. \cite{chacon1981}. 
Consequently, it provides a global path-by-path construction of fractional Brownian local time. 

Due to the absence of conventional probabilistic tools in the fractional setting, our approach utilizes completely different argument with a flavor of the subadditive ergodic theorem, combined with the shifted stochastic sewing lemma recently obtained in \cite{matsuda22}.

Furthermore, we prove an almost-sure convergence of the ($1/H$)-th variation of fractional Brownian motion with the \rafal{I added the} Hurst parameter $H$, along random partitions defined by level crossings, called Lebesgue partitions. This result raises an interesting conjecture on the limit, which seems to capture non-Markovian nature of fractional Brownian motion.
\bigskip



\end{abstract}

\tableofcontents

\section{Introduction}\label{sec:intro}
Level crossings of stochastic processes have been studied since the classical works of Kac \cite{Kac:1943aa} and Rice \cite{Rice:1945aa}. 
Depending on whether the process is smooth or rough, the study of its level crossings 
rely on different methods.
As for the smooth case, which is not the scope of this article, the reader can 
refer to the survey article \cite{Kratz:2006tb}, the textbook \cite{azais09} and the reference therein.  

By far the most prominent example of \rafal{I changed this sequence slightly}  a rough stochastic process  is a Brownian motion. 
The first work on level crossings of Brownian motion 
is attributed to Lévy \cite{levy1965}, who characterized its local time as \rafal{I changed this sequence slightly} a limit of normalised numbers of level crossings. 
More precisely, for a given process $w$, we set
\begin{equation*}
 U_{s, t} (\epsilon, w) 
 \assign \# \big\{ (u, v)  \of s \leq u < v \leq t, \,\,
   w_u = 0, w_v = \epsilon, \forall r \in (u, v) \: w_r \in (0,
   \epsilon) \big\}, 
\end{equation*}
as illustrated in Figure~\ref{fig:U_s_t}. 
The quantity $U_{s, t}(\epsilon, w)$ counts the number of upcrossings from $0$ to $\epsilon$ 
in the interval $[s, t]$.
For Brownian motion $W$ and $a \in \mathbb{R}$, we have
\begin{align*}
  \lim_{\epsilon \to 0} \epsilon U_{0, t}(\epsilon, W - a) = \frac{1}{2} L_t^W(a)
\end{align*}
almost surely, where $L_t^W(a)$ is the local time of $W$ at time $t$ and at level $a$. 
The local time is defined as the density of the occupation measure, see Definition \ref{def:local time} below.
This result can be found in standard textbooks such as \cite{ito_mckean,revuz_yor,Morters:2010vm},  
and it can be generalized for semimartingales \cite{El-Karoui:1978aa} and for Markov processes \cite{Fristedt:1983aa}.  
\begin{figure}[h]
  \raisebox{-0.5\height}{\includegraphics[width=\textwidth]{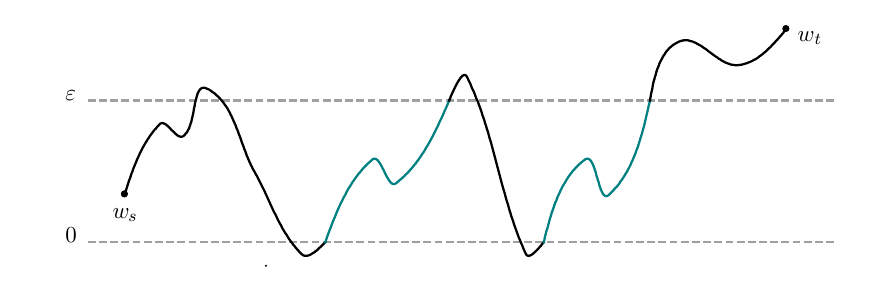}}
  \caption{Graphical illustration of $U_{s,t}(\epsilon, w)$. In the picture, $U_{s,t}(\epsilon, w) = 2$.}
  \label{fig:U_s_t}
\end{figure}

On the other hand, there exist rough stochastic processes that do not fall under the category of semimartingales or Markovian processes. One such example is the \emph{fractional Brownian motion} $B^H$, which is a Gaussian process parameterized by Hurst index $H \in (0, 1)$. Precisely, $B^H$ is neither a semimartingale nor Markovian for $H \neq 1/2$, and it becomes Brownian motion when $H = 1/2$.

Fractional Brownian motion possesses a local time \rafal{I added the reference to Definition \ref{def:local time}}(defined as in Definition \ref{def:local time}). Given Lévy's result on the local time of Brownian motion, a natural question arises regarding whether a similar result holds for the local time of fractional Brownian motion. However, a complete answer to this question has not been obtained thus far. This is surprising considering the age of Lévy's result and that of fractional Brownian motion.

Some works have explored the characterization of the fractional Brownian local time in relation to level crossings. For instance, the works \cite{azais90,Azais:1996aa} demonstrate that the number of zeros for certain smoothed fractional Brownian motions converges, in a suitable sense, to the local time. Furthermore, level crossings of stochastic processes have received attention in the context of pathwise stochastic calculus \cite{perkowski15, davis2018, cont19, RL2021, Kim:2022aa, ananova20}, as well as in some applied literatures \cite{feuerverger, Krug98}.

Constructing the local time via level crossings 
is not only a natural problem, but also it can lead to a significant implication on 
the path \rafal{I changed this sentence slightly} properties of the process.
This was first observed by Chacon et al. \cite{chacon1981} built on \cite{perkins81}. 
Therein proven is the existence of a measurable set $\Omega_W$ such that $\mathbb{P}(W \in \Omega_W) =1$ (recall that $W$ is Brownian motion) and 
for every $w \in \Omega_W$, $a \in \mathbb{R}$ and $t \geq 0$, the limit 
\begin{align}\label{eq:BM_loc_time_conv}
  \lim_{\epsilon \to 0} \epsilon U_{0, t}(\epsilon, w - a)
\end{align}
exists and is equal to one-half times the local time of $w$ at the level $a$. Hence, the existence of the limit of level crossings is a path property, 
and this provides a path-by-path construction of Brownian local time. 
Furthermore, the construction is global in that the convergence in \eqref{eq:BM_loc_time_conv} holds for all $a \in \mathbb{R}$ off a single null set.
This result explains why such construction of the local time receives attention in the pathwise stochastic calculus.
It is worth noting that the master thesis \cite{Lemieux_1983} of Lemieux extended the result for a large class of semimartingales.

The result of Chacon et al. has a remarkable consequence on \textit{pathwise quadratic variation}, calculated as a limit of sums of square increments, where the increments are taken along partitions of a fixed interval with vanishing mesh.  
The precise definition of the pathwise quadratic variation is as follows: given a sequence $\bm{\pi}$ of partitions $\pi_n$ ($n \in \mathbb{N}$) with vanishing mesh, the pathwise quadratic variation $[w]_{\bm{\pi}}$ of 
a process $w$ is defined by 
\begin{align*}
  [w]_{\bm{\pi}} \defby \lim_{n \to \infty} \sum_{[s, t] \in \pi_n} \abs{w_t - w_s}^2
\end{align*}
whenever the limit exists. 
A deterministic partition is a partition that does not depend on the path/process.
The classical works \cite{levy1940,levy1965} of Lévy show that for any refining deterministic partition sequence $\bm{\pi}$ of $[0, t]$ with vanishing mesh we have 
\begin{align}\label{eq:bm_qv}
  \mathbb{P}([W]_{\bm{\pi}} = t) = 1.
\end{align}
Dudley \cite{Dudley:1973aa} proved that if the deterministic sequence $\bm{\pi} = (\pi_n)$ satisfies  
\begin{align}\label{eq:dudley_cond}
  \abs{\pi_n} \defby \max_{[s, t] \in \pi_n} \abs{t-s} = o(1/\log n)
\end{align}
then \eqref{eq:bm_qv} holds, 
with the optimality of the condition \eqref{eq:dudley_cond} being shown as well. 
In general, the pathwise quadratic variation (even when it exists) may depend on the choice of a sequence of partitions \cite[page 47]{freedman}, \cite{davis2018}.
Furthermore, the null set in Lévy's and Dudley's works depends on the sequence $\bm{\pi}$ of partitions.

Hence, an obvious question is if, given a stochastic process $X$, there is any uncountable class $\mathrm{P}$ of partition sequences such that 
almost surely for any $\bm{\pi}, \bm{\pi}' \in \mathrm{P}$ we have $[X]_{\bm{\pi}} = [X]_{\bm{\pi}'}$. 
Here comes the result of Chacon et al., 
which proves
that Brownian motion has a measure zero set outside which any quadratic variation along any sequence of the \emph{Lebesgue partitions} 
(defined at the beginning of Section~\ref{subsec:main_results}) of 
$[0, t]$ with vanishing mesh is equal to $t$. 
We remark that, unlike in Dudley's result, there is \rafal{I changed this sequence} no condition on the decay of meshes of partitions and the null set is uniform over all sequences of the Lebesgue partitions.
\toyomu{To Purba: can you cite your works here?}

In this paper, 
we extend Levy's construction of the local time and the result of Chacon et al. for fractional Brownian motions with the Hurst parameter $H < 1/2$.  

\subsection{Main results}\label{subsec:main_results}
We write $B^H$ for a fractional Brownian motion with $B_0^H = 0$ and
with the Hurst parameter $H \in (0, 1)$. Specifically, $B^H$ is a one-dimensional centered Gaussian process 
satisfying $B^H_0 = 0$ and 
\begin{align*}
  \mathbb{E}[(B^H_t - B^H_s)^2] = a_H (t-s)^{2H}.
\end{align*}
We choose the constant $a_H$ in a manner that the identity \eqref{eq:mandelbrot} below holds. 
The specific value of $a_H$ is not important for our study.

To introduce the Lebesgue partitions,
let $\mathrm{P}$ be a partition of the space $\mathbb{R}$. That is, we have a strictly increasing sequence $(y_n)_{n\in \mathbb{Z}}$ of real numbers such that 
\begin{equation*}
    \lim_{n \to - \infty} y_n = - \infty, \quad \lim_{n \to +\infty} y_n = +\infty,
\end{equation*}
and we have $\mathrm{P} = \{ [y_{n-1}, y_n] : n \in \mathbb{Z} \}$.
\rafal{maybe it is in place to add here the definition of a partition we use?}
Let $\Lambda(\mathrm{P})$ be the set of all endpoints of intervals from $\mathrm{P}$, or $\Lambda(\mathrm{P}) = \{y_n : n \in \mathbb{Z}\}$.
Given a path
$w:[0,\infty) \to \mathbb{R}$, we set $T_0 (\mathrm{P}, w)
\assign 0$ and recursively define 
\begin{align}\label{eq:def_of_T_n}
 T_n (\mathrm{P}, w) \assign \inf \{ t > T_{n - 1} (\mathrm{P}, w) \of w_t
   \in \Lambda(\mathrm{P}) \setminus \{ w_{T_{n - 1} (\mathrm{P}, w)} \}
   \} 
\end{align}
with $\inf \emptyset = + \infty$.
(If $T_{n - 1} = + \infty$, we set $T_n \assign + \infty$.) Note that we do
not assume $w_0 = 0$. 
The partition given by
\begin{equation}\label{eq:lebesgue-partition}
  \{[T_{n-1}(\mathrm{P}, w), T_n(\mathrm{P}, w)] : n \in \mathbb{N}, T_n(\mathrm{P}, w) \leq t\}
\end{equation}
is called a \emph{Lebesgue partition}. 
See Figure~\ref{fig:T_n(epsilon,w)} for a graphical illustration.
Lebesgue partition is also called partition along $y$-axis.
\begin{figure}
  \centering
  \raisebox{-0.5\height}{\includegraphics[width=\textwidth]{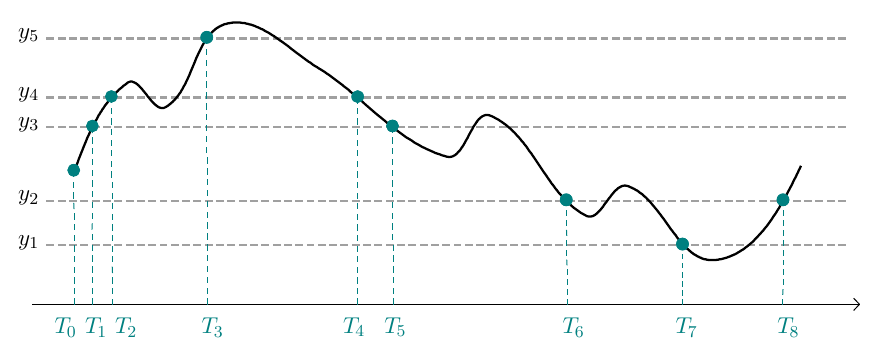}}
  \caption{Lebesgue partition}
  \label{fig:T_n(epsilon,w)}
\end{figure}

For a uniform partition $\mathrm{P}^{\epsilon} \defby \set{[\epsilon k, \epsilon (k+1)] \of k \in \mathbb{Z}}$, we simply write $T_n(\epsilon, w) \defby T_n(\mathrm{P}^{\epsilon}, w)$.
We denote by $K_{s,t}(\epsilon, w)$ the number of $\epsilon$-level crossings in the interval 
$[s, t]$, defined precisely as\footnote{Our convention is $\mathbb{N}= \{ 1, 2, 3, \ldots \}$.
In particular $0 \notin \mathbb{N}$.}
\begin{align}\label{eq:def_of_K}
  K_{s,t}(\epsilon, w) \defby 
  \# \{ n \in \mathbb{N} \setminus \{1\} \of T_n
   (\epsilon, w_{s + \cdot}) \leq t \} + 
   \indic_{\{w_s \in \epsilon \mathbb{Z}\}} \indic_{\{T_1(\epsilon, w_{s+\cdot}) \leq t\}},
\end{align}
where $\#A$ represents the cardinality of the 
set $A$.
%
We observe
\begin{equation*}
  \sum_{n: T_n(\epsilon, w) \leq t} \abs{w_{T_n(\epsilon, w)} - w_{T_{n-1}(\epsilon, w)}}^{\frac{1}{H}} 
  = \epsilon^{1/H} K_{0, t}(\epsilon, w) + 
  \indic_{\{w_0 \notin \epsilon \mathbb{Z}\}} \abs{w_{T_1(\epsilon, w)} - w_0}^{\frac{1}{H}}.
\end{equation*}
Note that $\abs{w_{T_1(\epsilon, w)} - w_0} \leq \epsilon$.
Therefore, the study of the ($1/H$)-th variation along a sequence of
uniform Lebesgue partitions is equivalent to that of limiting behavior of  $K_{0, t}(\epsilon, w)$ as $\epsilon \to 0$.
 

For $\rho \in \R$ and a process $w$, the process $w + \rho$ is defined by
$(w + \rho)_t \assign w_t + \rho$.
Our first main result is on the ($1/H$)-th variation of fractional Brownian motion along uniform Lebesgue partitions of fractional Brownian motion $B^H$.
\begin{theorem}[Convergence of ($1/H$)-th variation along uniform Lebesgue partitions]\label{thm:convergence-of-K}
For all $H \in (0, 1)$, there exists a positive finite constant $\mathfrak{c}_H$ with the following property.
  Let $\rho \in \R$, $t \in (0, \infty)$
  and $(\epsilon_n)_{n=1}^{\infty}$ be a sequence of positive numbers such that
  $\epsilon_n = O(n^{-\eta})$ for some $\eta > 0$. We then have
  \begin{equation}\label{eq:convergence-of-K}
    \lim_{n \to \infty} \epsilon_n^{1/H} K_{0, t}(\epsilon_n, B^H + \rho) = \mathfrak{c}_H t
    \quad \text{almost surely.}
  \end{equation}
\end{theorem}
The proof of Theorem~\ref{thm:convergence-of-K} will be given at the end of Section~\ref{subsec:var_conv}. A more explicit representation of the constant $\mathfrak{c}_H$ is given 
by \eqref{eq:const_c_H}.
For $H < 1/2$, we can remove the \rafal{I removed the word 'decaying'} condition $\epsilon_n = O(n^{-\eta})$, 
see Theorem~\ref{thm:lemieux_type_result} below.

Theorem~\ref{thm:convergence-of-K} concerns level crossings at all levels. 
We can also consider level crossings at a specific level, leading to Lévy's classical work on local time \cite{levy1965}.
For each $\epsilon \in (0, 1)$ and $w \in C ([0, T] ; \mathbb{R})$, we
consider the number of upcrossings \rafal{I changed this sentence
} by $w$ the space interval $[0, \epsilon]$ between times $s$ and $t$ by setting
\begin{align}\label{eq:def_of_U}
 U_{s, t} (\epsilon, w) \assign \# \big\{ (u, v) \in [s,t]^2 \of u < v, 
   w_u = 0, w_v = \epsilon, \forall r \in (u, v) \: w_r \in (0,
   \epsilon) \big\} . 
\end{align}
See Figure~\ref{fig:U_s_t} for illustration.
In the case of Brownian motion ($H=1/2$), it is well-known (e.g.,
\cite[Section~2.2]{ito_mckean}, \cite[Chapter~VI]{revuz_yor}, \cite[Section~6]{Morters:2010vm}) that 
provided $\epsilon_n \to 0$ and $\epsilon_n > 0$,\purba{Just to clarify: does $\downarrow 0$ means monotonically decreasing? I think we just need $\to 0$ and $\epsilon_n>0$ TM: corrected} we have
\begin{align}\label{eq:level_crossing_local_time_Brownian}
  \lim_{n \to \infty} \epsilon_n U_{0, t}(\epsilon_n, B^{1/2} - a) = \frac{1}{2} L_t^{1/2}(a)
  \quad \text{almost surely,}
\end{align}
where $L^{1/2}$ is the local time of Brownian motion $B^{1/2}$.
This representation of the local time was extended for semimartingales by Karoui~\cite{El-Karoui:1978aa}.

We recall that the local time of $B^H$, defined just below, exists for all $H \in (0, 1)$.
\begin{definition}[Local time]\label{def:local time}
  We denote by $(L_t^H (a))_{t \geq 0, a \in \mathbb{R}}$ the \emph{local time}
  of $B^H$ with $t$ representing time and $a$ representing level. 
  More precisely, $L^H$ is a unique random field satisfying the following occupation density formula:
  \begin{align*}
    \int_0^t f(B^H_r) \mathrm{d} r = \int_{\mathbb{R}} f(a) L_t^H(a) \mathrm{d} a, \quad \forall t\geq 0, \forall f \in C(\mathbb{R}). 
  \end{align*}
  As for the existence of $L^H$ and its continuity, see e.g., \cite{berman1973,Geman:1980ve}. 
\end{definition}
\toyomu{cite  Berman, S.M.: Local nondeterminism and local times of Gaussian processes. Bull. Am. Math. Soc.
79(2), 475–477 (1973). PD:added}
The following is an extension of Lévy's result to fractional Brownian motion.
\begin{theorem}[Local time via level crossings]\label{thm:convergence_to_local_time}
  Let $H < 1/2$, $a \in \R$, $t \in (0, \infty)$ and   
  $(\epsilon_n)_{n=1}^{\infty}$ be a sequence of positive numbers such that
  $\epsilon_n \to 0$.
  We then have 
  \begin{align}\label{eq:level_crossings_local_time_fbm}
    \lim_{n \to \infty} \epsilon_n^{\frac{1}{H} - 1} 
    U_{0, t}(\epsilon_n, B^H - a) = \frac{\mathfrak{c}_H}{2} L_t^H(a)
    \quad \text{almost surely,}
  \end{align}
where the constant $\mathfrak{c}_H$ is the same as in Theorem \ref{thm:convergence-of-K} and in Equation \eqref{eq:const_c_H}.
\end{theorem}
In principle, the null set of \eqref{eq:level_crossing_local_time_Brownian} 
or \eqref{eq:level_crossings_local_time_fbm} could potentially depend on the level $a$ 
and the sequence $(\epsilon_n)_{n=1}^{\infty}$.
However, an intriguing result by Chacon et al. \cite{chacon1981} demonstrates that for Brownian motion 
the null set can be chosen uniformly over the level $a$ and the sequence $(\epsilon_n)_{n=1}^{\infty}$. 
In the next theorem, we extend this result to fractional Brownian motion.

Before stating the theorem, we introduce some notation.
Analogously to $U_{s,t}$, we denote by $D_{s,t}$ the total number of downcrossings
\begin{align}\label{eq:def_of_D}
 D_{s, t} (\epsilon, w) \assign \# \big\{ (u, v) \in [s, t]^2 \of u < v,
   w_v = 0, w_u = \epsilon, \forall r \in (u, v) \: w_r \in (0,
   \epsilon) \big\} . 
\end{align}
Given a partition $\mathrm{P}$ of $\mathbb{R}$, we set 
\begin{align}\label{eq:def_vr_lebesgue}
  V_{s, t}(\mathrm{P}, w) \defby \sum_{[a, b] \in \mathrm{P}} (b-a)^{\frac{1}{H}} \left(U_{s,t}(b-a, w-a) + D_{s, t}(b-a, w-a)\right).
\end{align}
The quantity $V_{s, t}(\mathrm{P}, w)$ measures the ($1/H$)-th variation along the Lebesgue partition defined by $\mathrm{P}$.
For the uniform partition $\mathrm{P}^{\epsilon}$, 
we have 
\begin{equation*}
  V_{0, t}(\mathrm{P}^{\epsilon}, w)  = \epsilon^{1/H} K_{0, t}(\epsilon, w) + 
  \indic_{\{w_0 \notin \epsilon \mathbb{Z}\}} \abs{w_{T_1(\epsilon, w)} - w_0}^{\frac{1}{H}}.
\end{equation*}
\begin{theorem}[Fractional analogue of Chacon et al.]\label{thm:lemieux_type_result}
Let $H < 1/2$ and let $\mathfrak{c}_H$ be the constant of Theorem~\ref{thm:convergence-of-K}. Then, there exists a measurable set   $\Omega_H \subseteq C([0, \infty); \mathbb{R})$ with the following property.
  \begin{itemize}
    \item $\mathbb{P}(B^H \in \Omega_H) = 1$.
    \item For every $w \in \Omega_H$ and $T \in (0, \infty)$, we have 
      \begin{align*}
        \lim_{\epsilon \to 0, \epsilon > 0}
        \sup_{\substack{\mathrm{P} : \abs{\mathrm{P}} \leq \epsilon, \\  t \leq T}} 
        \abs{V_{0, t}(\mathrm{P}, w) - \mathfrak{c}_H t} = 0. 
      \end{align*}
    \item For every $w \in \Omega_H$, there exists a continuous map 
      \begin{align*}
        [0, \infty) \times \mathbb{R} \ni (t, a) \mapsto l_t(w, a) \in [0, \infty)
      \end{align*}
      such that for every $T \in (0, \infty)$ we have
      \begin{align*}
        \lim_{\epsilon \to 0, \epsilon > 0} \sup_{a \in \mathbb{R}, t \leq T} 
        \abs[\Big]{\epsilon^{\frac{1}{H} - 1} U_{0, t}(\epsilon, w - a) - \frac{\mathfrak{c}_H}{2} l_t(w, a)} = 0. 
      \end{align*}
      Furthermore,  the occupation density formula holds: 
      \begin{align*}
        \int_0^t f(w_r) \mathrm{d} r = \int_{\mathbb{R}} f(a) l_t(w, a) \mathrm{d} a, 
        \quad \forall t\geq 0, \forall f \in C(\mathbb{R}). 
      \end{align*}
  \end{itemize}
\end{theorem}

\begin{proof}
  It follows from Theorem~\ref{thm:lemieux_local_time} and Theorem~\ref{thm:lemieux_variation}.
\end{proof}
In \cite{chacon1981} the corresponding result for Brownian motion is interpreted as a \emph{quadratic arc length}. By analogy, we could say that the fractional Brownian motion of Hurst index $H<1/2$ has \emph{$1/H$ arc length} $\mathfrak{c}_H t$. Note that the $1/H$ arc length is a purely geometric, path-dependent quantity which is invariant under translations and time-reparametrizations.

\purba[inline]{The main proof uses the SW argument on $\bar K$. So don't we also have $\epsilon_n^{\frac{1}{H}}TV^{\epsilon_n}\to [B^H]^{(1/H)}$ along lebesgue partition almost surely? Or am I missing something? 
TM: It easily follows from the second item.}

\begin{remark}
  In \cite{matsuda22} the following quantitative estimate was proved. There exists an $\epsilon
  \in (0, 1)$ such that for every $m \in [1, \infty)$ we have
  \[ \Big\| \sum_{[u, v] \in \pi} | B^H_v - B^H_u |^{\frac{1}{H}} -\mathbb{E}
     \big[ | B^H_1 |^{\frac{1}{H}} \big] t \Big\|_{L^m (\mathbb{P})}
     \lesssim_{m, t} | \pi |^{\epsilon}  \]
  for any deterministic partition $\pi$ of $[0, t]$.
  Therefore, by the Borel--Cantelli lemma, for any sequence $(\pi_n)_{n =
  1}^{\infty}$ of deterministic partitions with
  \begin{equation}\label{eq:fbm_partition_decay}
   | \pi_n | = O (n^{- \delta}), \quad \delta \in (0, \infty), 
  \end{equation}
  we have $\lim_{n \to \infty} \sum_{[u, v] \in \pi^n} | B^H_v - B^H_u
  |^{\frac{1}{H}} =\mathbb{E} [ | B^H_1 |^{\frac{1}{H}} ] t$ almost
  surely.
  Unlike Theorem~\ref{thm:lemieux_type_result}, we need the decaying condition \eqref{eq:fbm_partition_decay}. 
  The work \cite{Dudley:1973aa} shows that
  the condition \eqref{eq:fbm_partition_decay} is not optimal for the Brownian case; 
  finding the optimal condition for the fractional case seems open.
\end{remark}
\subsection{Conjecture}
There is an interesting aspect on the constant $\mathfrak{c}_H$. 
For Brownian motion, the quadratic variation along any deterministic partition almost surely matches with the quadratic variation along any Lebesgue partitions. 
That is, 
\begin{align}\label{eq:c_H_Brownian}
  \mathfrak{c}_{\frac{1}{2}} = \mathbb{E}[(B^{1/2}_1)^2].
\end{align}
It is tempting to guess that such relation holds for $H \neq 1/2$ as well.
Indeed, such conjecture is stated in \cite[after Lemma~3.5]{cont19}.
However, the identity \eqref{eq:c_H_Brownian} is due to 
the strong Markov property of Brownian motion. 
Therefore, for $H \neq 1/2$, there is no reason why 
$\mathfrak{c}_H$ and $\mathbb{E}[\abs{B^H_1}^{\frac{1}{H}}]$ should be equal.
Motivated by the simulation shown in Figures~\ref{fig:lebesgue_partition_variation}, 
we propose the following conjecture.
\begin{conjecture}[The constant $\mathfrak{c}_H$]
  For fractional Brownian motion with $H\neq 1/2$, 
  we conjecture that the ($1/H$)-th variation of 
  fractional Brownian motion along deterministic partitions differs 
  from the ($1/H$)-th variation of fractional Brownian motion along uniform Lebesgue partitions. 
  To be more precise, we conjecture
  \begin{equation*}
    \begin{cases}
     \mathfrak{c}_H > \mathbb{E}[\abs{B_1^H}^{1/H}] & \text{if } H < 1/2, \\
     \mathfrak{c}_H < \mathbb{E}[\abs{B_1^H}^{1/H}] & \text{if } H > 1/2.
    \end{cases}
  \end{equation*}
  If this is indeed the case, the constant $\mathfrak{c}_H$ captures 
  non-Markovian nature of fractional Brownian motion.
\end{conjecture}

\begin{figure}
  \centering
  \begin{subfigure}{\textwidth}
    \centering
\subfloat[$H = 0.4$] {  \includegraphics[width=.49\textwidth]{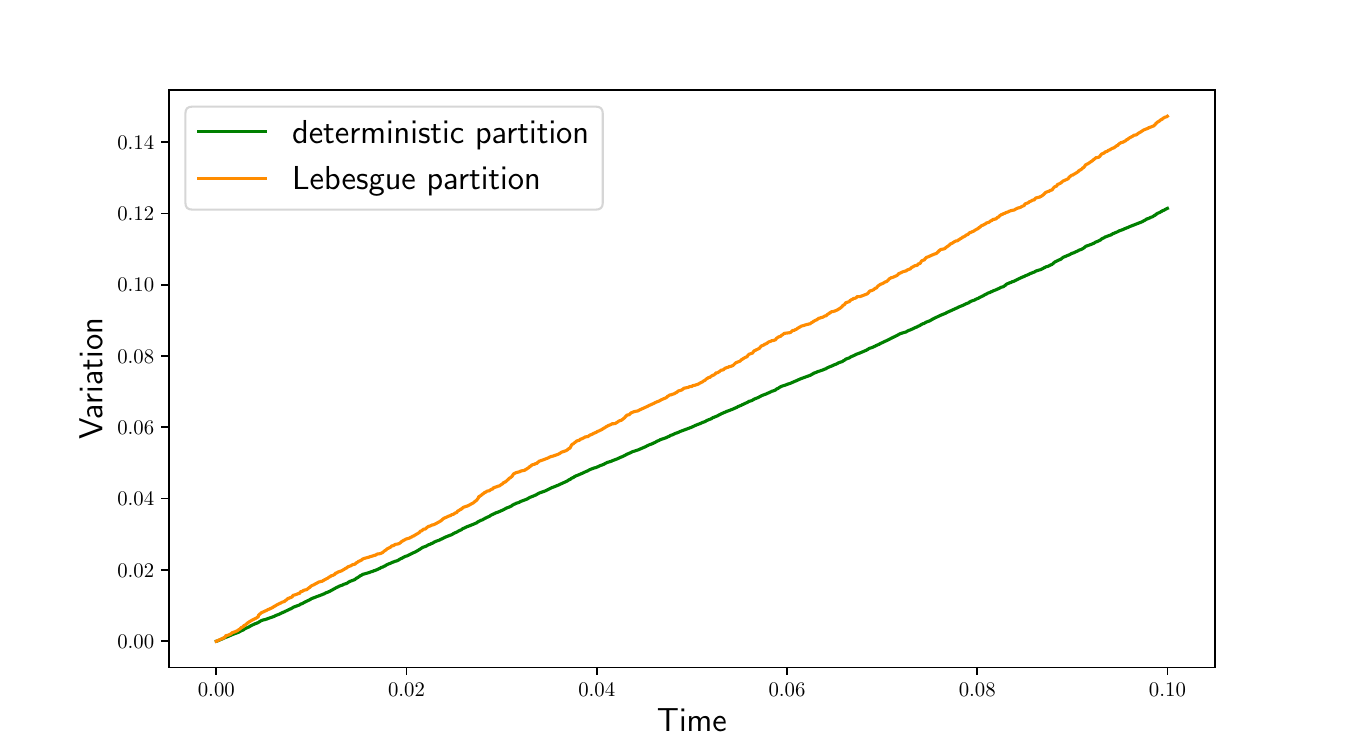}
\label{fig:lebesgue_partition_H_0.4}}
\subfloat[$H = 0.6$]{
 \includegraphics[width=.49\textwidth]{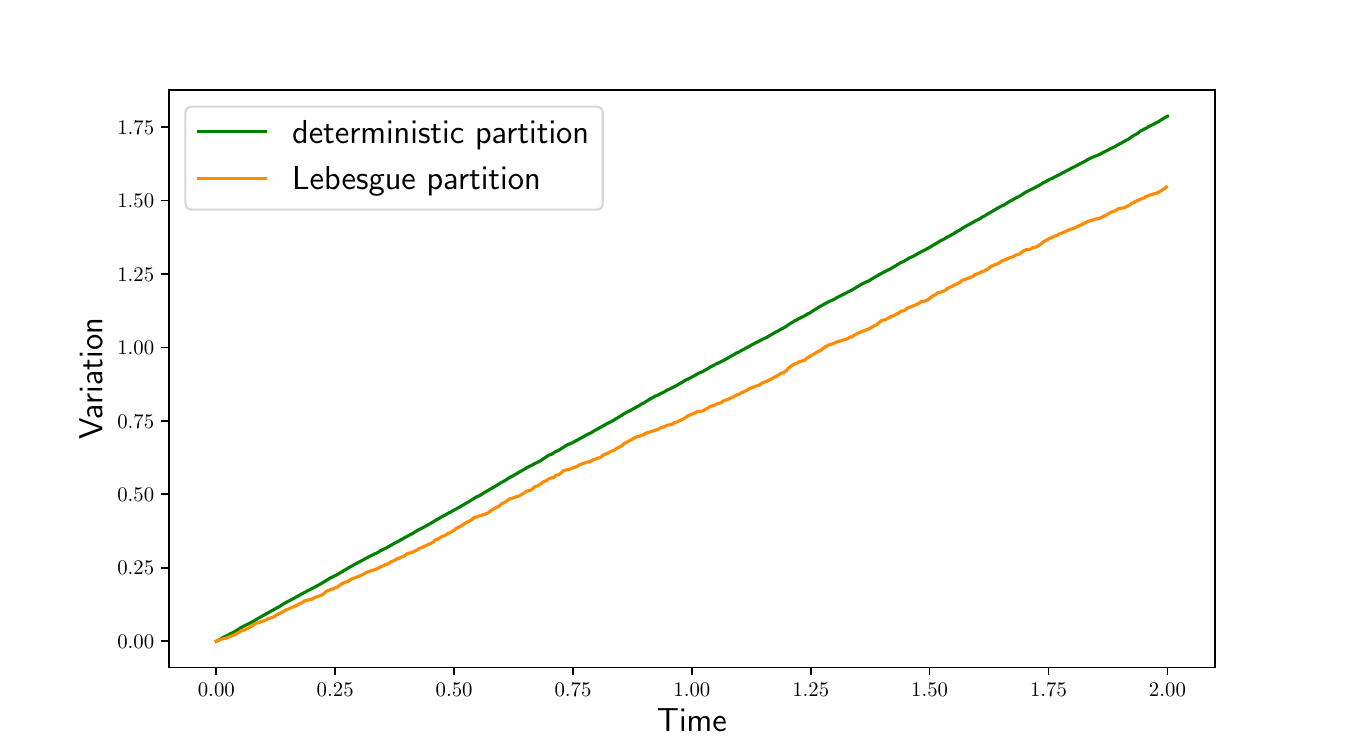}
}
    \label{fig:lebesgue_partition_H_0.6}
  \end{subfigure}
  \captionsetup{singlelinecheck=off}
  \caption[]{Comparison between the variation along a deterministic uniform partition and 
  that along a Lebesgue partition. 
  }
  \label{fig:lebesgue_partition_variation}
\end{figure}
\begin{SCfigure}[0.7]
  \caption{Comparison between the variation along a deterministic uniform partition and that along a Lebesgue partition for Brownian motion.}
  \includegraphics[width=0.49\textwidth]{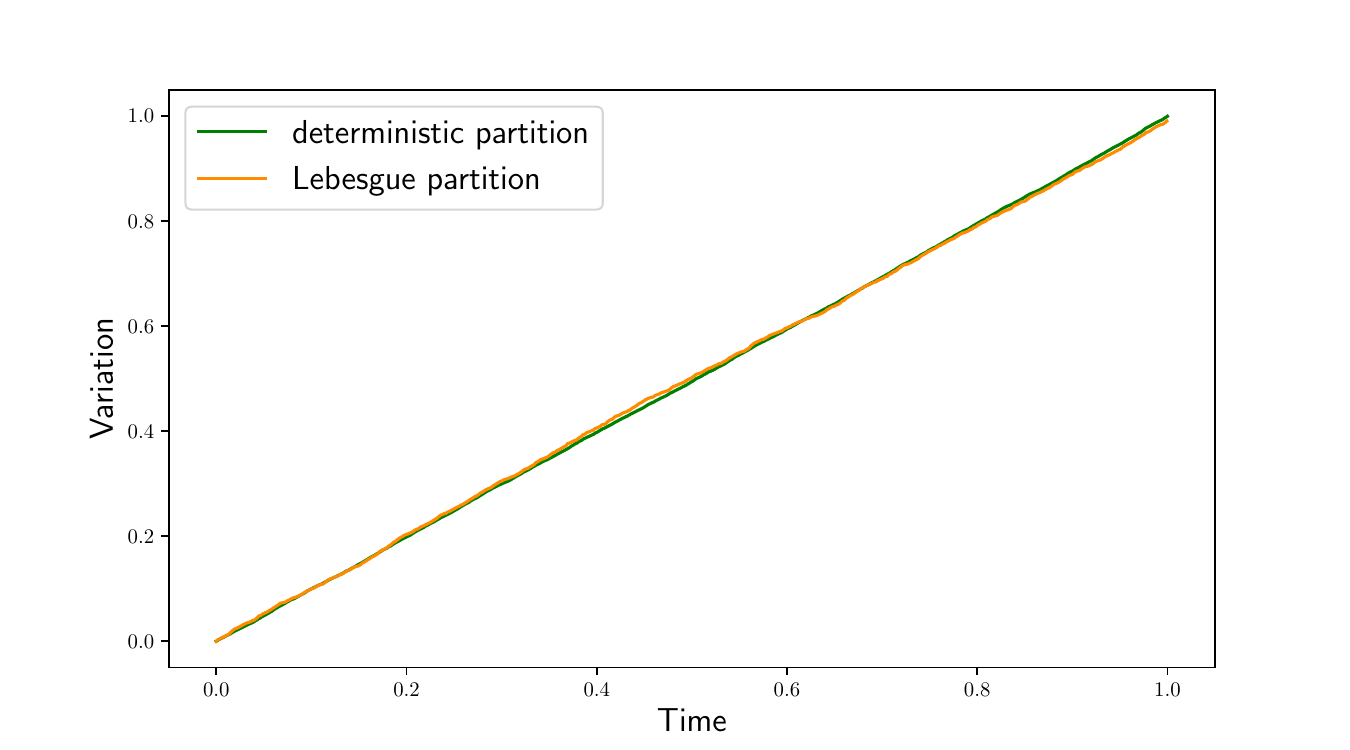}
  \label{fig:lebesgue_part_BM}
\end{SCfigure}
\begin{remark}
  We comment on the simulation for Figure~\ref{fig:lebesgue_partition_variation}.
  The variation, denoted as $V_t$, is given by
  \begin{displaymath}
    V_t \defby \sum_{[u, v] \in \pi^{\#}, v \leq t} \abs{B^H_v - B^H_u}^{\frac{1}{H}},
  \end{displaymath}
  up to time $T$, where $\# \in \{\text{deterministic}, \text{Lebesgue}\}$.
  We discretize the fractional Brownian motion with a step size of $T/n$.
  For the deterministic partition, we use the partition $\pi^{\text{deterministic}} = \{kT/n\}_{k=1}^n$,
  while for the Lebesgue partition, we utilize the partition $\pi^{\text{Lebesgue}} = \{T_k(\epsilon, B)\}_k$.
  To simulate the variation $V$ along a Lebesgue partition, it is crucial to appropriately choose the parameters $T$, $n$ and $\epsilon$ to ensure that the partition is neither too dense nor too sparse.
  For instance, when $H = 0.4$, we simulate $V$ with $T = 0.1$, $n = 30000$, and $\epsilon = 0.015$.
  Similarly, for $H = 0.6$, we simulate $V$ with $T = 2$, $n = 30000$, and $\epsilon = 0.013$.
  Figure \ref{fig:lebesgue_part_BM} represents a similar figure for Brownian motion (i.e. $H=1/2$). 
  The figure is generated with $T=1, n=30000, \epsilon=0.014$.
  The figure supports the fact that for a Brownian motion the quadratic variation along deterministic partitions is same as that along the Lebesgue partitions. 
\end{remark}
\subsection{Strategy of the proofs}
Let us outline our strategy to prove Theorem~\ref{thm:convergence-of-K}. 
The first observation is that it is easier to deal with an averaged version of $K$. Namely, we define 
\begin{align*}
  \bar{K}_{s, t}(\epsilon, w) \defby \epsilon^{-1} \int_{-\epsilon/2}^{\epsilon/2} K_{s, t}(\epsilon, w + \rho) \mathrm{d} \rho.
\end{align*}
The quantity $\bar{K}$ is related to \emph{truncated variation}, see Remark \ref{rem:truncated_var} below. 
By scaling (Lemma~\ref{lem:K-bar-subadditive-and-stationary}), we have 
\begin{align*}
  \bar{K}_{0, 1}(\epsilon, B^H) = \bar{K}_{0, \epsilon^{-1/H}}(1, B^H) \quad \text{in law.}
\end{align*}
Therefore, at least in law, the limiting behavior of $\bar{K}_{0, 1}(\epsilon, B^H)$ as $\epsilon \to 0$ is the same as that of $\bar{K}_{0, \epsilon^{-1/H}}(1, B^H)$. 
Rewriting $T = \epsilon^{-1/H}$, we thus hope to see the limit of $T^{-1} \bar{K}_{0, T}(1, B^H)$.

Regarding this, we are actually in the setting of the superadditive ergodic theorem. 
Indeed, it is not difficult to see that the family $(\bar{K}_{s,t}(1, B^H))_{s<t}$ is stationary and superadditive (Lemma~\ref{lem:K-bar-subadditive-and-stationary}). 
Therefore, the limit 
\begin{align*}
  \lim_{T \to \infty} \frac{1}{T} \bar{K}_{0, T}(\epsilon, B^H) = \sup_{T > 0} \frac{1}{T} \mathbb{E}[\bar{K}_{0, T}(\epsilon, B^H)] 
\end{align*}
exists almost surely and in $L^1(\mathbb{P})$. It turns out that the limit is the constant $\mathfrak{c}_H$ in Theorem~\ref{thm:convergence-of-K}.

To take advantage of this finding on $\bar{K}$, we can naively guess that 
\begin{align}\label{eq:K_different_rho}
  \abs{K_{0, 1}(\epsilon, B^H + \rho) -  K_{0, 1}(\epsilon, B^H + \rho')}
\end{align}
is small provided that $\abs{\rho - \rho'}$ is small. It is therefore expected to have 
\begin{align}\label{eq:K_bar_K_limit_equal}
  \lim_{\epsilon \to 0} \epsilon^{\frac{1}{H}} K_{0, 1}(\epsilon, B^H) = \lim_{\epsilon \to 0} \epsilon^{\frac{1}{H}} \bar{K}_{0, 1}(\epsilon, B^H) = \mathfrak{c}_H.
\end{align}
However, proving \eqref{eq:K_bar_K_limit_equal} requires a non-trivial argument. In fact, the map 
  $\rho \mapsto K_{0, 1}(\epsilon, w + \rho)$
can be highly discontinuous, and it is impossible to estimate \eqref{eq:K_different_rho} pathwisely. 

This inherent difficulty motivates us to employ a \emph{probabilistic} argument in order to prove Theorem~\ref{thm:convergence-of-K}. In addition to superadditivity, a crucial ingredient is the stochastic sewing lemma introduced by Lê \cite{le20}, which provides effective estimates for stochastic Riemann sums. For our specific problem, we require an extension of this lemma, called the shifted stochastic sewing lemma, recently obtained by the third and fourth authors \cite{matsuda22}. This extension is particularly suitable for capturing asymptotic decorrelation in stochastic Riemann sums.
It is worth mentioning that there are already some studies that leverage the stochastic sewing lemma for analyzing local times \cite{Harang:2021aa,Harang:2022aa,butkovsky2023stochastic}.

Our proof of Theorem~\ref{thm:convergence-of-K} proceeds as follows. Not only is the family $(K_{s,t}(\epsilon, B^H))_{s < t}$ superadditive, but it is also almost subadditive:
\begin{align*}
K_{s, t}(\epsilon, B^H) \leq K_{s, u}(\epsilon, B^H) + K_{u, t}(\epsilon, B^H) + 1, \quad s < u < t.
\end{align*}
This leads to an approximation:
\begin{align*}
\epsilon^{1/H} K_{0, 1}(\epsilon, B^H) \approx \sum_{[s, t] \in \pi_{\epsilon}} \epsilon^{1/H} K_{s, t}(\epsilon, B^H),
\end{align*}
where $\pi_{\epsilon}$ is a uniform partition of the interval $[0, 1]$ with $\lvert \pi_{\epsilon} \rvert \approx \epsilon^{1/H}$. A similar approximation holds for $\bar{K}$. Thus, we obtain
\begin{align*}
\epsilon^{1/H} K_{0, 1}(\epsilon, B^H) - \epsilon^{1/H} \bar{K}_{0, 1}(\epsilon, B^H) \approx \sum_{[s, t] \in \pi_{\epsilon}} \epsilon^{1/H} [K_{s, t}(\epsilon, B^H) - \bar{K}_{s, t}(\epsilon, B^H)].
\end{align*}
The Riemann sum on the right-hand side can then be estimated using the shifted stochastic sewing lemma, which ultimately yields the desired convergence result.

To establish the convergence to the local time (Theorem~\ref{thm:convergence_to_local_time}), we follow a similar line of argument. However, the technical difficulty increases substantially due to the lack of a counterpart to $\bar{K}$. We postpone an overview of this technical argument to Section~\ref{subsec:heuristic}.

One notable strength of the stochastic sewing lemma is its ability to provide a quantitative bound. Thanks to this property, we can obtain a quantitative version of Theorem~\ref{thm:convergence_to_local_time}, as stated in Theorem~\ref{thm:local-time-level-crossing}. This quantitative result enables us to employ the pathwise argument of Chacon et al. \cite{chacon1981}, which elevates the convergence result from Theorem~\ref{thm:convergence_to_local_time} to the more refined version of Theorem~\ref{thm:lemieux_type_result}.

\subsubsection*{Outline} 
The organization of the paper is as follows. 
Theorem~\ref{thm:convergence-of-K} is proven in Section~\ref{sec:variation}, 
while Theorem~\ref{thm:convergence_to_local_time} and Theorem~\ref{thm:lemieux_type_result} 
are proven in Section~\ref{sec:local_time}.

In Section~\ref{subsec:elementary}, we derive some elementary results on $K$, in Section~\ref{subsec:shifted_ssl}, 
we recall our key ingredient called the shifted stochastic sewing lemma, and in Section~\ref{subsec:var_conv}, we 
give our proof of Theorem~\ref{thm:convergence-of-K}. 

In Section~\ref{sec:local_time}, after stating a quantitative version of Theorem~\ref{thm:convergence_to_local_time} as 
Theorem~\ref{thm:local-time-level-crossing}, we review our technical strategy in Section~\ref{subsec:heuristic}. 
Section~\ref{subsec:local_time_key} is the most demanding part of the paper, which proves Theorem~\ref{thm:local-time-level-crossing} with the key estimate being Lemma~\ref{lem:U_conditioning}. 
As a consequence of Theorem~\ref{thm:local-time-level-crossing}, we complete the proof of Theorem~\ref{thm:lemieux_type_result} in Section~\ref{subsec:lemieux}. 
\subsubsection*{Notation} 
Given a path $f \from [0, T] \to \mathbb{R}^d$, we write $f_{s, t} \defby f_t - f_s$, and 
we denote by $\dot{f}$ the derivative $\frac{\mathrm{d} f}{\mathrm{d} t}$. 
We fix a probability space $(\Omega, \mathcal{F}, \mathbb{P})$ on which fractional Brownian motion is defined, 
we write $\mathbb{E}$ for the expectation with respect to $\mathbb{P}$, and we write 
\begin{equation*}
    \norm{F}_{L^p(\mathbb{P})} \defby \Big( \int_{\Omega} \abs{F}^p \mathrm{d} \mathbb{P} \Big)^{1/p}
\end{equation*}
with usual modification for $p = \infty$.
The expression $X \dequal Y$ means that the random variables $X$ and $Y$ have the same law.
We write $A \lesssim A'$ if there exists a positive constant $C$, depending only on some (unimportant) parameters, such that $A \leq C A'$.  If we want to emphasize the dependency on parameters $\alpha, \beta, \ldots$, then we write $A \lesssim_{\alpha, \beta, \ldots} A'$.  
The following objects appear throughout the paper. 
\begin{itemize}
  \item We write $B = B^H$ for fractional Brownian motion in one dimension. 
    Unlike Section~\ref{sec:intro}, we mostly suppress the script $H$ (i.e., we simply write $B$ for fractional Brownian motion).
    {\bfseries In this paper we will not write down dependency on $H$.} For instance, when we write $A \lesssim A'$, the proportional constant may depend on $H$.
  \item We denote by $K_{s, t}(\epsilon, w)$ the total number of $\epsilon$-level crossings in the interval $[s, t]$, as precisely defined at \eqref{eq:def_of_K}. 
    We denote by $U_{s, t}(\epsilon, w)$ the total number of upcrossings from $0$ to $\epsilon$ in the interval $[s, t]$, as precisely defined in \eqref{eq:def_of_U}.
  \item We write $(L_t^H(a))_{t \geq 0, a \in \mathbb{R}} = (L_t(a))_{t \geq 0, a \in \mathbb{R}}$ for the local time of fractional Brownian motion $B$, see Definition~\ref{def:local time}.
\end{itemize}

\subsection*{Acknowledgement}
Our work is aided by the Deutsche Forschungsgemeinschaft (DFG) through the Berlin--Oxford IRTG 2544 program, which enabled PD to visit Berlin and financially supports TM. R{\L}'s research was partially funded by the National Science Centre, Poland, under Grant No. 2019/35/B/ST1/04292.
\toyomu{Maybe comment on exchange program. PD: added}

\section{Variation along uniform Lebesgue partitions}\label{sec:variation}
The goal of this section is to prove Theorem~\ref{thm:convergence-of-K}.  
We begin observing elementary results on the counting $K$ of level crossings, defined by \eqref{eq:def_of_K}.  
\subsection{Elementary results}\label{subsec:elementary}
Let us first recall the Mandelbrot--Van Ness representation of fractional Brownian motion \cite{mandelbrot68}, which will be used throughout. 

\begin{definition}[Fractional Brownian motion]
  \label{def:mandelbrot_van_ness}
  We set 
  \begin{align}\label{eq:kernel_def}
    \kernel(t, s) \defby (t-s)^{H-\frac{1}{2}} - (-s)_+^{H - \frac{1}{2}}, 
    \quad s < t.  
  \end{align} 
  Let $W = (W_t)_{t \in \mathbb{R}}$ be a two-sided Brownian motion in one dimension.
  Throughout the paper, we suppose that fractional Brownian motion $B = B^H$ has the \emph{Mandelbrot--Van Ness representation} 
  \begin{align}\label{eq:mandelbrot}
    B_t = \int_{-\infty}^{t} \kernel(t, s) \mathrm{d} W_s.
  \end{align}
  Note that we have $\mathbb{E}[(B_t - B_s)^2] = a_H (t-s)^{2 H}$ 
  for some constant $a_H$, whose actual value is irrelevant for us.
\end{definition} 

\begin{lemma}[Scaling of $K$]
  \label{lem:K-scaling} 
  For $\lambda \in (0, \infty)$, we have
  \[ (K_{s, t} (\epsilon, B + \rho))_{s<t,\; \epsilon > 0, \rho \in \mathbb{R}}
     \dequal (K_{\lambda^{1/H} s, \lambda^{1 / H} t} (\lambda \epsilon, B + \lambda
     \rho))_{s < t,\; \epsilon > 0,  \rho \in \mathbb{R}} . \]
\end{lemma}
\begin{proof}
  We set $B^{(\lambda)}_t \assign \lambda B_{\lambda^{- 1 / H} t}$. Note
  that $B^{(\lambda)} \dequal B$ and observe that
  \begin{equation*}
    K_{s, t} (\epsilon, B_t + \rho)  =  K_{s, t} (\lambda \epsilon, \lambda
    (B + \rho)) \text{}
     =  K_{\lambda^{1/H}s, \lambda^{1/H}t} (\lambda \epsilon, B^{(\lambda)} +
    \lambda \rho). \qedhere
  \end{equation*}
\end{proof}
\begin{lemma}[Superadditivity of $K$]
  \label{lem:K-subadditive}
  Let $r < s < t$ and $w$ be a path. Then,
  \[ K_{r, s} (\epsilon, w) + K_{s, t} (\epsilon, w) 
    \leq K_{r, t} (\epsilon, w) 
    \leq K_{r, s} (\epsilon, w) + K_{s, t} (\epsilon, w) + 1. \]
\end{lemma}

\begin{proof}
  Without loss of generality, we set $r = 0$.
  Recalling the definition of $T_n$ from \eqref{eq:def_of_T_n}, we set
  \[ N \assign \max \{ n \in \mathbb{N} \cup \{ 0 \} \of  T_n (\epsilon,
     w) \leq s \} . \]
  \begin{figure}[h]
    \centering
    \includegraphics[width=\textwidth]{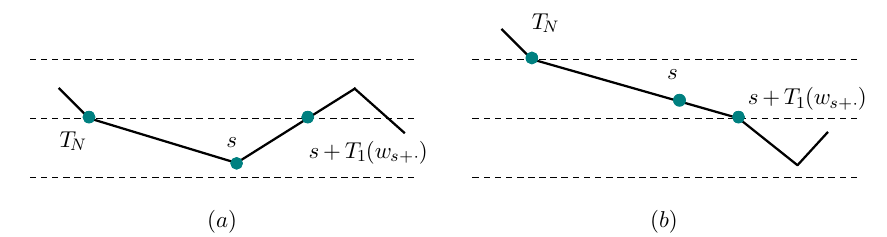}
    \caption{Level crossings around the middle time $s$. In case (b), the level crossing around $s$ is not counted in $K_{r, s}(\epsilon, w) + K_{s, t}(\epsilon, w)$.}
    \label{fig:K_superadditive}
  \end{figure}
  As shown in Figure~\ref{fig:K_superadditive}, we have the following cases.
  \begin{enumerate}[label=(\alph*)]
    \item If $w_{T_N (\epsilon, w)} = w_{T_1 (\epsilon, s+w_{s +
    \cdot})}$ or if $w_{T_N(\epsilon, w)} = w_s$, then $K_{r, t} (\epsilon, w) = K_{r, s} (\epsilon, w) +
    K_{s, t} (\epsilon, w)$.

    \item Otherwise, we have $K_{r, t} (\epsilon, w) = K_{r, s} (\epsilon, w)
    + K_{s, t} (\epsilon, w) + 1$. \qedhere
  \end{enumerate}
\end{proof}

For our arguments, the following variants of $K$ will appear.
\begin{notation}[Averaged $K$]\label{not:K-bar-and-J}
  We set 
  \begin{equation*}
    \bar{K}_{s, t}(\epsilon, w) \assign \epsilon^{-1} \int_{-\epsilon/2}^{\epsilon/2} K_{s, t}(\epsilon, w+\rho) \dd \rho.
  \end{equation*}
\purba{$\bar{K}$ is actually $\epsilon TV^\epsilon_{r,s}(\omega)$, which is something know. Should we use $TV^\epsilon$ notation?}
\toyomu{If $\epsilon$ remains, it seems nicer to use $\bar{K}$, but we should of course mention the relation.}
\end{notation}
\begin{remark}\label{rem:truncated_var}
  The quantity $\bar{K}$ is related to the so-called \emph{truncated variation} \cite{lochowski2008} defined by 
  \begin{align*}
    \operatorname{TV}^{\epsilon}(w, [s, t]) 
    \defby \sup_{\pi : \text{partition of } [s, t]} \sum_{[u, v] \in \pi} 
    \max \{ \abs{w_v - w_u} - \epsilon, 0 \}.
  \end{align*}
  Indeed, as shown in \cite[Theorem~1]{lochowski2017}, we have the identity 
  \begin{align*}
    \operatorname{TV}^{\epsilon}(w, [s, t]) = \epsilon^{-1} \bar{K}_{s,t}(\epsilon, w).
  \end{align*}
\end{remark}
The advantage of $\bar{K}$ is that in addition to the superadditivity, it is stationary.

\begin{lemma}[Scaling, superadditivity and stationarity of $\bar{K}$]
  \label{lem:K-bar-subadditive-and-stationary}
  Let $r < s < t$ and let $w$ be a process.
  \begin{enumerate}[label=(\roman*)]
    \item For $\lambda>0$, we have 
    \begin{align*}
      (\bar{K}_{s, t} (\epsilon, B))_{s<t, \epsilon>0}
    \dequal (\bar{K}_{\lambda^{1/H} s, \lambda^{1 / H} t}(\lambda \epsilon, B))_{s<t, \epsilon>0}.
    \end{align*}
    \item We have 
    \begin{align*}
      \bar{K}_{r, s}(\epsilon, w) + \bar{K}_{s, t}(\epsilon, w) \leq \bar{K}_{r, t}(\epsilon, w) 
      \leq \bar{K}_{r, s}(\epsilon, w) + \bar{K}_{s, t}(\epsilon, w) + 1. 
    \end{align*}
    \item We have $\bar{K}_{s, t} (\epsilon, w) = \bar{K}_{0, t-s} (\epsilon, w_{s +
    \cdot} - w_s)$. In particular, 
    \begin{align*}
      \bar{K}_{s, t} (\epsilon, B)
    \dequal \bar{K}_{0, t-s} (\epsilon, B).
    \end{align*}
  \end{enumerate}
\end{lemma}

\begin{proof}
  The claim (a) follows from Lemma
  \ref{lem:K-scaling} and
  the claim (b) follows from Lemma \ref{lem:K-subadditive}. For the
  claim (c), we observe that for every $\rho \in \mathbb{R}$ we have
  \begin{equation*}
    \bar{K}_{s, t}(\epsilon, w) = \bar{K}_{s, t}(\epsilon, w+\rho)
    = \bar{K}_{0, t-s}(\epsilon, w_{s+\cdot} + \rho)
  \end{equation*}
  In particular, we choose $\rho \assign - w_s$.
\end{proof}

\begin{lemma}[Moments of $K$]
  \label{lem:J-moment} For every $p, t, \epsilon \in (0, \infty)$ we have
    $\expect[\sup_{\rho \in \mathbb{R}} K_{0, t}(\epsilon, B+\rho)^p] < \infty$.
\end{lemma}

\begin{proof}
  For $\alpha \in (0, H)$, we set
  \begin{equation*} 
    \llbracket B \rrbracket_{C^{\alpha} ([0, t])}   \assign \sup_{0 \leq r < s \leq t} \frac{|
      B_s - B_r |}{(s - r)^{\alpha}} . 
  \end{equation*} 
By the Kolmogorov continuity theorem, we have
  \begin{equation}\label{eq:B-holder-moment}
    \mathbb{E} [\llbracket B \rrbracket _{C^{\alpha} ([0, t])}^p] < \infty.
  \end{equation}
  We set \rafal{In subscripts you use somtimes $\llbracket B \rrbracket_{C^{\alpha} ([0,
     t])}$ and sometimes $\llbracket B \rrbracket_{C^{\alpha} ([0,
     1])}$. I think we should decide for $t$}
  \begin{equation*}
     \delta \assign  \epsilon^{\frac{1}{\alpha}} (1 + \llbracket B \rrbracket_{C^{\alpha} ([0,
     t])})^{-\frac{1}{\alpha}} .
  \end{equation*} 
  To lead to a contradiction, suppose that there exist integers $k$ and $n$ such that
  \begin{equation*}
    k \delta \leq T_n (\epsilon, B + \rho) <
  T_{n + 1} (\epsilon, B + \rho) \leq (k + 1) \delta \quad \text{with } T_{n+1}(\epsilon, B+\rho) \leq t.
  \end{equation*}
  Then,
  \begin{align*}
   \epsilon = | B_{T_{n + 1} (\epsilon, B + \rho)} - B_{T_n (\epsilon, B + \rho)} | &\leq 
   \llbracket B \rrbracket_{C^{\alpha} ([0, t])} \delta^{\alpha} \\
     &= \epsilon \llbracket B \rrbracket_{C^{\alpha} ([0, t])}
     (1 + \llbracket B \rrbracket_{C^{\alpha} ([0, t])})^{- 1} < \epsilon,
  \end{align*}
  which is a contradiction. Thus, we must have
  \begin{equation*}
     \# \{ n \of k \delta \leq T_n (\epsilon, B + \rho) \leq (k + 1) \delta \} \leq
     1 \quad \text{for each } k 
  \end{equation*}
  and \rafal{I think that instead of $\delta^{-1}$ we should have $t/\delta$}
  \begin{equation}\label{eq:J_bound_by_Holder}
     \sup_{\rho \in \R} K_{0, t} (\epsilon, B + \rho) \leq  t \delta^{- 1} =  t \epsilon^{-\frac{1}{\alpha}} (1 + \llbracket B \rrbracket_{C^{\alpha} ([0, t])})^{\frac{1}{\alpha}} ,
  \end{equation} 
  which is $L^p(\P)$-integrable by \eqref{eq:B-holder-moment}.
\end{proof}

In view of Lemma~\ref{lem:K-bar-subadditive-and-stationary}, the family $(\expect[\bar{K}_{0, t}(1, B)])_{t \geq 0}$
satisfies
\begin{equation*}
  \expect[\bar{K}_{0, s+t}(1, B)] \geq \expect[\bar{K}_{0, s}(1, B)] + \expect[\bar{K}_{0, t}(1, B)].
\end{equation*}
Therefore, by Fekete's Lemma, the following limit exists in $[0, \infty]$:
\begin{equation}\label{eq:const_c_H}
  \mathfrak{c}_H \assign \lim_{t \to \infty} \frac{1}{t} \expect[\bar{K}_{0, t}(1, B)]
  = \sup_{t > 0} \frac{1}{t} \expect[\bar{K}_{0, t}(1, B)].
\end{equation}
The constant $\mathfrak{c}_H$ coincides with the one from Theorem~\ref{thm:convergence-of-K}.
The following lemma shows that $\mathfrak{c}_H$ is non-trivial.
\begin{lemma}[Non-triviality of $\mathfrak{c}_H$]
  We have $\mathfrak{c}_H \in (0, \infty)$.
\end{lemma}
\begin{proof}
  To see $\mathfrak{c}_H > 0$, we observe 
  \begin{align*}
    \mathfrak{c}_H \geq \mathbb{E}[\bar{K}_{0,1}(1, B)]
    \geq \mathbb{P}(B_1 \geq 2) > 0.
  \end{align*}
  To see $\mathfrak{c}_H < \infty$, we note by Lemma~\ref{lem:K-bar-subadditive-and-stationary} 
  that $(\bar{K}_{s,t}+1)$, $s<t$ is subadditive. Therefore, \rafal{$J$ is undefinied.}
  \begin{align*}
    \mathfrak{c}_H \leq \mathbb{E}[\bar{K}_{0,1}(1, B)] + 1 
    \leq \mathbb{E}[\sup_{\rho} K_{0, 1}(1, B-\rho)] + 1,
  \end{align*}
  which is finite by Lemma~\ref{lem:J-moment}.
\end{proof}
\begin{remark}
  By the subadditivity, we have 
  \begin{align*}
    \frac{\mathbb{E}[\bar{K}_{0, t}(1, B)]}{t} \leq \mathfrak{c}_H 
    \leq \frac{\mathbb{E}[\bar{K}_{0, t}(1, B)] + 1}{t}.
  \end{align*}
  In particular, 
  \begin{align}\label{eq:c_H_convergence_bound}
    \abs[\Big]{\mathfrak{c}_H - \frac{\mathbb{E}[\bar{K}_{0, t}(1, B)]}{t}} \leq t^{-1}.
  \end{align}
\end{remark}

\subsection{The shifted stochastic sewing lemma}\label{subsec:shifted_ssl}
A key ingredient to the proof of Theorem~\ref{thm:convergence-of-K} is the following lemma.
\begin{lemma}[Shifted stochastic sewing lemma, {\cite[Theorem~1.1]{matsuda22}}]\label{lem:shifted_ssl}
  Let $(\mathcal{F}_t)_{t \in [0, T]}$ be a filtration and 
  let $(A_{s,t})_{0 \leq s < t \leq T}$ 
  be a two-parameter stochastic process such that $A_{s, t}$ is $\mathcal{F}_t$-measurable. 
  Suppose that for some $p \in [2, \infty)$ we have $A_{s, t} \in L^p(\mathbb{P})$ 
  for every $s <t$. 
  Moreover, suppose that for $v < s < u < t$ and $M \in (0, \infty)$ we have the estimates
  \begin{align}
    \norm{A_{s, t} - A_{s, u} - A_{u, t}}_{L^p(\mathbb{P})} 
    &\leq  \Gamma_1 (t-s)^{\beta_1}, \notag \\
    \norm{\mathbb{E}[A_{s, t} - A_{s, u} - A_{u, t} \vert \mathcal{F}_v]}_{L^p(\mathbb{P})}
    &\leq \Gamma_2 (s-v)^{-\alpha} (t-s)^{\beta_2} \label{eq:ssl_conditional},
  \end{align}
  where $t-s \leq M^{-1} (s-v)$ is assumed in \eqref{eq:ssl_conditional}, with $\alpha, \beta_1, \beta_2$ satisfying 
  \begin{align*}
    \min \{2 \beta_1, 2 (\beta_2 - \alpha), \beta_2 \} > 1.
  \end{align*}
  Then there exists a unique $(\mathcal{F}_t)$-adapted stochastic process $(\mathcal{A}_t)_{t \in [0, T]}$ 
  with $\mathcal{A}_0 = 0$
  such that 
  \begin{align*}
    \norm{\mathcal{A}_{s, t} - A_{s, t}}_{L^p(\mathbb{P})} &\lesssim_{p, \alpha, \beta_1, \beta_2, M} \Gamma_1 (t-s)^{\beta_1}
    + \Gamma_2 (t-s)^{\beta_2 - \alpha}, \\
    \norm{\mathbb{E}[\mathcal{A}_{s, t} - A_{s, t} \vert \mathcal{F}_v]}_{L^p(\mathbb{P})} 
    &\lesssim_{p, \alpha, \beta_1, \beta_2, M} 
    \Gamma_2 (s-v)^{-\alpha} (t-s)^{\beta_2} 
  \end{align*}
  for every $v < s < t$, where $t-s \leq M^{-1}(s-v)$ is assumed in the second estimate.
  Furthermore, we can find a $\delta > 0$, 
  depending only on $\alpha, \beta_1, \beta_2$, such that 
  \begin{align*}
    \norm[\Big]{\mathcal{A}_T - \sum_{[s, t] \in \pi} A_{s, t}}_{L^p(\mathbb{P})}
    \lesssim_{p, \alpha, \beta_1, \beta_2, T} (\Gamma_1 + \Gamma_2) \abs{\pi}^{\delta}
  \end{align*}
  for every partition $\pi$ of $[0, T]$.
\end{lemma}
\begin{remark}
  The stochastic sewing lemma was first obtained in the seminal work \cite{le20}, 
  and the first shifted version, where $\alpha = 0$, was obtained by \cite{Gerencser:2022aa}, 
  which was extended by \cite{matsuda22} to handle the case $\alpha > 0$.
\end{remark}
\begin{remark}\label{rem:how_to_use_ssl}
  For our problems, the family $(A_{s, t})_{0 \leq s < t \leq T}$ satisfies 
  \begin{align*}
    \norm{A_{s,t}}_{L^p(\mathbb{P})} \leq \Gamma_1 (t-s)^{\beta_1}, 
    \quad \norm{\mathbb{E}[A_{s,t} \vert \mathcal{F}_v]}_{L^p(\mathbb{P})} \leq \Gamma_2 (s-v)^{-\alpha} (t-s)^{\beta_2}.
  \end{align*}
  In this case, by uniqueness of the limit $\mathcal{A}$, we must have $\mathcal{A} \equiv 0$. 
  In particular, with some $\delta = \delta (\alpha, \beta_1, \beta_2)$ we have 
  \begin{align*}
    \norm[\Big]{\sum_{[s,t] \in \pi} A_{s, t}}_{L^p(\mathbb{P})} 
    \lesssim_{T, p, \alpha, \beta_1, \beta_2} (\Gamma_1 + \Gamma_2) \abs{\pi}^{\delta}.
  \end{align*}
\end{remark}
Let us explain the strength of Lemma~\ref{lem:shifted_ssl}. 
The key is the estimate \eqref{eq:ssl_conditional}, 
which allows us to bring a \emph{weak} estimate into a strong estimate.
To illustrate an example, recall the Mandelbrot--Van Ness representation \eqref{eq:mandelbrot}, and we write $(\mathcal{F}_t)_{t \in \mathbb{R}}$ 
for the filtration generated by the Brownian motion $W$ in \eqref{eq:mandelbrot}.
Suppose that we want to estimate $\mathbb{E}[F(B) \vert \mathcal{F}_v]$. 
We then have 
\begin{align*}
  \mathbb{E}[F(B) \vert \mathcal{F}_v]
  = \mathbb{E}[F(y + \tilde{B})] \vert_{y = \mathbb{E}[B \vert \mathcal{F}_v]}, 
  \quad \tilde{B} \defby \int_v^{\cdot} \kernel(\cdot, s) \mathrm{d} W_s.
\end{align*}
Then the problem reduces to the estimate of a Gaussian expectation (weak estimate), which allows us to leverage regularity of the Gaussian density. 

Obviously, weak estimates provide better bounds than strong estimates do. As this point of view is the key to our arguments, 
let us elaborate on one simple example.
Let $X$ be a one-dimensional standard Gaussian random variable, and let $F \from 
\mathbb{R} \to \mathbb{R}$.  
If we want to estimate 
\begin{align*}
  \mathbb{E}[\abs{F(X) - F(X + a)}]
\end{align*}
for a small $a \in \mathbb{R}$, then we need to assume some regularity of $F$. 
On the other hand, if we want to estimate 
\begin{align*}
  \abs{\mathbb{E}[F(X)] - \mathbb{E}[F(X+a)]},
\end{align*}
then by the Gaussian change of variable we get 
\begin{align*}
  \mathbb{E}[F(X+a)] = e^{-\frac{a^2}{2}} \mathbb{E}[e^{a X} F(X)].
\end{align*}
Hence, only assuming $F$ is bounded, using the Cauchy--Schwarz inequality, 
\toyomu{a bit more detail, CS}
we get 
\begin{align*}
  \abs{\mathbb{E}[F(X)] - \mathbb{E}[F(X+a)]}
  \lesssim \norm{F}_{L^{\infty}} \abs{a}, \quad \forall a \text{ with } \abs{a} \leq 1.
\end{align*}
Very roughly speaking, we will go through such lines of reasoning to prove Theorem~\ref{thm:convergence-of-K}. The Gaussian change of variable will be replaced by 
Girsanov's theorem in the spirit of Picard~\cite{picard08}, see Lemma~\ref{lem:K-rho}.

\subsection{Convergence of variation}\label{subsec:var_conv}
As already suggested, to prove Theorem~\ref{thm:convergence-of-K}, we 
will apply the shifted stochastic sewing, Lemma~\ref{lem:shifted_ssl}. 
We denote by $(\mathcal{F}_t)_{t \in \mathbb{R}}$ the filtration generated by the Brownian motion $W$ appearing in 
the Mandelbrot--Van Ness representation \eqref{eq:mandelbrot}.
The following is the first observation.
\begin{lemma}[Asymptotic weak estimate of $\bar{K}$]
  \label{lem:K-bar-average}
  Let $\zeta \geq 1$ and $v < s < t$. 
  We set $\epsilon \defby (\frac{t-s}{\zeta})^H$. 
  Then, if $\frac{t-s}{s-v}$ is sufficiently small and $p>1$, we have 
  \begin{align*}
    \norm{\mathbb{E}[\bar{K}_{s,t}(\epsilon, B) \vert \mathcal{F}_v] 
    - \mathbb{E}[\bar{K}_{0, \zeta}(1, B)]}_{L^p(\mathbb{P})}
    \lesssim_{p, \zeta} \Big( \frac{t-s}{s-v} \Big)^{1-H}.
  \end{align*}
\end{lemma}
Lemma~\ref{lem:K-bar-average} is an easy consequence of the following result.
\begin{lemma}[Asymptotic independence, {\cite[Lemma~A.1]{picard08}}]
  Let $0 \leq v < s < t$. 
  Let $F$ and $G$ be measurable with respect to
  $\mathcal{F}_v$ and 
  \begin{equation}\label{eq:sigma_alg_by_diff}
      \sigma (B_{t'} - B_{s'} \of
     s \leq s' < t' \leq t)
  \end{equation}
  respectively, and suppose that $F, G \in L^p (\mathbb{P})$ with $p \in (1, \infty)$. If
  $(t- s) (s- v)^{- 1}$ is sufficiently small, then we have
  \begin{equation*}
     | \mathbb{E} [F G] -\mathbb{E} [F] \mathbb{E} [G] | \lesssim_{p}
     \Big( \frac{t-s}{s-v} \Big)^{1 - H} \| F \|_{L^p
     (\mathbb{P})} \| G \|_{L^p (\mathbb{P})} . 
  \end{equation*}
\end{lemma}
\begin{remark}\label{rem:how_to_use_asymp_indep}
  Consequently, we have the following estimate.
  Let $p \in [2, \infty)$, and we set $p' \defby p/(p-1)$.
  If $G$ is measurable with respect to the $\sigma$-algebra \eqref{eq:sigma_alg_by_diff}, then 
  for any $\mathcal{F}_v$-measurable $F$ we have 
  \begin{align*}
    \abs{\mathbb{E}[(G - \mathbb{E}[G]) F]} \lesssim_p 
    \Big( \frac{t-s}{s-v} \Big)^{1-H} \norm{F}_{L^{p'}(\mathbb{P})} \norm{G}_{L^{p'}(\mathbb{P})},
  \end{align*}
  provided that $\frac{t-s}{s-v}$ is sufficiently small. 
  Since $L^{p'}(\mathbb{P})$ is the dual of $L^p(\mathbb{P})$ and $p' \leq p$, we have
  \begin{align}\label{eq:asymptotic_indep}
    \norm{\mathbb{E}[G \vert \mathcal{F}_{v}] - \mathbb{E}[G]}_{L^p(\mathbb{P})}
    \lesssim_p \Big( \frac{t-s}{s-v} \Big)^{1 - H} \norm{G}_{L^{p}(\mathbb{P})},
  \end{align}
  where we do not need to assume that $\frac{t-s}{s-v}$ is small.
\end{remark}
\begin{proof}[Proof of Lemma~\ref{lem:K-bar-average}]
  By Lemma~\ref{lem:K-bar-subadditive-and-stationary}-(iii) (or by Remark~\ref{rem:truncated_var}), \rafal{rather by Remark \ref{rem:truncated_var}}
  the random variable $\bar{K}_{s, t}(\epsilon, B)$ is 
  measurable with respect to 
  $\sigma(B_r - B_s : s \leq r \leq t)$. 
  The estimate \eqref{eq:asymptotic_indep} implies
  \begin{align*}
    \norm{\mathbb{E}[\bar{K}_{s, t}(\epsilon, B) \vert \mathcal{F}_v] - 
    \mathbb{E}[\bar{K}_{s, t}(\epsilon, B)]}_{L^p(\mathbb{P})}
    \lesssim \Big( \frac{t-s}{s-v} \Big)^{1-H} 
    \norm{\bar{K}_{s, t}(\epsilon, B)}_{L^p(\mathbb{P})}.
  \end{align*}
  By the stationarity and the scaling (Lemma~\ref{lem:K-bar-subadditive-and-stationary}), 
  \begin{align*}
    \bar{K}_{s, t}(\epsilon, B) \dequal \bar{K}_{0, \zeta}(1, B)
  \end{align*}
  and the moment 
   $\norm{\bar{K}_{0, \zeta}(1, B)}_{L^p(\mathbb{P})}$ is bounded by Lemma~\ref{lem:J-moment}. 
  The claim now follows.
\end{proof}
We recall the Mandelbrot--Van Ness representation (Definition~\ref{def:mandelbrot_van_ness}).
The next lemma is a consequence of Girsanov's theorem.
\begin{lemma}[Weak estimate on $K$]
  \label{lem:K-rho}Let $v < s < t$, $\epsilon \in (0, 1)$, 
  $\rho, \rho' \in [- \epsilon / 2, \epsilon / 2]$ and $y
  : [v, t] \to \mathbb{R}$ be a deterministic continuous path. We set
  \begin{equation}\label{eq:tilde-B}
   \tilde{B}^v_r \assign  \int_v^r (r - u)^{H - 1 / 2} \mathrm{d} W_u, \quad
     v \leq r \leq t
  \end{equation}
  \toyomu{consistent with $\mathcal{B}$ or $\tilde{B}$}
  and
  \begin{equation*}
     b_H \assign \frac{1}{4(1-H)} \Big( \frac{1}{\Gamma (H + 1 / 2)
     \Gamma (3 / 2 - H)} \Big)^2, 
  \end{equation*}
  where $\Gamma$ is the usual Gamma function
  \begin{align}\label{eq:gamma_fcn}
     \Gamma(z) \defby \int_0^{\infty} t^{z-1} e^{-t} \mathrm{d} t.
  \end{align}
  We then have the bound
  \begin{multline*}
       | \mathbb{E} [K_{s, t} (\epsilon, \tilde{B}^v + y + \rho)] -\mathbb{E} [K_{s, t} (\epsilon,
       \tilde{B}^v + y + \rho' )] | \\
       \lesssim  e^{b_H | \rho - \rho' |^2 (s - v)^{- 2} (t - v)^{2 - 2
       H}} 
       \times \mathbb{E} [K_{s, t} (\epsilon, \tilde{B}^v + y + \rho)^2]^{\frac{1}{2}} |
       \rho - \rho' | (s - v)^{- 1} (t - v)^{1 - H} .
  \end{multline*}
\end{lemma}

\begin{proof}
  The proof is inspired by {\cite[Theorem A.1]{picard08}}. Let $\delta \assign
  \rho' - \rho$ and
  \begin{equation*} 
    h_r \assign
      \begin{cases}
       (s - v)^{- 1} (r - v) \delta &  \text{if } v \leq r \leq s,\\
       \delta &   \text{if } s \leq r.
      \end{cases}
  \end{equation*}   
  Note that the functions $r \mapsto \tilde{B}^v_r + y_r + \rho'$ and $r \mapsto \tilde{B}^v_r + y_r + h_r + \rho$
  are equal on the interval $[s, t]$.
  Thus,
  \begin{equation*}
     K_{s, t} (\epsilon, \tilde{B}^v + y + \rho') = K_{s, t} (\epsilon,
     \tilde{B}^v + y + h + \rho) . 
  \end{equation*}
  We claim
  \begin{equation*}
     h_r =  \int_v^r (r - u)^{H - 1 / 2} \mathrm{d} g_u, 
  \end{equation*}
  where for $r > v$,
  \begin{equation*}
     g_r \assign \frac{\delta\{ (r - v)^{3 / 2 - H} - (r - r \wedge s)^{3 /
     2 - H} \}}{\Gamma (H + 1 / 2) \Gamma (3 / 2 - H) (3 / 2
     - H) (s-v)}   . 
  \end{equation*}
  Indeed,
  \begin{align*}
   \dot{g}_r &\assign \frac{\dd g_r}{\dd r} \\
   &\phantom{\vcentcolon} = \frac{1}{\Gamma (H + 1 / 2) \Gamma (3 / 2 - H)}
     \frac{\delta}{s - v} \{ (r - v)^{1 / 2 - H} - (r - s)^{1 / 2 - H}
     \indic_{\{ r > s \}} \} 
  \end{align*}
  and
  \begin{align*}
    \int_v^r (r - u)^{H - 1 / 2} (u - v)^{1 / 2 - H} \mathrm{d} u & =  \int_0^{r
    - v} (r - v - u)^{H - 1 / 2} u^{1 / 2 - H} \mathrm{d} u\\
    & =  (r - v) \int_0^1 (1 - u)^{H - 1 / 2} u^{1 / 2 - H} \mathrm{d} u\\
    & =  \Gamma (H + 1 / 2) \Gamma (3 / 2 - H) (r - v),
  \end{align*}
  where in the last line the relation between the Beta function and the Gamma function is used.
  Therefore,
  \begin{equation*}
      \int_v^r (r - u)^{H - 1 / 2} \mathrm{d} g_u = \frac{\delta}{s - v} \{ (r
     - v) - (r - s) \indic_{\{ r > s \}} \} = h_r . 
  \end{equation*}

  If we set
  \begin{equation*}
     F (w) \assign K_{s, t} \Big( \epsilon,  \int_v^{\cdot} (\cdot -
     u)^{H - 1 / 2} \mathrm{d} w_u + y + \rho  \Big), 
  \end{equation*}
  then $K_{s, t} (\epsilon, \tilde{B}^v + y + \rho') = F (W + g)$ and by Girsanov's theorem 
  (or the Cameron-Martin theorem)
  \begin{equation*}
     \mathbb{E} [F (W + g)] =\mathbb{E} \Big[ e^{\int_v^t \dot{g_r} \mathrm{d}
     W_r - \frac{1}{2} \int_v^t | \dot{g}_r |^2 \mathrm{d} r} F (W) \Big] . 
  \end{equation*}
  Thus,
  \begin{multline*}
   \mathbb{E} [K_{s, t} (\epsilon, \tilde{B}^v + y + \rho')] -\mathbb{E} [K_{s, t} (\epsilon, \tilde{B}^v + y
     + \rho )]  \\
     =\mathbb{E} \Big[ \Big\{ e^{\int_v^t \dot{g_r}
     \mathrm{d} W_r - \frac{1}{2} \int_v^t | \dot{g}_r |^2 \mathrm{d} r} - 1 \Big\}
     K_{s, t} (\epsilon,  \tilde{B}^v + y + \rho) \Big] .
  \end{multline*}
  By the Cauchy--Schwarz inequality, it is bounded by
  \begin{equation*}
     \mathbb{E} \Big[ \Big( e^{\int_v^t \dot{g_r} \mathrm{d} W_r - \frac{1}{2}
     \int_v^t | \dot{g}_r |^2 \mathrm{d} r} - 1 \Big)^2 \Big]^{1 / 2}
     \mathbb{E} [K_{s, t} (\epsilon, \tilde{B}^v + y + \rho)^2]^{\frac{1}{2}} . 
  \end{equation*}
  Since $\int_v^t \dot{g}_r \mathrm{d} W_r$ is centered Gaussian with variance 
  \begin{align*}
    \int_v^t \abs{\dot{g}_r}^2 \mathrm{d} r \leq 2 b_H \delta^2 (s-v)^{-2} (t-v)^{2- 2H},
  \end{align*}
 \rafal{in the formula above I changed $a_H$ into $b_H$} we obtain
  \begin{align*}
    \mathbb{E} \Big[ \Big( e^{\int_v^t \dot{g_r} \mathrm{d} W_r - \frac{1}{2}
    \int_v^t | \dot{g}_r |^2 \mathrm{d} r} - 1 \Big)^2 \Big] 
    &= e^{\int_v^t
    | \dot{g}_r |^2 \mathrm{d} r} - 1\\
    &\leq \int_v^t | \dot{g}_r |^2 \mathrm{d} r \times e^{\int_v^t | \dot{g}_r
    |^2 \mathrm{d} r}\\
    &\lesssim e^{2 b_H | \rho - \rho' |^2 (s - v)^{- 2} (t - v)^{2 - 2
    H}} | \rho - \rho' |^2 (s - v)^{- 2} (t - v)^{2 - 2 H},
  \end{align*}
 (in the second line we used $e^a -1 \le a \times e^a$ for $a\ge 0$) which completes the proof.
\end{proof}

\begin{proof}[Proof of Theorem~\ref{thm:convergence-of-K}]
  In view of the scaling, we may suppose that $T = 1$. 
  The proof resembles that of the subadditive ergodic theorem \cite[Theorem~6.4.1]{Durrett2019}.

  {\tmstrong{Step 1, lower bound.}}   
  We fix a parameter $\zeta \geq 1$, which will go to infinity at the end. 
  (The parameter $\zeta$ corresponds to the parameter $m$ in \cite[Theorem~6.4.1]{Durrett2019}.) 
  Let $\pi_{\epsilon, \zeta}$ be 
  the partition of $[0, 1]$ with identical mesh size $\zeta \epsilon^{\frac{1}{H}}$.  
  By the superadditivity (Lemma~\ref{lem:K-subadditive}), using the relation $t-s = \zeta \epsilon^{1/H}$ for $[s,t] \in \pi_{\epsilon, \zeta}$, we obtain
  \begin{align*}
    \epsilon^{\frac{1}{H}} K_{0, 1}(\epsilon, B+\rho)
    \geq \sum_{[s, t] \in \pi_{\epsilon, \zeta}} \epsilon^{\frac{1}{H}} 
      K_{s, t}(\epsilon, B + \rho) = \zeta^{-1} 
      \sum_{[s, t] \in \pi_{\epsilon, \zeta}}  A_{s, t}^1, 
  \end{align*}
  where $A^1_{s, t} \defby K_{s, t}((\frac{t-s}{\zeta})^H, B+\rho) (t-s)$.
  Furthermore, we set 
  \begin{align*}
    A^2_{s, t} \defby \bar{K}_{s, t}\Big(\Big(\frac{t-s}{\zeta}\Big)^H, B \Big) (t-s),
    \quad A^3_{s, t} \defby \mathbb{E}[\bar{K}_{0, \zeta}(1, B)] (t-s).
  \end{align*} 
  We see that $A_{s, t} \defby A^1_{s, t} - A^3_{s, t}$ satisfies the condition 
  of Lemma~\ref{lem:shifted_ssl}. Indeed, by scaling we have 
  \begin{align*}
    \norm{K_{s, t}(\epsilon, B + \rho)}_{L^p(\mathbb{P})}
    +\norm{\bar{K}_{s, t}(\epsilon, B + \rho)}_{L^p(\mathbb{P})} 
    \lesssim_{p, \zeta} 1
  \end{align*}
  and hence 
  \begin{align*}
    \norm{A_{s, t}}_{L^p(\mathbb{P})} \leq 
    \norm{A^1_{s, t}}_{L^p(\mathbb{P})} +
     \norm{A^3_{s, t}}_{L^p(\mathbb{P})}
    \lesssim_{p, \zeta} (t-s).
  \end{align*}

  To estimate the conditional expectation, let $(t-s)/(s-v)$ be so small that the claim of Lemma~\ref{lem:K-bar-average} holds.
  Since 
  \begin{align*}
    K_{s, t}(\epsilon, B + \rho) - \bar{K}_{s, t}(\epsilon, B) 
    = \epsilon^{-1} \int_{-\epsilon/2}^{\epsilon/2} 
    \{ K_{s, t}(\epsilon, B + \rho) - K_{s, t}(\epsilon, B + \rho + \rho')  \}\mathrm{d} \rho',
  \end{align*}
  by Lemma~\ref{lem:K-rho}, using $|\rho - \rho'| \le \epsilon$, $(s-v)^{-1}(t-v)^{1-H} \le \sqrt{2}(s-v)^{-H}$ (this holds for $(t-s)/(s-v)$ sufficiently small) and $\epsilon (s-v)^{-H} = ((t-s)/(s-v))^H / \zeta^H$ \rafal{I added few explanations here} we have 
  \begin{align*}
    \MoveEqLeft[3]
    |\mathbb{E}[K_{s, t}(\epsilon, B + \rho) - \bar{K}_{s, t}(\epsilon, B) \vert \mathcal{F}_v]| \\
    &\lesssim \frac{\mathbb{E}[K_{s,t}(\epsilon, B + \rho)^2 \vert \mathcal{F}_v]^{\frac{1}{2}}}{\epsilon} 
    \int_{-\epsilon/ 2}^{\epsilon/2} e^{2 b_H \epsilon^2 (s-v)^{-2H}} \epsilon (s-v)^{-H} \mathrm{d} \rho' \\
    &\lesssim_{\zeta} \mathbb{E}[K_{s,t}(\epsilon, B + \rho)^2 \vert \mathcal{F}_v]^{\frac{1}{2}} \Big( \frac{t-s}{s-v} \Big)^H,
  \end{align*}
  which readily yields
  \begin{align*}
    \norm{\mathbb{E}[A^1_{s,t} - A^2_{s, t} \vert \mathcal{F}_v]}_{L^p(\mathbb{P})} 
    \lesssim_{p, \zeta} \Big( \frac{t-s}{s-v} \Big)^H (t-s).
  \end{align*}
  By Lemma~\ref{lem:K-bar-average}, 
  \begin{align*}
    \norm{\mathbb{E}[A^2_{s, t} - A^3_{s, t} \vert \mathcal{F}_v] }_{L^p(\mathbb{P})}     
    \lesssim_{p, \zeta} \Big( \frac{t-s}{s-v} \Big)^{1 - H} (t-s).
  \end{align*}
  Therefore, 
  \begin{align*}
    \norm{\mathbb{E}[A_{s, t} \vert \mathcal{F}_v]}_{L^p(\mathbb{P})}
    \lesssim_{p, \zeta} \Big( \frac{t-s}{s-v} \Big)^{\min\{H, 1-H\}} (t-s),
  \end{align*}
  and we indeed see that $(A_{s, t})_{s < t}$ satisfies the conditions of 
  Lemma~\ref{lem:shifted_ssl}.

  Consequently, recalling Remark \ref{rem:how_to_use_ssl}, we obtain 
  \begin{align}\label{eq:K_quantitative_lbd}
    \epsilon^{\frac{1}{H}} K_{0, 1}(\epsilon, B) 
    \geq \frac{\mathbb{E}[\bar{K}_{0, \zeta}(1, B)]}{\zeta} - R_{\epsilon, \zeta},
  \end{align}
  where 
  \begin{align*}
    \norm{R_{\epsilon, \zeta}}_{L^p(\mathbb{P})} 
    \lesssim_{p, \zeta} \epsilon^{\delta} 
  \end{align*}
  for some $\delta$ depending only on $H$. By the Borel--Cantelli lemma, 
  if $\epsilon_n = O(n^{-\eta})$ for some $\eta > 0$, then 
  $R_{\epsilon_n, \zeta} \to 0$ a.s. This implies 
  \begin{align*}
    \liminf_{n \to \infty} \epsilon_n^{\frac{1}{H}} K_{0, 1}(\epsilon_n, B + \rho) \geq 
    \frac{\mathbb{E}[\bar{K}_{0, \zeta}(1, B)]}{\zeta}
    \quad \text{a.s.}
  \end{align*}
  Since $\zeta$ is an arbitrary real no smaller than $1$, the lower bound is obtained.  

  {\bfseries Step 2, upper bound.} Since $(K_{s, t}(\epsilon, B + \rho) + 1)_{s < t}$ 
  is subadditive, 
  we obtain 
  \begin{align}\label{eq:K_quantitative_ubd}
    \epsilon^{\frac{1}{H}} K_{0, 1}(\epsilon, B) 
    \leq \frac{\mathbb{E}[\bar{K}_{0, \zeta}(1, B)]}{\zeta} + \frac{1}{\zeta} + R_{\epsilon, \zeta},
  \end{align}
  and we similarly obtain the upper bound.
\end{proof}

\section{Local time via level crossings}\label{sec:local_time}

In this section, we are interested in level crossings at a specific level. 
Our goal of this section is to prove Theorem~\ref{thm:lemieux_type_result}, fractional analogue of Chacon et al. \cite{chacon1981}. 
The key is to obtain a more quantitative version of Theorem~\ref{thm:convergence_to_local_time}, as stated just below. 
Recall the definition of $U_{s, t}(\epsilon, w)$ from \eqref{eq:def_of_U}, 
which counts the total number of upcrossings from $0$ to $\epsilon$ in the interval $[s, t]$.

\begin{theorem}[Quantitative bound on number of upcrossing $U$]
  \label{thm:local-time-level-crossing}
  Let $H  \in (0, 1/2)$, $T \in (0, \infty)$ and $a \in \mathbb{R}$.
  The constant $\mathfrak{c}_H$ is defined by \eqref{eq:const_c_H}.
  %
  Almost surely, we have the following quantitative bound:
  \begin{equation*}
     \Big| \epsilon^{\frac{1}{H} - 1} U_{0, T} (\epsilon, B - a) - \frac{\mathfrak{c}_H}{2} L_T (a) \Big|
     \leq \zeta^{-1} L_T (a) + \mathcal{R}_{\epsilon, \zeta, T, a},  
  \end{equation*}
  for all $\epsilon \in (0, \infty)$ and $\zeta \in (1, \infty)$,
  where there exists a positive $\kappa$ such that for every $p \in (0, \infty)$ we have
  \begin{equation*}
     \| \mathcal{R}_{\epsilon, \zeta, T, a} \|_{L^p (\mathbb{P})} \leq C_{p, \zeta} T^{1-H} \epsilon^{\kappa} 
  \end{equation*}
  with $C_{p, \zeta}$ independent of $\epsilon$, $T$ and $a$.
\end{theorem}
The proof of Theorem~\ref{thm:local-time-level-crossing} is somewhat 
similar to that of Theorem~\ref{thm:convergence-of-K}, especially the bounds \eqref{eq:K_quantitative_lbd} and \eqref{eq:K_quantitative_ubd}. 
Indeed, it is based on 
the super(sub)-additivity, Girsanov's theorem and shifted stochastic sewing lemma. However, a major difficulty here is that we cannot find a counterpart to $\bar{K}$. This leads to more involved technical arguments. 
Therefore, instead of directly going to the proof, in the next section we heuristically explain our strategy.

\subsection{Heuristics}\label{subsec:heuristic}
Herein we explain our heuristic strategy to prove Theorem~\ref{thm:local-time-level-crossing}. 
Let $(\mathcal{F}_t)_{t \in \mathbb{R}}$ be the filtration generated by $W$ in the Mandelbrot--Van Ness representation \eqref{eq:mandelbrot}.
We set 
\begin{align*}
  A_{s, t} \defby U_{s, t}((t-s)^H, B - a) (t-s)^{1 - H}.
\end{align*}
In view of Lemma~\ref{lem:shifted_ssl}, our goal is to show 
\begin{align}\label{eq:A_approx_L}
  \mathbb{E}[A_{s, t} \vert \mathcal{F}_v] \approx \frac{\mathfrak{c}_H}{2} \mathbb{E}[L_{s, t}(a) \vert \mathcal{F}_v].
\end{align}
Indeed, once the estimate \eqref{eq:A_approx_L} is proven, the rest of the argument is similar 
to the proof of Theorem~\ref{thm:convergence-of-K}.

We thus explain heuristically how to prove \eqref{eq:A_approx_L}. 
For simplicity, we set $a = 0$, and we write $\epsilon \defby (t-s)^H$. 
(Strictly speaking, we actually introduce another parameter $\zeta$ going to infinity 
and set $\epsilon \defby (\frac{t-s}{\zeta})^H$, 
but for simplicity here we set $\zeta = 1$.)
Let us introduce another parameter $u \in (v, s)$ 
(in mind $t-s \ll s-u \ll u-v$), and, recalling the Mandelbrot--Van Ness representation from Definition~\ref{def:mandelbrot_van_ness}, for $r \in [s, t]$ we decompose 

\begin{align*}
   B_r  &= \int_{- \infty}^{v} \kernel (r, \theta) \mathrm{d}
   W_{\theta} + \int_{v}^{u}
   \kernel (r, \theta) \mathrm{d} W_{\theta} + \int_{u}^r \kernel (r,
   \theta) \mathrm{d} W_{\theta} \\
   &\backassign X_r + Y_r + Z_r. 
\end{align*}
In the interval $[s, t]$ the smooth
processes $X$ and $Y$ do not change much compared to $Z$. Therefore, we
can freeze time of $X$ and $Y$ (Lemma~\ref{lem:fix-X-and-Y}):
\[ \mathbb{E}[U_{s, t} (\epsilon, B) \vert \mathcal{F}_v] 
\approx \mathbb{E}[U_{s, t} (\epsilon, X_s + Y_s + Z) \vert \mathcal{F}_v] . \]
But we see
\begin{equation*}
 \mathbb{E} [ U_{s, t} (\epsilon, X_{s} + Y_{s} + Z) | \mathcal{F}_{v}   ] 
   =  \mathbb{E} [ U_{s, t} (\epsilon, x + Y_s + Z) ] |_{x = X_s}, 
\end{equation*}
and the Gaussian change of variable to $Y$ yields
\begin{equation*}
 \mathbb{E} [ U_{s, t} (\epsilon, x + Y_s + Z) ] 
   = e^{- \frac{1}{2} ( \frac{x}{\sigma_Y}
   )^2} \mathbb{E} \Big[ e^{\frac{x Y_s}{\sigma_Y^2}}
   U_{s, t} (\epsilon, Y_s + Z)
   \Big], 
\end{equation*}
where $\sigma_Y$ is the variance of $Y_s$ (Lemma~\ref{lem:Y_change_of_variables}).

For $U_{s, t} (\epsilon, Y_s + Z)$
to be positive,  $Y_s$ must be
around $0$ with high probability. (In other words, if $Y_s$ is far away from $0$, the process $Z$
must move quite a lot, which is costly.) Therefore (Lemma~\ref{lem:exp-U}),
\begin{align*}
  \mathbb{E} \Big[ e^{\frac{x Y_s}{\sigma_Y^2}} U_{s, t} (\epsilon, Y_s + Z) \Big]  \approx 
  \mathbb{E} [ U_{s, t} (\epsilon, Y_s + Z) ]
   \approx  \mathbb{E} \Big[ U_{s, t} (\epsilon,
  Y + Z) \Big] .
\end{align*}
As $v \ll u \ll s$, we have $\sigma_Y \approx \sigma_{Y+Z}$, with $\sigma_{Y+Z}$ being the variance of $Y+Z$ (Lemma~\ref{lem:sigma_Y_vs_sigma_Y_Z}). 
In the end, we have (Lemma~\ref{lem:U_conditioning}) \rafal{I think that instead of $U_{s, t} (1, Y+Z)$ on the rhs of the equation below we should have $U_{s, t} (\epsilon, Y+Z)$ there}
\[ \mathbb{E}[U_{s,t}(\epsilon, B) \vert \mathcal{F}_v] 
\approx \mathbb{E} \Big[ U_{s, t} (\epsilon, Y+Z) \Big]  e^{- \frac{1}{2} (\frac{X_s}{\sigma_{Y+Z}})^2} . \]

It is well-known that the local time is heuristically represented as integral of Dirac's delta function along $B$ (see Lemma~\ref{lem:L-germ}).
We then observe (Lemma~\ref{lem:L-germ}) \rafal{Maybe it is in place to here to recall a relationship between the local time and the integral below. TM: sentence added.}
\begin{align*}
  \mathbb{E} \Big[  \int_{s}^{t} \delta_0 (B_r)
  \mathrm{d} r \Big| \mathcal{F}_{v} \Big] 
  & \approx \int_s^t \mathbb{E}[\delta_0(B_s) \vert \mathcal{F}_v] \mathrm{d} r\\ 
  &=  \frac{1}{\sqrt{2 \pi} \sigma_{Y+Z}}  e^{- \frac{X_s^2}{2 \sigma_{Y+Z}^2}} (t-s).
\end{align*}
It is not obvious, but in Lemma~\ref{lem:expect_U_convergence} we prove 
\begin{align*}
  \sqrt{2 \pi} \sigma_{Y+Z} (t-s)^{-H} \mathbb{E}[U_{s, t}(\epsilon, Y+Z)] 
  \approx \frac{\mathfrak{c}_H}{2}.
\end{align*}
Now we see \eqref{eq:A_approx_L}.
With this heuristic argument in mind, we move to a rigorous proof in the next section.
\subsection{Convergence to local time}\label{subsec:local_time_key}
\subsubsection{Estimates on level crossings}\label{subsubsec:level_crossings}
The following process will appear in our argument.
\begin{definition}\label{def:riemann_liouville}
  The kernel $\mathcal{K}$ is defined by \eqref{eq:kernel_def}.
  We denote by $\tilde{B} = \tilde{B}^H$ the \emph{Riemann-Liouville process}
  \[ \tilde{B}_t \assign \int_0^t \kernel(t, r) \mathrm{d} W_r . \]
  In view of the Mandelbrot--Van Ness representation (Definition~\ref{def:mandelbrot_van_ness}), we have 
  \begin{align}\label{eq:B_and_tilde_B}
    B_t = \int_{-\infty}^0 \kernel(t, r) \mathrm{d} W_r + \tilde{B}_t.
  \end{align}
\end{definition}
We begin with three elementary lemmas.
\begin{lemma}[Scaling of $U$]
  \label{lem:U-scaling}We have the following scaling property: for $\lambda > 0$,
  \[ (U_{s, t} (\epsilon, B + \rho))_{s < t, \epsilon > 0,
     \rho \in \mathbb{R}} \overset{\mathrm{d}}{=} (U_{\lambda^{1 / H} s, \lambda^{1 /
     H} t} (\lambda \epsilon, B + \lambda \rho))_{s < t,
     \epsilon > 0, \rho \in \mathbb{R}} . \]
   A similar result holds with $B$ replaced by $\tilde{B}$.
\end{lemma}

\begin{proof}
Similarly to Lemma~\ref{lem:K-scaling}, it follows from the scaling property of $B$ and $\tilde{B}$.
\end{proof}

\begin{notation}\label{not:U_bar}\purba{maybe we want this as a noration instead of definition TM: corrected}
  We set
  \[ \bar{U}_{s, t} (\epsilon, w) \assign U_{s, t} (\epsilon, w)
     +\indic_{\{ w_s \in (0, \epsilon) \}} . \]
\end{notation}
\begin{lemma}[Sub/super-additivity of $U$]
  \label{lem:U-additivity}For $s < u < t$ we have
  \[ U_{s, t} (\epsilon, w) \geq U_{s, u} (\epsilon, w) + U_{u, t}
     (\epsilon, w), \quad \bar{U}_{s, t} (\epsilon, w) \leq \bar{U}_{s,
     u} (\epsilon, w) + \bar{U}_{u, t} (\epsilon, w) . \]
\end{lemma}

\begin{proof}
  We have
  \begin{align*}
    U_{s, t} (\epsilon, w) = 
    U_{s, u} (\epsilon, w) + U_{u, t} (\epsilon, w) + 1
  \end{align*}
  if there exist $a$ and $b$ such that $s \le a < u < b \le t$, $w_a = 0, w_b =
  \epsilon$ and $w_r \in (0, \epsilon)$  for all  $r \in (a, b)$, 
  and otherwise  
  \begin{equation*}
    U_{s, t} (\epsilon, w) = 
    U_{s, u} (\epsilon, w) + U_{u, t} (\epsilon, w). \qedhere
  \end{equation*}
\end{proof}

\begin{lemma}[Moment bound on $U$]\label{lem:U_uniform_bound}
   There exists a positive constant $\beta=\beta(H)$ \rafal{maybe better to use $\beta$ instead of $\kappa$ here, since it is later used where $\kappa$ appears in quite different context} such that 
   for $a\in \mathbb{R}$, $\epsilon \in (0, \infty)$, $s < t$ and $p \in (0, \infty)$ we have
   \begin{align*}
     \norm{U_{s,t}(\epsilon, B-a)}_{L^p(\mathbb{P})} 
     \lesssim_{H, p, \epsilon} 1 + (t-s)^{\beta}.
   \end{align*}
   A similar estimate holds with $B$ replaced by $\tilde{B}$.
\end{lemma}
\begin{proof}
  Regarding $B$, the claim follows from the obvious inequality $U_{s,t}(\epsilon, B - a) 
  \leq K_{s, t}(\epsilon, B - a)$ and the estimate \eqref{eq:J_bound_by_Holder}. \rafal{Regarding $\tilde{B}$, it seems for me too fast. Maybe some more detailed argument like scaling and Kolmogorov's lemma too?}
  Regarding $\tilde{B}$, by \eqref{eq:B_and_tilde_B} we observe that 
  \begin{align*}
    \norm{\tilde{B}_{s, t}}_{L^p(\mathbb{P})} \leq \norm{B_{s, t}}_{L^p(\mathbb{P})}.  
  \end{align*}
  This yields an estimate on Hölder norm of $\tilde{B}$ in $L^m(\mathbb{P})$, with which we can proceed as in $B$.
\end{proof}

\toyomu{$\zeta$ is defined here; in particular it is independent of $\epsilon$.}
We introduce some notation that will be used throughout Section~\ref{subsubsec:level_crossings}.
As in the proof of Theorem~\ref{thm:convergence-of-K}, we fix $\zeta \geq 1$, and at the very end we let $\zeta \to \infty$. 
We fix $v < u < s <t$ with 
$t-s \ll s-u \ll u-v$ and set 
\begin{equation} \label{eps_sect_3}
\epsilon \defby \left(\frac{t-s}{\zeta} \right)^H,
\end{equation} as shown in Figure~\ref{fig:parameters}.
\begin{figure}
  \centering
   \raisebox{-0.5\height}{\includegraphics[width=0.8\textwidth]{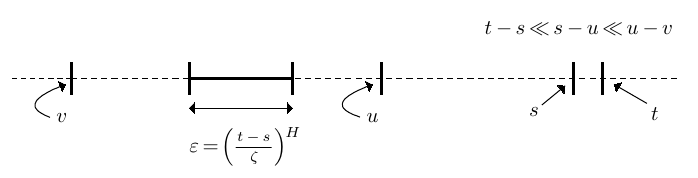}}
   \caption{Parameters for Lemma~\ref{lem:U_conditioning}}
   \label{fig:parameters}
\end{figure}
We set
\begin{equation}\label{eq:def_X_Y_Z}
     X_r \assign  \int_{- \infty}^v \kernel (r, \theta) \mathrm{d} W_{\theta} - a, \quad
   Y_r \assign \int_v^u \kernel (r, \theta) \mathrm{d} W_{\theta}, \quad Z_r \assign
   \int_u^r \kernel (r, \theta) \mathrm{d} W_{\theta}  
\end{equation}
for $r \in [s, t]$. 
Let $(\mathcal{F}_t)_{t \in \mathbb{R}}$ be the filtration generated by $W$ in the Mandelbrot--Van Ness representation \eqref{eq:mandelbrot}.
We have the identity
\begin{equation*}
     \mathbb{E} [U_{s, t} (\epsilon, B - a) | \mathcal{F}_v] =\mathbb{E}
   [U_{s, t} (\epsilon, x + Y + Z)] |_{x = X} . 
\end{equation*}
Finally, we write 
\begin{align}\label{eq:def_of_sigma_Y}
  \sigma_Y^2 &\defby \mathbb{E}[Y_s^2] = \frac{1}{2H} \{(s-v)^{2H} - (s-u)^{2H} \}, \\
  \label{eq:def_of_sigma_Y_Z}
  \sigma_{Y+Z}^2 &\defby \mathbb{E}[(Y_s+Z_s)^2] = \frac{1}{2H} (s-v)^{2H}.
\end{align}

In the spirit of the shifted stochastic sewing (Lemma~\ref{lem:shifted_ssl}), we will estimate
\begin{equation*}
     \mathbb{E} [U_{s, t} (\epsilon, B - a) | \mathcal{F}_v] . 
\end{equation*}
The most crucial ingredient to Theorem~\ref{thm:local-time-level-crossing} is the following. 
\toyomu{This is the most technical but most important lemma.}
\begin{lemma}[Asymptotic weak estimate on $U$]\label{lem:U_conditioning}
  Let $H \in (0, 1/2)$ and $p \in (1, \infty)$.
  We further let $0 \leq v < s < t \leq 1$ and $\zeta \in [1, \infty)$, 
  and set $\epsilon$ as in \eqref{eps_sect_3}.
  Let $(\mathcal{F}_t)_{t \in \mathbb{R}}$ be the filtration generated by $W$ in the Mandelbrot--Van Ness representation \eqref{eq:mandelbrot}.
  We define $X$ by \eqref{eq:def_X_Y_Z} and $\sigma_{Y+Z}$ by \eqref{eq:def_of_sigma_Y_Z}.
  For every  $\kappa \in (0, 1)$,
  if $\frac{t-s}{s-v}$ is sufficiently small, we have 
  %
  %
\rafal{Maybe instead of $\Big( \frac{t-s}{\zeta} \Big)^H$ it is better to write $\epsilon$ here? so we get:}
 \begin{equation}\label{eq:U_conditioning_asymptotic}
    \mathbb{E} [U_{s, t} (\epsilon, B - a) | \mathcal{F}_v] 
    =  \frac{\mathbb{E}[\bar{K}_{0, \zeta}(1, B)]}{2\sqrt{2 \pi} \sigma_{Y+Z}} 
    e^{-\frac{1}{2}(\frac{X_s}{\sigma_{Y+Z}})^2} \epsilon
    + R_{v,s,t},
  \end{equation}
where (we omit the $a$-dependence of $R_{v,s,t}$ from the notation and the following estimate holds uniformly in $a$):
\begin{align}\label{eq:R_estimate}
  \norm{R_{v, s, t}}_{L^p(\mathbb{P})}
  \lesssim_{p, \zeta, \kappa} \Big( \frac{t-s}{s-v} \Big)^{(2 - \kappa) H (1-H)} (t-s)^{- \kappa H}.
\end{align}
\end{lemma}
\begin{remark}\label{rem:T_to_one}
  Due to scaling, it suffices to prove Theorem~\ref{thm:local-time-level-crossing} with $T = 1$. Therefore, in Lemma~\ref{lem:U_conditioning} we assume $t \leq 1$. 
  We then do not have to keep track of dependency of constants on the final time $T$.
\end{remark}
%
%
%

The proof of Lemma~\ref{lem:U_conditioning}, to which the rest of Section~\ref{subsubsec:level_crossings} is devoted, will be built on several technical lemmas.
For the sake of the next lemma, we recall the Riemann--Liouville operator 
(e.g., \cite{picard11})
\begin{equation*}
     I_{\alpha} f (r) \assign \frac{1}{\Gamma (\alpha)} \int_s^r (r -
   \theta)^{\alpha - 1} f (\theta) \mathrm{d} \theta, \quad r > s, 
\end{equation*}
where $\Gamma$ is the Gamma function \eqref{eq:gamma_fcn} and $\alpha > 0$.
If $f$ is Lipschitz with $f_s = 0$ and $\alpha \in (- 1, 0]$, we set
\begin{equation*}
     I_{\alpha} f (r) \assign \frac{1}{\Gamma (1 + \alpha)} \int_s^r (r -
   \theta)^{\alpha} \dot{f} (\theta) \mathrm{d} \theta . 
\end{equation*}
The family $(I_{\alpha})_{\alpha > -1}$ has the semigroup property $I_{\alpha} I_{\beta} = I_{\alpha + \beta}$.

\begin{lemma}[Fix time of $X$ and $Y$]
  \label{lem:fix-X-and-Y}
  For $p\in (1, \infty)$, $\zeta  \in [1, \infty)$, $\epsilon$ given by \eqref{eps_sect_3} and $\kappa \in (1-1/p, 1)$, there exists  
  a positive constant $c$, depending on $H$ only, \purba{what is the conditions on $c$?} \toyomu{$c$ is $c_2$ in the proof.} such that if $(t-s)(s-u)^{-1}$ is sufficiently small, 
  \toyomu{how small depends on $p$ and $\kappa$}
  we have
  \begin{multline*}
   \| \mathbb{E} [U_{s, t} (\epsilon, B - a) | \mathcal{F}_u] -\mathbb{E}
     [U_{s, t} (\epsilon, X_s + Y_s + Z) | \mathcal{F}_u] \|_{L^p
     (\mathbb{P})} \\
     \lesssim_{p, \kappa, \zeta} \norm{U_{s,t}(\epsilon, B - a)}_{L^1(\mathbb{P})}^{1 - \kappa} e^{- c a^2}  \Big(\frac{t-s}{s-u} \Big)^{ 1 - H} . 
  \end{multline*} 
\purba{$\zeta$ is not in the statement. I think $\zeta= \epsilon^{-1/H}$}
\end{lemma}
\begin{proof}
  The proof is similar to \cite[Lemma~A.1]{picard08}.
  We have
  \begin{equation*}
     \mathbb{E} [U_{s, t} (\epsilon, X_s + Y_s + Z) | \mathcal{F}_u]
     =\mathbb{E} [U_{s, t} (\epsilon, x_s + y_s + Z)] |_{x = X, y = Y} 
  \end{equation*}
  and
  \begin{equation*}
     \mathbb{E} [U_{s, t} (\epsilon, x_s + y_s + Z)] =\mathbb{E} [U_{s, t}
     (\epsilon, x + y + w + Z)] 
  \end{equation*}
  where
  \begin{equation*}
     w_r \assign - (x_r + y_r - x_s - y_s), \quad r \in [s, t] . 
  \end{equation*}
  Since $X$ and $Y$ are smooth on $[s, t]$, we may suppose that the realizations $x$ and $y$ are smooth as well.
  As $w_s = 0$,
  \begin{equation*}
     w = I_1 \dot{w} = I_{H - \frac{1}{2}} I_{\frac{3}{2} - H} \dot{w} 
  \end{equation*}
  and
  \begin{equation*}
     w_r + \int_s^r \kernel (r, \theta) \mathrm{d} W_{\theta} = \int_s^r \kernel (r, \theta)
     \mathrm{d} \Big( W_{\theta} + c_1  \big( I_{\frac{3}{2} - H} \dot{w}
     \big)_{\theta} \Big) 
  \end{equation*}
  for some constant $c_1$ depending only on $H$.
  
  By Girsanov's theorem,
  \begin{multline*}
   \mathbb{E} [U_{s, t} (\epsilon, x + y + w + Z)] 
   = \mathbb{E} \Big[
     U_{s, t} (\epsilon, x + y + Z) \\
     \times \exp \Big( c_1 \int_s^t
     \frac{\mathrm{d}}{\mathrm{d} \theta} I_{\frac{3}{2} - H} \dot{w} \,\mathrm{d}
     W_{\theta} - \frac{c_1^2}{2} \int_s^t \Big| \frac{\mathrm{d}}{\mathrm{d} \theta}
     I_{\frac{3}{2} - H} \dot{w} \Big|^2 \mathrm{d} \theta \Big) \Big].
  \end{multline*}
  Therefore, if $p^{-1} + q^{-1} = 1$, by Hölder's inequality,
  \begin{multline*}
       | \mathbb{E} [U_{s, t} (\epsilon, x + y + Z)] -\mathbb{E} [U_{s, t}
       (\epsilon, x_s + y_s + Z)] |\\
       \lesssim \mathbb{E} \Big[ \Big| \exp \Big( c_1 \int_s^t
       \frac{\mathrm{d}}{\mathrm{d} \theta} I_{\frac{3}{2} - H} \dot{w} \mathrm{d}
       W_{\theta} - \frac{c^2_1}{2} \int_s^t \Big| \frac{\mathrm{d}}{\mathrm{d}
       \theta} I_{\frac{3}{2} - H} \dot{w} \Big|^2 \mathrm{d} \theta \Big) - 1
       \Big|^{q} \Big]^{\frac{1}{q}} \\
        \times \mathbb{E} [U_{s, t} (\epsilon, x + y +
       Z)^{p}]^{\frac{1}{p}}. 
  \end{multline*}
  Since the random variable
  \begin{equation*}
     \int_s^t \frac{\mathrm{d}}{\mathrm{d} \theta} I_{\frac{3}{2} - H} \dot{w} (\theta) \mathrm{d}
     W_{\theta} 
  \end{equation*}
  is Gaussian, by Lemma~\ref{lem:log_normal_moment}, 
  \begin{multline*}
   \mathbb{E} \Big[ \Big| \exp \Big( c_1 \int_s^t \frac{\mathrm{d}}{\mathrm{d}
     \theta} I_{\frac{3}{2} - H} \dot{w} \mathrm{d} W_{\theta} - \frac{c^2_1}{2}
     \int_s^t \Big| \frac{\mathrm{d}}{\mathrm{d} \theta} I_{\frac{3}{2} - H} \dot{w}
     \Big|^2 \mathrm{d} \theta \Big) - 1 \Big|^{q} \Big]^{\frac{1}{q}} \\
     \lesssim_{q} \Big(  \int_s^t \Big| \frac{\mathrm{d}}{\mathrm{d} \theta} I_{\frac{3}{2} - H} \dot{w}
     \Big|^2 \mathrm{d} \theta \Big)^{\frac{1}{2}} 
     \exp \Big( C_{q} \int_s^t \Big| \frac{\mathrm{d}}{\mathrm{d} \theta} I_{\frac{3}{2} - H} \dot{w}
     \Big|^2 \mathrm{d} \theta \Big).  
  \end{multline*}
  Hence, by setting 
  \begin{align*}
    S_{s, t} \defby \int_s^t \Big| \frac{\mathrm{d}}{\mathrm{d} \theta} I_{\frac{3}{2} - H} (\dot{X}+ \dot{Y})
     \Big|^2 \mathrm{d} \theta,
  \end{align*}
  we have \rafal{What is $\mathbb{E}_u ...$ here? I guess it is just $\mathbb{E} (... | {\cal F}_u)$ }
  %
  \begin{align*}
    \MoveEqLeft[10]
       \| \mathbb{E} [U_{s, t} (\epsilon, B - a) | \mathcal{F}_u]
       -\mathbb{E} [U_{s, t} (\epsilon, X_s + Y_s + Z) | \mathcal{F}_u]
       \|_{L^p (\mathbb{P})}\\
       &\lesssim \Big[ \mathbb{E} \left( \mathbb{E} [U_{s, t} (\epsilon, B
       - a)^{p}\vert \mathcal{F}_u]  S_{s,t}^{\frac{p}{2}} e^{p C_q S_{s,t}} \right) \Big]^{\frac{1}{p}}\\
       &\leq \| U_{s, t} (\epsilon, B - a) \|_{L^{p_1} (\mathbb{P})}
       \Big\|    S_{s,t}^{\frac{1}{2}} e^{ C_q S_{s,t}}    \Big\|_{L^{q_1} (\mathbb{P})} ,
  \end{align*}
  where $p_1^{-1} + q_1^{-1} = p^{-1}$.
  Choose $p_2$ so that $p_1^{-1} = (1-\kappa) + \kappa p_2^{-1}$ (since $\kappa > 1- p^{-1}$, this is possible by choosing $p_1$ close to $p$). By the log-convexity of $L^p$ norms, 
  \begin{align*}
    \| U_{s, t} (\epsilon, B - a) \|_{L^{p_1} (\mathbb{P})}
    \leq  \| U_{s, t} (\epsilon, B - a) \|_{L^{1} (\mathbb{P})}^{1-\kappa}
    \| U_{s, t} (\epsilon, B - a) \|_{L^{p_2} (\mathbb{P})}^{\kappa}.
  \end{align*}
  We also have (by the Cauchy--Schwarz inequality) 
  \begin{align*}
    \| U_{s, t} (\epsilon, B - a) \|_{L^{p_2} (\mathbb{P})}
    &\leq \mathbb{P}(\norm{B}_{L^{\infty}([0, 1])} \geq a)^{\frac{1}{2}} \norm{U_{s, t}(\epsilon, B-a)}_{L^{2p_2}(\mathbb{P})} \\
    &\lesssim e^{- c_2 a^2} \norm{U_{s, t}(\epsilon, B-a)}_{L^{2p_2}(\mathbb{P})}.
  \end{align*}
  The scaling property (Lemma~\ref{lem:U-scaling}) gives
  \begin{equation*}
     \| U_{s, t} (\epsilon, B - a) \|_{L^{2 p_2} (\mathbb{P})} = \| U_{s/(t-s), t/(t-s)} (\zeta^{-H}, B - (t-s)^{-H} a) \|_{L^{2 p_2} (\mathbb{P})} . 
  \end{equation*}
  By Lemma~\ref{lem:U_uniform_bound}, 
  \begin{align*}
    \| U_{s/(t-s), t/(t-s)} (\zeta^{-H}, B - (t-s)^{-H} a) \|_{L^{2 p_2} (\mathbb{P})} 
    \lesssim_{p_2, \zeta} 1.
  \end{align*}\purba{$\zeta$ depends on $H$ and $\epsilon$} \rafal{opposite, $\epsilon$ depends on $\zeta$, $t-s$ and $H$}
  It remains to see 
  \begin{align*}
    \Big\|    S_{s,t}^{\frac{1}{2}} e^{ C_q S_{s,t}}    \Big\|_{L^{q_1} (\mathbb{P})} 
    \lesssim \Big( \frac{t-s}{s-u} \Big)^{1-H}, 
    \quad \text{if } \frac{t-s}{s-u} \text{ is sufficiently small.}
  \end{align*}
  This was essentially proven in \cite[Lemma~A.1]{picard08} (our $S_{s,t}$ 
  corresponds to $L$ therein). 
  \toyomu[inline]{
    Easy way to see this: we have 
    \begin{align*}
      \norm{I_{1/2 - H}(\dot{X} + \dot{Y})}_{L^2([0, T])}
      = \norm{I_{-1/2 - H}(X + Y)}_{L^2([0, T])} 
      \sim \norm{X + Y}_{W^{H+1/2, 2}([0, T])}.
    \end{align*}
    Therefore, 
    \begin{align*}
      \norm{S_{s, t}}_{L^1(\mathbb{P})}^2 \lesssim \int_s^t \int_s^t \frac{\norm{(X+Y)_{r_1, r_2}}_{L^2(\mathbb{P})}^2}{\abs{r_1 - r_2}^{2(H+1/2) + 1}} \mathrm{d} r_1 \mathrm{d} r_2.
    \end{align*}
    Since 
    \begin{align*}
      \norm{(X+Y)_{r_1, r_2}}_{L^2(\mathbb{P})}^2 \lesssim (s-u)^{2H - 2} \abs{r_1 - r_2}^2,
    \end{align*}
    we see 
    \begin{align*}
      \norm{S_{s, t}}_{L^1(\mathbb{P})}^2 \lesssim \int_s^t \int_s^t \frac{(s-u)^{2H-2} (r_2 - r_1)^2}{\abs{r_2 - r_1}^{2H + 2}} \mathrm{d} r_1 \mathrm{d} r_2 
      \lesssim (s-u)^{2H - 2} (t-s)^{2 - 2H}.
    \end{align*}
    Now apply Fernique's theorem.
  }
  \rafal[inline]{Maybe it is worth to write this remark in the text, with the explanation what does the notation $W^{H+1/2, 2}([0, T])$ and $\norm{(X+Y)_{r_1, r_2}}$ mean ...}
\end{proof}
\begin{remark}\label{rem:bound_U_Y_s_by_U_Y}
  We note that a similar reasoning shows that for $p < p_1 < \infty$
  if $\frac{t-s}{s-u}$ is sufficiently small, we have 
  \begin{align*}
    \norm{U_{s, t}(\epsilon, Y_s + Z)}_{L^p(\mathbb{P})}
    \lesssim_{\zeta, p, p_1} \norm{U_{s, t}(\epsilon, Y + Z)}_{L^{p_1}(\mathbb{P})}.
  \end{align*}\purba{$\lesssim_{p, \epsilon,H}$. Do the bound also depend on choice of $p_1$?}
  \toyomu[inline]{Yes, it depends on $p_1$. I think it is better not to write dependency on $H$, as it is not important. I will comment in Notation section.}
\end{remark}
\begin{lemma}[Gaussian change of variable]\label{lem:Y_change_of_variables}
  Recall $\sigma_Y$ from \eqref{eq:def_of_sigma_Y}. We have the estimate
  \begin{equation*}
     \mathbb{E} [U_{s, t} (\epsilon, X_s + Y_s + Z) | \mathcal{F}_v] = e^{-
     \frac{1}{2} ( \frac{X_s}{\sigma_Y} )^2} \mathbb{E} \Big[
     \Big. e^{\frac{X_s Y_s}{\sigma_Y^2}} U_{s, t} (\epsilon, Y_s + Z)
     \Big| \mathcal{F}_v \Big] . 
  \end{equation*}
\end{lemma}

\begin{proof}
  We set $F (\eta) \assign \mathbb{E} [U_{s, t} (\epsilon, \eta + Z)]$ for $\eta \in \mathbb{R}$.
  Since $X, Y$ and $Z$ are independent and $X$ is $\mathcal F_v$-measurable, 
  %
  %
  \rafal[inline]{Maybe it is easier to omit the middle term and just write}
  \begin{equation*}
     \mathbb{E} [U_{s, t} (\epsilon, X_s + Y_s + Z) | \mathcal{F}_v]
      = \mathbb{E}[F(x_s + Y_s)] \vert_{x = X}. 
  \end{equation*}
Since $Y_s$ is Gaussian with the variance $\sigma_Y^2$ and the mean $0$, we observe 
  \begin{align*}
    \mathbb{E} [F (x_s + Y_s)] & =  \frac{1}{\sqrt{2 \pi}} \int_{\mathbb{R}}
    F (x_s + \sigma_Y \eta) e^{- \frac{\eta^2}{2}} \mathrm{d} \eta\\
    & =  \frac{1}{\sqrt{2 \pi}} \int_{\mathbb{R}} F (\sigma_Y \eta) e^{-
    \frac{1}{2} (\eta - \sigma_Y^{- 1} x_s)^2} \mathrm{d} \eta\\
    & =  e^{- \frac{1}{2} ( \frac{x_s}{\sigma_Y} )^2}
    \mathbb{E} \Big[ e^{\frac{Y_s x_s}{\sigma_Y^2}} F (Y_s) \Big] .
  \end{align*} \purba{I am strugeling to understand the fist equation with $\mathbb{E} [F (x_s + Y_s)]=$}
  \toyomu[inline]{$Y_s$ is Gaussian with variance $\sigma_Y^2$.}
  The claim thus follows.
\end{proof}

\begin{lemma}[$Y_s$ must be near $0$]
  \label{lem:exp-U}For every $p_1 \in (1, \infty)$, if $\frac{t-s}{s-u}$ is 
  sufficiently small, then
  \begin{multline*}
   \Big| \mathbb{E} \Big[ e^{\frac{x_s Y_s}{\sigma_Y^2}} U_{s, t}
     (\epsilon, Y_s + Z) \Big] -\mathbb{E} [U_{s, t} (\epsilon, Y_s +
     Z)] \Big| \\
     \lesssim_{H, p_1} \frac{| x_s | (t - u)^H}{\sigma_Y^2}
     e^{c ( \frac{| x_s | (t - u)^H}{\sigma_Y^2} )^2} \| U_{s,
     t} (\epsilon, Y + Z) \|_{L^{p_1} (\mathbb{P})} 
  \end{multline*}
  with $c$ depending only on $H$ and $p_1$.
\end{lemma}

\begin{proof}
  For $U_{s, t} (\epsilon, Y_s + Z)$ to be non-zero, we must have $\inf_{r
  \in [s, t]} | Y_s + Z_r | = 0$. Therefore,
  \begin{multline*}
   \mathbb{E} \Big[ e^{\frac{x_s Y_s}{\sigma_Y^2}} U_{s, t}
     (\epsilon, Y_s + Z) \Big] -\mathbb{E} [U_{s, t} (\epsilon, Y_s +
     Z)]  \\
     =\mathbb{E} \Big[ \Big( e^{\frac{x_s Y_s}{\sigma_Y^2}} - 1
     \Big) U_{s, t} (\epsilon, Y_s + Z) \indic_{\{ \| Z
     \|_{L^{\infty} ([s, t])} \geq | Y_s | \}} \Big] . 
  \end{multline*}
  Using the inequality
  \begin{equation*}
     | e^{\lambda} - 1 | \leq e^{| \lambda |} | \lambda |, \quad \lambda \in
     \mathbb{R}, 
  \end{equation*}
  we estimate
  \begin{multline*}
       \Big| \mathbb{E} \Big[ \Big( e^{\frac{x_s Y_s}{\sigma_Y^2}} - 1
       \Big) U_{s, t} (\epsilon, Y_s + Z) \indic_{\{ \| Z
       \|_{L^{\infty} ([s, t])} \geq | Y_s | \}} \Big] \Big|\\
       \leq \frac{| x_s |}{\sigma_Y^2} \mathbb{E} \Big[ e^{\frac{|
       x_s | \| Z \|_{L^{\infty} ([s, t])}}{\sigma_Y^2}} \| Z
       \|_{L^{\infty} ([s, t])} U_{s, t} (\epsilon, Y_s + Z) \Big],
  \end{multline*}
  and, by H{\"o}lder's inequality, the expectation on the right hand side is bounded by \rafal{I changed the order of apperance of square brackets}
  \begin{equation*}
     \left[ \mathbb{E} U_{s, t} (\epsilon, Y_s + Z)^{p_1} \right]^{\frac{1}{p_1}}
     \left[ \mathbb{E} \| Z \|_{L^{\infty} ([s, t])}^{p_2} \right]^{\frac{1}{p_2}}
     \left[ \mathbb{E} e^{\frac{p_3 | x_s | \| Z \|_{L^{\infty} ([s,
     t])}}{\sigma_Y^2}} \right]^{\frac{1}{p_3}}, 
  \end{equation*}
where $p_1, p_2, p_3 \in (1, \infty)$ satisfy
  \begin{equation*}
     \frac{1}{p_1} + \frac{1}{p_2} + \frac{1}{p_3} = 1. 
  \end{equation*}
  By Remark~\ref{rem:bound_U_Y_s_by_U_Y}, if $\frac{t-s}{s-u}$ is sufficiently small, we have
  \begin{equation*}
     \left[ \mathbb{E} U_{s, t} (\epsilon, Y_s + Z)^{p_1} \right]^{\frac{1}{p_1}}
     \lesssim_{p_1} \left[ \mathbb{E} [U_{s, t} (\epsilon, Y +
     Z)^{p_4} \right]^{\frac{1}{p_4}}, \quad p_4 \assign p_1^2 . 
  \end{equation*}
  Recalling $\tilde{B}$ from Definition~\ref{def:riemann_liouville}, the scaling 
  property yields
  \begin{equation*}
     \mathbb{E} \| Z \|_{L^{\infty} ([s, t])}^{p_2} \leq
     \mathbb{E} \| Z \|_{L^{\infty} ([u, t])}^{p_2} = (t -
     u)^{p_2 H} \mathbb{E} \| \tilde{B} \|_{L^{\infty} ([0,
     1])}^{p_2} 
  \end{equation*}
  and similarly
  \begin{equation*}
     \mathbb{E} \Big[ e^{\frac{p_3 | x_s | \| Z \|_{L^{\infty} ([s,
     t])}}{\sigma_Y^2}} \Big] \leq \mathbb{E} \Big[ e^{\frac{p_3 | x_s
     | (t - u)^H}{\sigma_Y^2} \| \tilde{B} \|_{L^{\infty} ([0, 1])}}
     \Big] . 
  \end{equation*}
  Since $\| \tilde{B} \|_{L^{\infty} ([0, 1])}$ has a Gaussian tail by
  Fernique's theorem, there exists a constant $c$ depending only on $H$ such
  that
  \begin{equation*}
     \mathbb{E} \Big[ e^{\frac{p_3 | x_s | (t - u)^H}{\sigma_Y^2} \|
     \tilde{B} \|_{L^{\infty} ([0, 1])}} \Big] \lesssim e^{c ( \frac{p_3
     | x_s | (t - u)^H}{\sigma_Y^2} )^2} . 
  \end{equation*}
  Now the claim is proved.
\end{proof}

\begin{lemma}[Sharp bound on $U$]
  \label{lem:U-moment}For every $p_1, p_2 \in (1, \infty)$ we have
  \[ \| U_{s, t} (\epsilon, Y + Z) \|_{L^{p_1} (\mathbb{P})} \lesssim_{
     \zeta, p_1, p_2} 
     \Big( \frac{t-s}{s-v} \Big)^{\frac{H}{p_1 p_2}}. \]
\end{lemma}

\begin{proof}
  By the scaling, \purba{takeing $\zeta = \epsilon^{1/H}$}
  \[ \| U_{s, t} (\epsilon, Y + Z) \|_{L^{p_1} (\mathbb{P})} = \| U_{s -
     v, t - v} (\epsilon, \tilde{B}) \|_{L^{p_1} (\mathbb{P})} = \| U_{\frac{s-v}{t-s}, \frac{t-v}{t-s}} (\zeta^{-H}, \tilde{B}) \|_{L^{p_1} (\mathbb{P})}
     . \]
  We set $k_1 \assign \frac{s-v}{t-s}$ and $k_2 \assign \frac{t-v}{t-s}$. \rafal{ I moved the remark that $k_2 - k_1 = 1$ just before the application of Fernique} We
  observe
  \begin{align*}
   \MoveEqLeft[3]
    \| U_{k_1, k_2} (\zeta^{-H}, \tilde{B}) \|_{L^{p_1} (\mathbb{P})}^{p_1}\\
    & = 
    \sum_{a \in \mathbb{Z}} \mathbb{E} \Big[ U_{k_1, k_2} (\zeta^{-H},
    \tilde{B})^{p_1} \indic_{\{ \tilde{B}_{k_1} \in (a - 1, a] \}}
    \Big]\\
    & =  \sum_{a \in \mathbb{Z}} \mathbb{E} \Big[ U_{k_1, k_2} (\zeta^{-H},
    \tilde{B})^{p_1} \indic_{\{ \tilde{B}_{k_1} \in (a - 1, a] \}}\indic_{\{ \max_{r \in [k_1, k_2]} | \tilde{B}_r -
    \tilde{B}_{k_1} | \geq | a | - 1 \}} \Big]\\
    & \leq  \sum_{a \in \mathbb{Z}} \mathbb{E} [U_{k_1, k_2} (\zeta^{-H},
    \tilde{B})^{p_1 q_2}]^{\frac{1}{q_2}} \mathbb{P} (\max_{r \in [k_1,
    k_2]} | \tilde{B}_r - \tilde{B}_{k_1} | \geq | a | -
    1)^{\frac{1}{q_2}} \\
    & \hspace{6cm} \times\mathbb{P} (\tilde{B}_{k_1} \in (a - 1,
    a])^{\frac{1}{p_2}},
  \end{align*}     
  where $p_2^{-1} + (2q_2)^{-1} = 1$ \rafal{I guess it should be $p_2^{-1} + 2q_2^{-1} = 1$}. 
  By Lemma~\ref{lem:U_uniform_bound}, 
  \begin{align*}
    \mathbb{E} [U_{k_1, k_2} (\zeta^{-H},
    \tilde{B})^{p_1 q_2}]^{\frac{1}{q_2}} \lesssim_{p_1, p_2, \zeta} 1.
  \end{align*}
  Since
   \begin{align*}
     \norm{\tilde{B}_{r_1} - \tilde{B}_{r_2}}_{L^2(\mathbb{P})}^2
     \leq \norm{B_{r_1} - B_{r_2}}_{L^2(\mathbb{P})}^2 \lesssim \abs{r_1 - r_2}^{2 H},
   \end{align*}
   by Kolmogorov's continuity theorem and Fernique's theorem (note that $k_2 - k_1 = 1$)
   we obtain
   \begin{align*}
     \mathbb{P} (\max_{r \in [k_1, k_2]} | \tilde{B}_r - \tilde{B}_{k_1} |
     \geq | a | - 1) \lesssim e^{- c a^2 }. 
   \end{align*}
   Finally, considering the density of $\tilde{B}_{k_1}$ we see
     $\mathbb{P}
     (\tilde{B}_{k_1} \in (a - 1, a]) \lesssim k_1^{- H}$,
  \toyomu{Since the density of $\tilde{B}_{k_1}$ is bounded by $k_1^{-H}$ }and the claim follows.
\end{proof}

\begin{lemma}[Fix time of $Y$]
  \label{lem:fix-Y} 
  For every $p_1 \in (1, \infty)$, 
  if $\frac{t-s}{s-u}$ is sufficiently small, we have
  \begin{equation*}
     | \mathbb{E} [U_{s, t} (\epsilon, Y_s + Z)] -\mathbb{E} [U_{s, t}
     (\epsilon, Y + Z)] | \lesssim_{\zeta, p_1}  \Big( \frac{t-s}{s-v} \Big)^{H/p_1} 
     \Big( \frac{t-s}{s-u} \Big)^{1 - H}. 
  \end{equation*}
\end{lemma}

\begin{proof}
  As in  Lemma \ref{lem:fix-X-and-Y}, we get 
  \begin{align*}
     | \mathbb{E} [U_{s, t} (\epsilon, Y_s + Z)] -\mathbb{E} [U_{s, t}
     (\epsilon, Y + Z)] | \lesssim_{\zeta, p_1} \norm{U_{s, t}(\epsilon, Y+Z)}_{L^{p_1}}  
     \Big( \frac{t-s}{s-u} \Big)^{1 - H}. 
  \end{align*}
  We then apply Lemma~\ref{lem:U-moment}.
\end{proof}

\begin{lemma}[$\sigma_Y$ vs $\sigma_{Y+Z}$]\label{lem:sigma_Y_vs_sigma_Y_Z}
  If $\frac{s-u}{s-v} \leq \frac{1}{2}$, for every $p_1 \in (1, \infty)$, we have 
  \begin{equation*}
   \Big|  e^{- \frac{1}{2} (
     \frac{X_s}{\sigma_{Y + Z}} )^2} - 
     e^{- \frac{1}{2} ( \frac{X_s}{\sigma_Y} )^2}
     \Big| 
     \lesssim_{p_1} e^{- \frac{1}{2p_1} (
     \frac{X_s}{\sigma_{Y + Z}} )^2} \Big( \frac{s-u}{s-v} \Big)^{2H}.
  \end{equation*}
\end{lemma}
\begin{proof}
  Recalling \eqref{eq:def_of_sigma_Y}, we have
  \begin{align*}
    | \sigma_{Y + Z}^{- 2} - \sigma_Y^{- 2} | =
     \sigma_{Y+Z}^{-2} \sigma_Y^{-2} \frac{(s-u)^{2H}}{2H} 
    \lesssim \sigma_{Y+Z}^{-2} \Big( \frac{s-u}{s-v} \Big)^{2H}.
  \end{align*}
\rafal[inline]{I think that $2H$ should be moved above the fraction bar:}

  %
  Using the inequality $1-e^{-\lambda} \leq \lambda$ for $\lambda \geq 0$,  we observe
  \begin{align*}
    e^{- \frac{1}{2} (
     \frac{X_s}{\sigma_{Y + Z}} )^2} - e^{-
    \frac{1}{2} (
     \frac{X_s}{\sigma_{Y }} )^2} 
    & =  e^{-
    \frac{1}{2} (
     \frac{X_s}{\sigma_{Y + Z}} )^2} \Big\{1 - e^{-\frac{X_s^2}{2} (\sigma_Y^{-2} - \sigma_{Y+Z}^{-2})} \Big \}\\
    & \lesssim  e^{- \frac{1}{2} (
     \frac{X_s}{\sigma_{Y + Z}} )^2} X_s^2 (\sigma_Y^{-2} - \sigma_{Y+Z}^{-2})\\
    & \lesssim  e^{- \frac{1}{2} (
     \frac{X_s}{\sigma_{Y + Z}} )^2} \sigma_{Y+Z}^{-2} X_s^2  \Big( \frac{s-u}{s-v} \Big)^{2H}.
  \end{align*}
  Since $\sup_{\lambda \geq 0} \lambda e^{-(\frac{1}{2} - \frac{1}{2p_1}) \lambda} < \infty$,
  we obtain the claimed estimate.
\end{proof}
\begin{lemma}[Combining estimates obtained so far]\label{lem:U_conditioning_rough}
  For every $p \in (1, \infty)$, $\kappa \in (1 - p^{-1}, 1)$ and $p_1 \in (1, 2)$,
  if $\frac{t-s}{s-u}$ and $\frac{s-u}{u-v}$ are sufficiently small, we have 
  \begin{equation} \label{lemma_316_est}
    \mathbb{E} [U_{s, t} (\epsilon, B - a) | \mathcal{F}_v] 
    = \mathbb{E}[U_{s-v, t-v}(\epsilon, \tilde{B})] e^{-\frac{1}{2}(\frac{X_s}{\sigma_{Y+Z}})^2}
    + R^1_{v, u, s,t}(a) + R^2_{v, u, s, t}(a),
  \end{equation}
  where  
  \begin{align}\label{eq:R_1_precise}
    \norm{R^1_{v, u, s,t}(a)}_{L^p(\mathbb{P})} \lesssim_{p, \zeta, \kappa} 
    \mathbb{E}[U_{s,t}(\epsilon, B-a)]^{1- \kappa} e^{-c a^2} \Big( \frac{t-s}{s-u} \Big)^{1 - H}
  \end{align}
  with $c$ being a constant depending only on $H, \kappa, p$,
  and almost surely
  \begin{equation}\label{eq:R_2_precise}
    \abs{R^2_{v, u, s, t}(a)} \lesssim_{p_1, \zeta} 
    e^{-\frac{1}{2p_1} ( \frac{X_s}{\sigma_{Y+Z}} )^2}
    \frac{ (t - u)^H}{\sigma_{Y+Z}} \Big( \frac{t-s}{s-v} \Big)^{H /p_1} 
    + e^{-\frac{1}{2} (\frac{X_s}{\sigma_Y})^2} \Big( \frac{t-s}{s-v} \Big)^{ H/p_1} \Big( \frac{t-s}{s-u} \Big)^{1-H}.
  \end{equation}
  In particular, 
  \begin{multline}\label{eq:R_1_R_2_rough}
    \norm[\Big]{\mathbb{E} [U_{s, t} (\epsilon, B - a) | \mathcal{F}_v] 
    - \mathbb{E}[U_{s-v, t-v}(\epsilon, \tilde{B})] e^{-\frac{1}{2}(\frac{X_s}{\sigma_{Y+Z}})^2}}_{L^p(\mathbb{P})} \\
    \lesssim_{p, \zeta, \kappa}
    \Big( \frac{t-s}{s-u} \Big)^{1-H} + \Big( \frac{t-s}{s-u} \Big)^{-H} \Big( \frac{t-s}{s-v} \Big)^{(2-\kappa)H}.
  \end{multline}
  \rafal[inline]{I think that in the expression for $R^2$, $(t-u)^H$ should be replaced by $(t-u)^{2H} \vee (t-u)^{H}$ ($(t-u)^{2H}$ appears in the estimate for $R_4$) and $p_1$ by $p_1^2$: 
  TM: we have $t-u \leq 1$ and we can replace $p_1^2$ by $p_1$.
RL: OK.  }
\end{lemma}
\begin{proof}
  In view of Lemma~\ref{lem:Y_change_of_variables}, we decompose 
  \begin{equation*}
     \mathbb{E}[U_{s, t}(\epsilon, B - a) \vert \mathcal{F}_v] = R_1 + R_2 + R_3 + R_4 + R_5, 
  \end{equation*}
  where
  \begin{align*}
    R_1 &\assign  
    \mathbb{E} [U_{s, t} (\epsilon, B-a) | \mathcal{F}_v] -\mathbb{E}
    [ U_{s, t} (\epsilon,  X_s + Y_s + Z ) | \mathcal{F}_v ] , \\ 
    R_2  &\assign  
    e^{- \frac{1}{2} ( \frac{X_s}{\sigma_{Y}} )^2} \mathbb{E} \Big[ \Big.
    e^{\frac{X_s Y_s}{\sigma_{Y}^2}} U_{s, t}(\epsilon, Y_s + Z)  \Big| \mathcal{F}_v
    \Big] 
      -  e^{- \frac{1}{2} ( \frac{X_s}{\sigma_{Y}} )^2} 
      \mathbb{E} [ U_{s, t} (\epsilon, Y_s + Z )], \\
    R_3  &\assign  
    e^{- \frac{1}{2} ( \frac{X_s}{\sigma_{Y}} )^2} \mathbb{E} [ U_{s, t} (\epsilon, Y_s + Z ) ] 
      -  e^{- \frac{1}{2} ( \frac{X_s}{\sigma_{Y}} )^2} \mathbb{E} 
      [U_{s, t} (\epsilon, Y+Z)], \\
    R_4 &\assign  e^{-\frac{1}{2} ( \frac{X_s}{\sigma_{Y}} )^2} \mathbb{E} [U_{s, t} (\epsilon, Y+Z)]
    - e^{-
    \frac{1}{2} ( \frac{X_s}{\sigma_{Y + Z}} )^2} \mathbb{E} [U_{s, t} (\epsilon, Y+Z)],  \\
    R_5 &\defby   
    e^{-\frac{1}{2} ( \frac{X_s}{\sigma_{Y + Z}} )^2} \mathbb{E} [U_{s, t} (\epsilon, Y+Z)] 
    = \mathbb{E}
    [U_{s-v, t-v} (\epsilon, \tilde{B})] e^{-\frac{1}{2} ( \frac{X_s}{\sigma_{Y + Z}} )^2}.
  \end{align*}
  By Lemma \ref{lem:fix-X-and-Y},
  \begin{equation*}
     \| R_1 \|_{L^p (\mathbb{P})}  \lesssim_{H, p, \zeta} 
    \mathbb{E}[U_{s, t}(\epsilon, B - a)]^{1 - \kappa} e^{- c a^2}
   \Big( \frac{t-s}{s - u} \Big)^{1-H} . 
  \end{equation*}
  By Lemma \ref{lem:fix-Y}, 
  \begin{align*}
    \abs{R_3} \lesssim_{\zeta, p_1} 
    e^{-\frac{1}{2} (\frac{X_s}{\sigma_Y})^2} \Big( \frac{t-s}{s-v} \Big)^{H/p_1}\Big( \frac{t-s}{s - u} \Big)^{1-H}.
  \end{align*}
 To estimate $R_2$, by Lemma \ref{lem:exp-U},  
  \begin{align*}
    \abs{R_2} \lesssim e^{-\frac{1}{2} ( \frac{X_s}{\sigma_{Y}} )^2}
    \frac{| X_s | (t - u)^H}{\sigma_Y^2}
     e^{c ( \frac{| X_s | (t - u)^H}{\sigma_Y^2} )^2} \| U_{s,
     t} (\epsilon, Y + Z) \|_{L^{p_1} (\mathbb{P})} .
  \end{align*}
  If $\frac{t-u}{u-v}$ is sufficiently small, we have 
  \begin{align*}
    c  (\frac{(t-u)^H}{\sigma_Y} )^2 \leq \frac{1}{2} - \frac{1}{2 p_1},
  \end{align*}
  hence 
  \begin{align*}
    \abs{R_2} \lesssim e^{-\frac{1}{2 p_1} ( \frac{X_s}{\sigma_{Y}} )^2}
    \frac{| X_s | (t - u)^H}{\sigma_Y^2}
     \| U_{s, t} (\epsilon, Y + Z) \|_{L^{p_1} (\mathbb{P})}.
  \end{align*}
  Using the estimate $\sup_{\lambda \geq 0} \lambda e^{-(\frac{1}{2p_1} - \frac{1}{2p_1^2}) \lambda^2} < \infty$ and Lemma~\ref{lem:U-moment}, 
  we get 
  \begin{align*}
    \abs{R_2} &\lesssim e^{-\frac{1}{2 p_1^2} ( \frac{X_s}{\sigma_{Y}} )^2}
    \frac{ (t - u)^H}{\sigma_Y} \Big( \frac{t-s}{s-v} \Big)^{H/p_1^2} \\
    &\lesssim e^{-\frac{1}{2 p_1^2} ( \frac{X_s}{\sigma_{Y+Z}} )^2}
    \frac{ (t - u)^H}{\sigma_{Y+Z}} \Big( \frac{t-s}{s-v} \Big)^{H/p_1^2}.
  \end{align*}
  Finally, we estimate $R_4$. By Lemma~\ref{lem:sigma_Y_vs_sigma_Y_Z}, we get 
  \begin{align*}
    \abs{R_4} \lesssim_{p_1} \| U_{s,
     t} (\epsilon, Y + Z) \|_{L^{1} (\mathbb{P})}
     e^{- \frac{1}{2p_1} (
     \frac{X_s}{\sigma_{Y + Z}} )^2} \Big( \frac{s-u}{s-v} \Big)^{2H}.
  \end{align*}
  By Lemma~\ref{lem:U-moment}, we obtain 
  \begin{align*}
    \abs{R_4} 
    &\lesssim e^{-\frac{1}{2 p_1} ( \frac{X_s}{\sigma_{Y+Z}} )^2}
    \Big(\frac{ s-u}{s-v} \Big)^{2H} \Big( \frac{t-s}{s-v} \Big)^{H/p_1^2}\\
    &\lesssim e^{-\frac{1}{2 p_1^2} ( \frac{X_s}{\sigma_{Y+Z}} )^2}
    \frac{ (t - u)^{2H}}{\sigma_{Y+Z}} \Big( \frac{t-s}{s-v} \Big)^{H/p_1^2}.
  \end{align*}
  Setting $R_{v, u, s, t}^1(a) \defby R_1 $ and $R_{v, u, s, t}^2(a) \defby R_2 + R_3 + R_4$, we get the estimates 
  \eqref{eq:R_1_precise} and \eqref{eq:R_2_precise}. 
  In particular, by the trivial bound on the exponential function and Lemma~\ref{lem:U_uniform_bound},
  \begin{equation*}
    \begin{gathered}
    \norm{R^1_{v, u, s, t}(a)}_{L^p(\mathbb{P})} \lesssim_{p, \zeta} \Big( \frac{t-s}{s-u} \Big)^{1-H}, \\
    \abs{R^2_{v, u, s, t}(a)} \lesssim_{\zeta, \kappa} \Big( \frac{s-u}{t-s} \Big)^{H} \Big( \frac{t-s}{s-v} \Big)^{(2-\kappa)H} 
    + \Big( \frac{t-s}{s-u} \Big)^{1-H} \Big( \frac{t-s}{s-v} \Big)^{(1-\kappa)H}.
    \end{gathered}
  \end{equation*}
  Since $\frac{t-s}{s-v} \leq 1$, these bounds lead to \eqref{eq:R_1_R_2_rough}.
\end{proof}

As a final ingredient in the proof of Lemma~\ref{lem:U_conditioning}, we estimate $\mathbb{E}[U_{s-v, t-v}(\epsilon, \tilde{B})]$. 
\begin{lemma}[Asymptotics on constants]\label{lem:expect_U_convergence}
  For every $\kappa \in (0, 1)$, if $\frac{t-s}{s-u}$ and $\frac{s-u}{u-v}$ are sufficiently small, we have the estimate \rafal{notation $\bar{K}_0^{\zeta}(1, B)$ or $\bar{K}_{0,{\zeta}}(1, B)$}
  \begin{multline*}
    \abs[\Big]{\sqrt{2\pi} \epsilon^{-1} \sigma_{Y+Z} \mathbb{E}[U_{s-v, t-v}(\epsilon, \tilde{B})] 
    - \frac{1}{2} \mathbb{E}[\bar{K}_{0, \zeta}(1, B)]} \\
    \lesssim_{\kappa, \zeta} \Big( \frac{t-s}{s-u} \Big)^{-H} \Big( \frac{t-s}{s-v} \Big)^{(1- \kappa) H} +(t-s)^{-\kappa H} \Big( \frac{t-s}{s-u} \Big)^{1-H}.
  \end{multline*}
\end{lemma}
\begin{proof}
  By integrating \eqref{lemma_316_est} \rafal{I changed this, since in the previous version there was a reference to estimate in Lemma 3.7 which we still need to prove} over $\mathbb{R}$ with 
  respect to $a$, we get 
  \begin{multline}\label{eq:U_expect_int}
    \int_{\mathbb{R}} \mathbb{E}[U_{s, t}(\epsilon, B-a)] \mathrm{d} a 
    = \mathbb{E}[U_{s-v, t-v}(\epsilon, \tilde{B})] \mathbb{E}\Big[ \int_{\mathbb{R}} e^{-\frac{1}{2} (\frac{X_s}{\sigma_{Y+Z}})^2} \mathrm{d} a \Big] \\
    + \mathbb{E}\Big[ \int_{\mathbb{R}} R^1_{v, u, s, t}(a) \mathrm{d} a \Big]
    + \mathbb{E}\Big[ \int_{\mathbb{R}} R^2_{v, u, s, t}(a) \mathrm{d} a \Big].
  \end{multline}
  We will estimate each term of \eqref{eq:U_expect_int}.

  By the scaling (Lemma~\ref{lem:U-scaling}), 
  \begin{align*}
    \int_{\mathbb{R}} \mathbb{E}[U_{s, t}(\epsilon, B - a)] \mathrm{d} a 
    &= \int_{\mathbb{R}} \mathbb{E}[U_{\frac{s}{t-s}\zeta, \frac{t}{t-s} \zeta} 
    (1, B - \zeta^{H} (t-s)^{-H} a)] \mathrm{d} a  \\ 
    &= \zeta^{-H} (t-s)^H  \int_{\mathbb{R}} \mathbb{E}[U_{\frac{s}{t-s}\zeta, \frac{t}{t-s} \zeta} (1, B -  a)] \mathrm{d} a   \\ 
    &= \epsilon \int_{\mathbb{R}} \mathbb{E}[U_{\frac{s}{t-s}\zeta, \frac{t}{t-s} \zeta} (1, B -  a)] \mathrm{d} a. 
  \end{align*}
  We recall the downcrossing $D_{s, t}(\epsilon, w)$ from \eqref{eq:def_of_D}. By definition, we have 
  \begin{align*}
    \bar{K}_{s, t}(\epsilon, w) = \int_{\mathbb{R}} \{U_{s, t}(\epsilon, w - a) + D_{s, t}(\epsilon, w - a)\} \mathrm{d} a 
  \end{align*}
  Since 
  $D_{s, t}(\epsilon, w) = U_{s, t}(\epsilon, -w - \epsilon)$
  and $B \dequal - B$, we obtain 
  \begin{align*}
    \int_{\mathbb{R}} \mathbb{E}[U_{\frac{s}{t-s}\zeta, \frac{t}{t-s} \zeta} (1, B -  a)] \mathrm{d} a
    =\int_{\mathbb{R}} \mathbb{E}[D_{\frac{s}{t-s}\zeta, \frac{t}{t-s} \zeta} (1, B -  a)] \mathrm{d} a
    = \frac{1}{2} \mathbb{E}[\bar{K}_{\frac{s}{t-s}\zeta, \frac{t}{t-s}\zeta}(1, B)].
  \end{align*}
  By the stationarity of $\bar{K}$ (Lemma~\ref{lem:K-bar-subadditive-and-stationary}),
  \begin{align*}
    \mathbb{E}[\bar{K}_{\frac{s}{t-s}\zeta, \frac{t}{t-s}\zeta}(1, B)] = \mathbb{E}[\bar{K}_{0, \zeta}(1, B)].
  \end{align*}
  Therefore, 
  \begin{align}\label{eq:U_expect_int_bar_K} 
    \int_{\mathbb{R}} \mathbb{E}[U_{s, t}(\epsilon, B - a)] \mathrm{d} a = \frac{\mathbb{E}[\bar{K}_{0, \zeta}(1, B)]}{2} \epsilon.
  \end{align}

  Recalling how $X$ depends on $a$ from \eqref{eq:def_X_Y_Z}, for any $\sigma >0$
  \begin{equation*}
    \int_{\mathbb{R}} e^{-\frac{1}{2} ( \frac{X_s}{\sigma} )^2} \mathrm{d} a = \sqrt{2 \pi} \sigma ,
  \end{equation*}
  in particular, we have 
  \begin{align}\label{eq:int_X_over_a}
    \int_{\mathbb{R}} e^{-\frac{1}{2} (\frac{X_s}{\sigma_{Y+Z}})^2} \mathrm{d} a = \sqrt{2 \pi} \sigma_{Y+Z}, \quad \int_{\mathbb{R}} e^{-\frac{1}{2} (\frac{X_s}{\sigma_{Y}})^2} \mathrm{d} a = \sqrt{2 \pi} \sigma_{Y}.
  \end{align}
  Combining \eqref{eq:U_expect_int}, \eqref{eq:U_expect_int_bar_K} and \eqref{eq:int_X_over_a}, we obtain 
  \begin{align*}
    \abs[\Big]{\sqrt{2\pi} \epsilon^{-1} \sigma_{Y+Z} \mathbb{E}[U_{s-v, t-v}(\epsilon, \tilde{B})] 
    -  \frac{\mathbb{E}[\bar{K}_{0, \zeta}(1, B)]}{2}} 
    \leq  \epsilon^{-1} \sum_{i = 1, 2} \int_{\mathbb{R}} \norm{R^i_{v, u, s, t}(a)}_{L^1(\mathbb{P})} \mathrm{d} a.
  \end{align*}
  It remains to estimate the right-hand side.

  By \eqref{eq:R_1_precise}, 
  \begin{align*}
    \norm{R^1_{v, u, s, t}(a)}_{L^1(\mathbb{P})}
    \lesssim_{\zeta, \kappa} \mathbb{E}[U_{s, t}(\epsilon, B - a)]^{1-\kappa} e^{-c a^2} \Big( \frac{t-s}{s-u} \Big)^{1-H}.
  \end{align*}
  By Jensen's inequality and \eqref{eq:U_expect_int_bar_K},
  \begin{align*}
    \int_{\mathbb{R}} \mathbb{E}[U_{s,t}(\epsilon, B - a)]^{1-\kappa} e^{-c a^2} \mathrm{d} a 
    \lesssim_{\kappa} \Big( \int_{\mathbb{R}} \mathbb{E}[U_{s,t}(\epsilon, B - a)] e^{-c a^2} \mathrm{d} a  \Big)^{1-\kappa}  
    \lesssim_{\zeta, \kappa} \epsilon^{1-\kappa}.
  \end{align*}
  This gives an estimate for $R^1$. 
  To estimate $R^2$ we use \eqref{eq:R_2_precise} and \eqref{eq:int_X_over_a}, and obtain
  \begin{align*}
    \int_{\mathbb{R}} \mathbb{E} \left| R^2_{v, u, s, t}(a) \right| \mathrm{d} a & = \mathbb{E}  \int_{\mathbb{R}} \left| R^2_{v, u, s, t}(a) \right| \mathrm{d} a \\
   & \lesssim_{\zeta, \kappa} (t-u)^H \Big( \frac{t-s}{s-v} \Big)^{(1-\kappa)H} + \sigma_Y \Big( \frac{t-s}{s-v} \Big)^{(1-\kappa)H} 
    \Big( \frac{t-s}{s-u} \Big)^{1-H}.
  \end{align*}
  Recalling \eqref{eq:def_of_sigma_Y}, we have $\sigma_Y \lesssim (s-v)^H$. Using also $t-u \leq 2 (s-u)$ and $\epsilon = (\frac{t-s}{\zeta})^H$, 
  we get
  \begin{align*}
    \epsilon^{-1} \int_{\mathbb{R}} \norm{R^2_{v, u, s, t}(a)}_{L^1(\mathbb{P})} \mathrm{d} a 
    \lesssim_{\zeta, \kappa} \Big( \frac{t-s}{s-u} \Big)^{-H} 
    \Big( \frac{t-s}{s-v} \Big)^{(1-\kappa)H} +  \Big( \frac{t-s}{s-v} \Big)^{-\kappa H} 
    \Big( \frac{t-s}{s-u} \Big)^{1-H}.
  \end{align*}
  Noting that $( \frac{t-s}{s-v} )^{-\kappa H} \leq (t-s)^{-\kappa H}$ (recall that we assume that $0 \leq v < s < t \le 1$), we conclude the proof.
\end{proof}
\begin{proof}[Proof of Lemma~\ref{lem:U_conditioning}]
  By the bound \eqref{eq:R_1_R_2_rough} in Lemma~\ref{lem:U_conditioning_rough}, we have 
  \begin{multline*}
    \norm[\Big]{U_{s, t}(\epsilon, B - a) -   \frac{\mathbb{E}[\bar{K}_{0, \zeta}(1, B)] \epsilon}{2\sqrt{2 \pi} \sigma_{Y+Z}} 
    e^{-\frac{1}{2} (\frac{X_s}{\sigma_{Y+Z}})^2}  }_{L^p(\mathbb{P})}  \\
    \lesssim_{p, \zeta, \kappa} 
    \norm[\Big]{\Big( \mathbb{E}[U_{s-v, t-v}(\epsilon, \tilde{B})] -  \frac{\mathbb{E}[\bar{K}_{0, \zeta}(1, B)]}{2\sqrt{2 \pi} \sigma_{Y+Z}} \epsilon \Big) e^{-\frac{1}{2} (\frac{X_s}{\sigma_{Y+Z}})^2}}_{L^p(\mathbb{P})} \\
    + \Big( \frac{t-s}{s-u} \Big)^{1-H} + \Big( \frac{t-s}{s-u} \Big)^{-H} \Big( \frac{t-s}{s-v} \Big)^{(2-\kappa)H}
  \end{multline*}
  if $\frac{t-s}{s-u}$ and $\frac{s-u}{u-v}$ are sufficiently small.
  By the bound $\frac{\epsilon}{\sigma_{Y+Z}} \lesssim_{\zeta} (\frac{t-s}{s-v})^H \leq 1$ and Lemma~\ref{lem:expect_U_convergence} we have 
  \begin{multline*}
    \abs[\Big]{\mathbb{E}[U_{s-v, t-v}(\epsilon, \tilde{B})] -  \frac{\mathbb{E}[\bar{K}_{0, \zeta}(1, B)]}{2\sqrt{2 \pi} \sigma_{Y+Z}} \epsilon} \le 
    \abs[\Big]{{\sqrt{2 \pi} \epsilon^{-1} \sigma_{Y+Z}} \mathbb{E}[U_{s-v, t-v}(\epsilon, \tilde{B})] -  \frac{1}{2}\mathbb{E}[\bar{K}_{0, \zeta}(1, B)] }
    \\
    \lesssim_{\zeta, \kappa} 
    (t-s)^{-\kappa H} \Big( \frac{t-s}{s-u} \Big)^{1-H} + \Big( \frac{t-s}{s-u} \Big)^{-H} \Big( \frac{t-s}{s-v} \Big)^{(2-\kappa)H}.
  \end{multline*}
  Therefore, 
  \begin{multline*}
    \norm[\Big]{U_{s, t}(\epsilon, B - a) -   \frac{\mathbb{E}[\bar{K}_{0, \zeta}(1, B)] \epsilon}{2\sqrt{2 \pi} \sigma_{Y+Z}} 
    e^{-\frac{1}{2} (\frac{X_s}{\sigma_{Y+Z}})^2}  }_{L^p(\mathbb{P})}  \\
    \lesssim_{p, \zeta, \kappa} 
    (t-s)^{-\kappa H} \Big( \frac{t-s}{s-u} \Big)^{1-H} + \Big( \frac{t-s}{s-u} \Big)^{-H} \Big( \frac{t-s}{s-v} \Big)^{(2-\kappa)H}
  \end{multline*}
  if $\frac{t-s}{s-u}$ and $\frac{s-u}{u-v}$ are sufficiently small.
  To optimize, we choose $u$ so that 
  \begin{align*}
    \frac{t-s}{s-u} = \Big( \frac{t-s}{s-v} \Big)^{(2-\kappa) H}.
  \end{align*}
  Note that, as $H < 1/2$, the exponent $(2 - \kappa) H$ is less than $1$.
  Therefore, if $\frac{t-s}{s-v}$ is sufficiently small, then 
  $\frac{t-s}{s-u}$ and $\frac{s-u}{u-v}$ are sufficiently small as well.
  This gives the claimed bound.
\end{proof}
\begin{remark}\label{rem:U_bar_conditioning}
  Recall $\bar{U}$ from Notation~\ref{not:U_bar}. 
  Since 
  \begin{align*}
    \int_{\mathbb{R}} \mathbb{E}[\bar{U}_{s, t}(\epsilon, B - a)] \mathrm{d} a = 
    \frac{\mathbb{E}[\bar{K}_{0, \zeta}(1, B)] + 1}{2} \epsilon,
  \end{align*}
  we similarly obtain 
  \begin{align*}
    \mathbb{E} [\bar{U}_{s, t} (\epsilon, B - a) | \mathcal{F}_v] 
    =  \frac{\mathbb{E}[\bar{K}_{0, \zeta}(1, B)] + 1}{2\sqrt{2 \pi} \sigma_{Y+Z}} 
    e^{-\frac{1}{2}(\frac{X_s}{\sigma_{Y+Z}})^2} \epsilon
    + \bar{R}_{v,s,t}
  \end{align*}
  with, provided that $\frac{t-s}{s-v}$ is sufficiently small, 
  \begin{align*}
    \norm{\bar{R}_{v, s, t}}_{L^p(\mathbb{P})}
    \lesssim_{p, \zeta, \kappa} \Big( \frac{t-s}{s-v} \Big)^{(2 - \kappa) H (1-H)} (t-s)^{- \kappa H}.
  \end{align*}
\end{remark}

\subsubsection{Estimates on the local time}
The following is the last technical ingredient for Theorem~\ref{thm:local-time-level-crossing}.
\begin{lemma}[Local time approximation]
  \label{lem:L-germ}Let $H \in (0, 1/2)$. We set
  \begin{align}
    \tilde{A}_{s, t} &\assign \mathbb{E} [\delta_0 (B_s - a) | \mathcal{F}_{s -
    (t - s)}] (t - s) 
    \label{eq:def-of-tilde-A} \\
    &\phantom{\vcentcolon}= \sqrt{\frac{H}{\pi}} e^{- \frac{H}{(t - s)^{2 H}}
    \mathbb{E} [B_s - a| \mathcal{F}_{s - (t - s)}]^2} (t - s)^{1 - H} . \notag
  \end{align}
  Then, there exists a $\delta> 0$ such that for any $p < \infty$ and 
  for any partition $\pi$ of $[0, 1]$,
  \begin{equation*}
     \Big\| L_1 (a) - \sum_{[s, t] \in \pi} \tilde{A}_{s, t} \Big\|_{L^p
     (\mathbb{P})} \lesssim_p | \pi |^{\delta} . 
  \end{equation*}
\end{lemma}

\begin{proof}
  We write $L_{s, t}(a) \defby L_t(a) - L_s(a)$.
  We use the shifted stochastic sewing (Lemma~\ref{lem:shifted_ssl}). To this end, it suffices to check
  \begin{equation}
    \| L_{s, t} (a) \|_{L^p (\mathbb{P})} \lesssim_p (t - s)^{1 - H}, \quad \|
    \tilde{A}_{s, t} \|_{L^p (\mathbb{P})} \lesssim_p (t - s)^{1 - H}
    \label{eq:L-and-A-tilde}
  \end{equation}
  and
  \begin{equation}
    \| \mathbb{E} [L_{s, t} (a) - \tilde{A}_{s, t} | \mathcal{F}_v] \|_{L^p
    (\mathbb{P})} \lesssim_p (s - v)^{- 1 - H} (t - s)^2, \quad t-s \leq s-v. \label{eq:L-tilde-A-conditioning}
  \end{equation}
  \rafal[inline]{I am not sure if in the above relations one should not have $\sim_p$ instead of $\sim$. TM: Yes, it depends on $p$}
  The estimate for $L$ in {\eqref{eq:L-and-A-tilde}} is well known and can be shown for example by (non-shifted) stochastic sewing with $\Xi_{u,v} = \mathbb{E}[\int_u^v \delta(B_r-a)\mathrm{d}r|\mathcal F_u]$, and the estimate for $\tilde A$ in {\eqref{eq:L-and-A-tilde}} is not difficult to show. Hence, we focus on the
  estimate {\eqref{eq:L-tilde-A-conditioning}}. In \cite[Lemma~4.7]{matsuda22}, an 
  estimate similar to {\eqref{eq:L-tilde-A-conditioning}} is obtained, but the exponents therein
  depend on $p$. We slightly improve the argument.
  
  We have \rafal{Some reference or explanation where does these relationships (also (45)) come from would be desirable, $L_{s,t}$ needs to be defined}
  \begin{multline*}
   \mathbb{E} [L_{s, t} (a) - \tilde{A}_{s, t} | \mathcal{F}_v] 
   =
   \sqrt{\frac{H}{\pi}} \int_s^t \Big\{ e^{- \frac{H}{(r - v)^{2 H}}
   \mathbb{E} [B_r - a| \mathcal{F}_v]^2} (r - v)^{- H} \\
   - e^{- \frac{H}{(s -
   v)^{2 H}} \mathbb{E} [B_s - a| \mathcal{F}_v]^2} (s - v)^{- H} \Big\}
   \mathrm{d} r. 
  \end{multline*}
  For simplification, we replace $B - a$ by $B$. We decompose the integrand as
  $R_1 + R_2 + R_3$, where
  \begin{equation*}
     R_1 \assign e^{- \frac{H}{(r - v)^{2 H}} \mathbb{E} [B_r |
     \mathcal{F}_v]^2} (r - v)^{- H} - e^{- \frac{H}{(r - v)^{2 H}} \mathbb{E}
     [B_r | \mathcal{F}_v]^2} (s - v)^{- H}, 
  \end{equation*}
  \begin{equation*}
     R_2 \assign e^{- \frac{H}{(r - v)^{2 H}} \mathbb{E} [B_r |
     \mathcal{F}_v]^2} (s - v)^{- H} - e^{- \frac{H}{(s - v)^{2 H}} \mathbb{E}
     [B_r | \mathcal{F}_v]^2} (s - v)^{- H}, 
  \end{equation*}
  \begin{equation*}
     R_3 \assign e^{- \frac{H}{(s - v)^{2 H}} \mathbb{E} [B_r |
     \mathcal{F}_v]^2} (s - v)^{- H} - e^{- \frac{H}{(s - v)^{2 H}} \mathbb{E}
     [B_s | \mathcal{F}_v]^2} (s - v)^{- H} . 
  \end{equation*}
  To obtain \eqref{eq:L-tilde-A-conditioning}, it suffices to show 
  \begin{align*}
    \norm{R_1}_{L^p(\mathbb{P})} + \norm{R_2}_{L^p(\mathbb{P})} + \norm{R_3}_{L^p(\mathbb{P})}
    \lesssim_p (s-v)^{-1-H} (t-s).
  \end{align*}

  Since
  \begin{equation*}
     0 \leq (s - v)^{- H} - (r - v)^{- H} \lesssim (s - v)^{- H - 1} (r - s),
  \end{equation*}
  we have
  \begin{equation*}
     | R_1 | \lesssim (s - v)^{- H - 1} (t - s) . 
  \end{equation*}
  We observe
  \begin{align*}
   \MoveEqLeft[3]
    e^{- \frac{H}{(r - v)^{2 H}} \mathbb{E} [B_r | \mathcal{F}_v]^2} - e^{-
    \frac{H}{(s - v)^{2 H}} \mathbb{E} [B_r | \mathcal{F}_v]^2} \\
    & =  e^{-
    \frac{H}{(r - v)^{2 H}} \mathbb{E} [B_r | \mathcal{F}_v]^2} (1 - e^{- H
    ((s - v)^{- 2 H} - (r - v)^{- 2 H}) \mathbb{E} [B_r | \mathcal{F}_v]^2})\\
    & \lesssim  e^{- \frac{H}{(r - v)^{2 H}} \mathbb{E} [B_r |
    \mathcal{F}_v]^2} \mathbb{E} [B_r | \mathcal{F}_v]^2 ((s - v)^{- 2 H} - (r
    - v)^{- 2 H})\\
    & \lesssim  e^{- \frac{H}{(r - v)^{2 H}} \mathbb{E} [B_r |
    \mathcal{F}_v]^2} \mathbb{E} [B_r | \mathcal{F}_v]^2 (s - v)^{- 2 H - 1}
    (r - s)\\
    & \lesssim  (r - v)^{2 H} (s - v)^{- 2 H - 1} (r - s) \\
    & \lesssim (s-v)^{-1}(t-s),
  \end{align*}
  where in the last step we used that $(r-v)\leq (t-s) + (s-v) \leq 2(s-v)$. Hence,
  \begin{equation*}
     | R_2 | \lesssim (s - v)^{- 1 - H} (t - s) . 
  \end{equation*}
  Finally, we estimate $R_3$. Suppose that $\mathbb{E} [B_r | \mathcal{F}_v]^2
  \leq \mathbb{E} [B_s | \mathcal{F}_v]^2$. Then,
  \begin{multline*}
   \Big| e^{- \frac{H}{(s - v)^{2 H}} \mathbb{E} [B_r | \mathcal{F}_v]^2}
     - e^{- \frac{H}{(s - v)^{2 H}} \mathbb{E} [B_s | \mathcal{F}_v]^2}
     \Big| \\
     \leq e^{- \frac{H}{(s - v)^{2 H}} \mathbb{E} [B_r |
     \mathcal{F}_v]^2} \frac{H}{(s - v)^{2 H}} (\mathbb{E} [B_s |
     \mathcal{F}_v]^2 -\mathbb{E} [B_r | \mathcal{F}_v]^2) . 
  \end{multline*}
  Since
  \begin{align*}
    \mathbb{E}[B_s \vert \mathcal{F}_v]^2 - \mathbb{E}[B_r \vert \mathcal{F}_v]^2
    = 2 \mathbb{E}[B_r \vert \mathcal{F}_v] (\mathbb{E}[B_s \vert \mathcal{F}_v] - \mathbb{E}[B_r \vert \mathcal{F}_v]) 
    + (\mathbb{E}[B_s \vert \mathcal{F}_v] - \mathbb{E}[B_r \vert \mathcal{F}_v])^2
  \end{align*}
  and
  \begin{equation*}
     e^{- \frac{H}{(s - v)^{2 H}} \mathbb{E} [B_r | \mathcal{F}_v]^2} (s -
     v)^{- H} | \mathbb{E} [B_r | \mathcal{F}_v] | \lesssim 1, 
  \end{equation*}
  we obtain
  \begin{multline*}
   \Big| e^{- \frac{H}{(s - v)^{2 H}} \mathbb{E} [B_r | \mathcal{F}_v]^2} -
     e^{- \frac{H}{(s - v)^{2 H}} \mathbb{E} [B_s | \mathcal{F}_v]^2} \Big| \\
     \lesssim (s - v)^{- H} | \mathbb{E} [B_s | \mathcal{F}_v] -\mathbb{E}
     [B_r | \mathcal{F}_v] | + (s - v)^{- 2 H} | \mathbb{E} [B_s |
     \mathcal{F}_v] -\mathbb{E} [B_r | \mathcal{F}_v] |^2 . 
  \end{multline*}
  A similar estimate holds if $\mathbb{E} [B_r | \mathcal{F}_v]^2 \geq
  \mathbb{E} [B_s | \mathcal{F}_v]^2$. Therefore, it remains to note
  \begin{equation*}
     \| \mathbb{E} [B_s | \mathcal{F}_v] -\mathbb{E} [B_r | \mathcal{F}_v]
     \|_{L^p (\mathbb{P})} \lesssim_p (s - v)^{H - 1} (t - s) . \qedhere 
  \end{equation*}
\end{proof}

\subsubsection{Concluding estimates}
Now we can finish the proof of Theorem~\ref{thm:local-time-level-crossing}.
Recall from Remark~\ref{rem:T_to_one} that we can set $T = 1$.
Let $\pi$ be a partition of $[0, 1]$.
By Lemma \ref{lem:U-additivity}, 
\begin{equation}
  \epsilon^{\frac{1}{H} - 1} U_{0, 1} (\epsilon, B - a) 
  \geq \sum_{[s, t] \in \pi} \epsilon^{\frac{1}{H} - 1} U_{s, t}(\epsilon, B - a), \label{eq:U-superadditive}
\end{equation}
\begin{equation}
  \epsilon^{\frac{1}{H} - 1} U_{0, 1} (\epsilon, B - a) 
  \leq \sum_{[s, t] \in \pi} \epsilon^{\frac{1}{H} - 1} \bar{U}_{s, t}(\epsilon, B - a). \label{eq:U-subadditive}
\end{equation}
Here and henceforth, $\epsilon$ is an independent parameter; unlike Subsection~\ref{subsubsec:level_crossings}, we do not 
assume the relation \eqref{eps_sect_3}.
\begin{lemma}[Lower bound on $U$]
  \label{lem:U-lower-bound}  
  Let $H \in (0, 1/2)$, $p \in [2, \infty)$, $\epsilon \in (0, 1)$ and $\zeta \in [1, \infty)$.
  Then, we have
  \begin{equation*}
     \epsilon^{\frac{1}{H} - 1} U_{0, 1} (\epsilon, B-a) 
     \geq \frac{1}{2 \zeta} \mathbb{E}[\bar{K}_{0, \zeta}(1, B)] L_1(a) - R_{\epsilon}, 
  \end{equation*}
  where for some $\delta$ depending only on $H$ we have
  \begin{equation*}
   \| R_{\epsilon} \|_{L^p (\mathbb{P})} \lesssim_{p, \zeta} \epsilon^{\delta}.  
  \end{equation*}
\end{lemma}

\begin{proof}
  We define $\tilde{A}$ by \eqref{eq:def-of-tilde-A}, and we set 
  \begin{align*}
    \hat{A}_{s, t} 
    \defby U_{s, t}(\zeta^{-H} (t-s)^H, B - a) \Big( \frac{t-s}{\zeta} \Big)^{1 - H}.
  \end{align*}
  By Lemma~\ref{lem:U_uniform_bound}, we have 
  \begin{align*}
    \norm{\hat{A}_{s, t}}_{L^p(\mathbb{P})} \lesssim (t-s)^{1 - H}.
  \end{align*}
  By Lemma~\ref{lem:U_conditioning}, 
  \begin{align*}
    \mathbb{E}[\hat{A}_{s,t} \vert \mathcal{F}_v]
    & = \frac{1}{2 \zeta} \mathbb{E}[\bar{K}_{0, \zeta}(1, B)] \sqrt{\frac{H}{\pi (s -v)^{2H}}} 
    e^{-\frac{H \mathbb{E}[B_s \vert \mathcal{F}_v]^2}{(s - v)^{2 H}}} (t-s) + R_{v, s, t} \\
    &= \frac{1}{2 \zeta} \mathbb{E}[\bar{K}_{0, \zeta}(1, B)] \mathbb{E}[\tilde{A}_{s,t} \vert \mathcal{F}_v]
    + R_{v, s, t},
  \end{align*}
  where 
  \begin{align*}
    \norm{R_{v,s,t}}_{L^p(\mathbb{P})}
    \lesssim_{p, \zeta, \kappa} \Big( \frac{t-s}{s-v} \Big)^{(2 - \kappa)H(1-H)} (t-s)^{1 - (1+\kappa) H}.
  \end{align*}
  for any $\kappa \in (0, 1)$.
  Since $H < 1/2$, choosing $\kappa$ sufficiently small, we can suppose that 
  \begin{align*}
    1 - (1 + \kappa) H > \frac{1}{2}, \quad 1 - (1 + \kappa) H + (2 - \kappa)H(1-H) > 1.
  \end{align*}
  Hence, by Lemma~\ref{lem:shifted_ssl}, with some $\delta = \delta(H)$,
  \begin{align*}
    \norm[\Big]{\sum_{[s, t] \in \pi} \Big( \hat{A}_{s,t} - \frac{1}{2 \zeta} \mathbb{E}[\bar{K}_{0, \zeta}(1, B)] \tilde{A}_{s,t} \Big)  }_{L^p(\mathbb{P})}
    \lesssim_{p, \zeta} \abs{\pi}^{\delta}.
  \end{align*}
  In particular, considering a partition of size $\zeta \epsilon^{\frac{1}{H}}$,
  the claim follows in view of \eqref{eq:U-superadditive} and Lemma~\ref{lem:L-germ}.
\end{proof}
\begin{lemma}[Upper bound on $U$]
  \label{lem:U-upper-bound}  
  Let $H \in (0, 1/2)$, $p \in [2, \infty)$, $\epsilon \in (0, 1)$ and $\zeta \in [1, \infty)$.
  Then, we have
  \begin{equation*}
     \epsilon^{\frac{1}{H} - 1} U_{0, 1} (\epsilon, B-a) 
     \leq \frac{1}{2 \zeta} (\mathbb{E}[\bar{K}_{0, \zeta}(1, B)] + 1) L_1(a) + \bar{R}_{\epsilon}, 
  \end{equation*}
  where for some $\delta$ depending only on $H$ we have
  \begin{equation*}
   \| \bar{R}_{\epsilon} \|_{L^p (\mathbb{P})} \lesssim_{p, \zeta} \epsilon^{\delta}.  
  \end{equation*}
\end{lemma}
\begin{proof}
  In view of Remark~\ref{rem:U_bar_conditioning} and {\eqref{eq:U-subadditive}}, the proof is similar to Lemma~\ref{lem:U-lower-bound}.
\end{proof}

\begin{proof}[Proof of Theorem~\ref{thm:local-time-level-crossing}]
  It readily follows from Lemma~\ref{lem:U-lower-bound}, Lemma~\ref{lem:U-upper-bound} 
  and the estimate \eqref{eq:c_H_convergence_bound}.
\end{proof}

\subsection{Uniform convergence}\label{subsec:lemieux}
A naive application of Theorem~\ref{thm:local-time-level-crossing} yields that, by the Borel--Cantelli lemma, for any $a \in \mathbb{R}$ and for
any $\bm{\epsilon} = (\epsilon_n)_{n=1}^{\infty}$ with polynomial decay, there exists 
a measurable set $\Omega_{a, \bm{\epsilon}}$ such that 
$\mathbb{P}(B \in \Omega_{a, \bm{\epsilon}}) = 1$ and for every $w \in \Omega_{a, \bm{\epsilon}}$ 
the limit 
\begin{align*}
  \lim_{n \to \infty} \epsilon_n^{\frac{1}{H} - 1} U_{0, t}(\epsilon_n, w - a)
\end{align*} 
exists for every $t \geq 0$.
However, as observed by Chacon et al. \cite{chacon1981}, the quantitative estimate in 
Theorem~\ref{thm:local-time-level-crossing} implies more strongly that we can take $\Omega_{a, \bm{\epsilon}}$ 
uniformly over $a$ and $\bm{\epsilon}$.
Furthermore, we can remove the polynomial decaying condition. 

The arguments below are essentially given in \cite{chacon1981} and \cite{Lemieux_1983}, but we repeat them for the reader's convenience.
We begin with the following lemma. 
\begin{lemma}[Uniform convergence over grids]\label{lem:U_local_time_R_k_bound}
  Let $H \in (0, 1/2)$ and $t \in (0, \infty)$.
  We define the grid
  \begin{align*}
  G_k:= \{ i k^{-7} : i\in\mathbb{Z} , |i|\leq k^8 \}, \quad k \in \mathbb{N}\cup\{0\}.
  \end{align*}
  We then have 
  \begin{align*}
    \lim_{k \to \infty} 
    \max_{x \in G_k} \abs[\big]{k^{-6(\frac{1}{H}-1)} U_{0,t}(k^{-6},B-x) - \frac{\mathfrak{c}_H}{2} L_t(x)} = 0 \quad \text{almost surely.}
  \end{align*}
\end{lemma}
\begin{proof}
  In the notation of Theorem~\ref{thm:local-time-level-crossing}, we have 
  \begin{align*}
    \max_{x \in G_k} 
    \abs[\big]{k^{-6(\frac{1}{H}-1)} U_{0,t}(k^{-6},B-x) - \frac{\mathfrak{c}_H}{2} L_t(x)} 
    \leq \zeta^{-1} \sup_{x \in \mathbb{R}} L_t(x) + \max_{x \in G_k} R_{k, \zeta, x}. 
  \end{align*}
  Since $x \mapsto L_t(x)$ is continuous and $L_t(\cdot)$ is supported on 
  \begin{align*}
    \{x \in \mathbb{R} : \abs{x} \leq \norm{B}_{L^{\infty}([0, t])} \},
  \end{align*} 
  we see that $\sup_{x \in \mathbb{R}} L_t(x) < \infty$ a.s. 
  By Theorem~\ref{thm:local-time-level-crossing}, 
  \begin{align*}
    \norm{\max_{x \in G_k} R_{k, \zeta, x}}_p^p 
    \leq \sum_{x \in G_k} \norm{R_{k, \zeta, x}}_p^p 
    \lesssim_{p, \zeta} k^{-p \delta + 8},
  \end{align*}
  where $\delta$ is independent of $p$. Since $p$ can be arbitrarily large, 
  the Borel--Cantelli lemma implies that almost surely we have 
  \begin{align*}
    \lim_{k \to \infty}
    \max_{x \in G_k} R_{k, \zeta, x} = 0 
  \end{align*}
  and  
  \begin{align*}
    \limsup_{k \to \infty} \max_{x \in G_k} \abs[\big]{k^{-6(\frac{1}{H}-1)} U_{0,t}(k^{-6},B-x) - \frac{\mathfrak{c}_H}{2} L_t(x)} \leq \zeta^{-1} \sup_{x \in \mathbb{R}} L_t(x).
  \end{align*}
  Since $\zeta$ is arbitrary, we complete the proof.
\end{proof}

\begin{theorem}[{Uniform convergence to local time, \cite[Theorem~II.2.4]{Lemieux_1983}}]\label{thm:lemieux_local_time}
  Let $H \in (0, 1/2)$ and $T \in (0, \infty)$. 
  Almost surely, we have 
  \[\lim_{\epsilon\to 0} \sup_{t \leq T} \sup_{a\in \mathbb{R}} \abs[\big]{\epsilon^{\frac{1}{H}-1} U_{0,t}(\epsilon,B-a) - \frac{\mathfrak{c}_H}{2} L_t(a)} = 0. 
 \]
\end{theorem}
\begin{proof}
Firstly, by an elementary argument using monotonicity of $U$ and continuity of $L$ \cite[Note after Theorem~II.2.4]{Lemieux_1983}, 
it suffices to prove that for each $t \in (0, \infty)$ we have
\begin{align*}
  \lim_{\epsilon\to 0}  \sup_{a\in \mathbb{R}} \abs[\big]{\epsilon^{\frac{1}{H}-1} U_{0,t}(\epsilon,B-a) - \frac{\mathfrak{c}_H}{2} L_t(a)} = 0
  \quad \text{a.s.} 
\end{align*}

By Lemma~\ref{lem:U_local_time_R_k_bound}, we can find an 
$\Omega_{1} \subseteq \Omega$ with $\mathbb{P}(\Omega_1) = 1$ such that 
for any $\delta \in (0, 1)$ and $\omega \in \Omega_1$ there exists an $N = N(\delta, \omega)$
with the following inequalities: 

\begin{gather}
     (k-1)^{-6}-(k-1)^{-7}>k^{-6}  \quad \forall k\geq N, \label{eq:k_6_k_7} \\
    \norm{B(\omega)}_{L^{\infty}([0, t])} < N - 1, \label{eq:B_L_infty_bound_by_N} \\
    \sup_{k \geq N} \max_{x\in G_k}\abs[\big]{k^{-6(\frac{1}{H}-1)} U_{0,t}(k^{-6},B(\omega)-x) 
    - \frac{\mathfrak{c}_H}{2} L_t(x)(\omega)}<\delta. 
    \label{eq:U_and_L_over_grid}
\end{gather}

The argument below holds on the event $\Omega_{1}$.
For $\epsilon\leq (N+1)^{-6}$, there exists a unique $m = m_{\epsilon}\geq N+1$ such that 
\begin{align*}
  (m+1)^{-6}< \epsilon \leq m^{-6}.
\end{align*}
If $|x| \geq N- 1$, then by \eqref{eq:B_L_infty_bound_by_N} we have $L^H_t(x) = 0$. 
On the other hand, if $|x| < N - 1$, 
then we define   
\begin{align*}
  x_k \defby \max_{y \in G_k} \{y \leq x\}
\end{align*}
for all $k \geq N$. Since $x<x_{m-1}+ (m-1)^{-7}$, we have
\begin{itemize}
    \item $x_{m-1}\leq x<x+\epsilon <x_{m-1}+(m-1)^{-7} + m^{-6} \leq x_{m-1}+(m-1)^{-6}$ and 
    \item $x<x_{m+2}+(m+2)^{-7}<x_{m+2}+(m+2)^{-7}+(m+2)^{-6}<x+\epsilon$,
\end{itemize}
where \eqref{eq:k_6_k_7} are applied in both items.
Hence, defining the two sets $I_{m-1}$ and $\bar I_{m+2}$ as
\begin{equation*}
  I_{m-1}:=\left[x_{m-1},x_{m-1}+(m-1)^{-6}\right], \quad
  \bar I_{m+2}:=\left[\bar x_{m+2}, \bar x_{m+2}+(m+2)^{-6}\right], 
\end{equation*}
where $\bar x_{m+2} \defby x_{m+2} + (m+2)^{-7}$, 
we have the inclusions
\begin{align}\label{eq:I_x_epsilon}
  \bar I_{m+2} \subseteq [x,x+\epsilon] \subseteq I_{m-1}.
\end{align}

Now we move to the bound on $U$. We first observe the monotonicity of $U$: 
\begin{align*}
  U_{0, t}(\epsilon_1, B - x_1) \leq U_{0, t}(\epsilon_2, B - x_2)
\end{align*}
provided that $[x_2, x_2 + \epsilon_2] \subseteq [x_1, x_1 + \epsilon_1]$. 
The relation \eqref{eq:I_x_epsilon} thus yields 
\begin{align*}
  U_{0, t}((m-1)^{-6}, B - x_{m-1})
  \leq U_{0, t}(\epsilon, B - x) \leq U_{0, t}((m+2)^{-6}, B - \bar x_{m+2}).
\end{align*}
Hence, 
\begin{equation}\label{eq:U_and_L_bound}
  \sup_{x \in \mathbb{R}} 
  \abs[\big]{\epsilon^{\frac{1}{H}-1} U_{0,t}(\epsilon,B-x) - \frac{\mathfrak{c}_H}{2} L_t(x)} 
  \leq A_{\epsilon} + \bar A_{\epsilon} + \frac{\mathfrak{c}_H}{2} 
  \sup_{x, y: \abs{x-y} \leq 2 \epsilon} \abs{L_t(x) - L_t(y)},
\end{equation}
where 
\begin{align*}
  A_{\epsilon} &\defby \sup_{x\in G_{m_{\epsilon}-1}} 
  \abs[\big]{\epsilon^{\frac{1}{H} - 1} U_{0,t}((m_{\epsilon}-1)^{-6},B-x) - \frac{\mathfrak{c}_H}{2} L_t(x)}, \\
  \bar A_{\epsilon} &\defby \sup_{x\in G_{m_{\epsilon}+2}} 
  \abs[\big]{\epsilon^{\frac{1}{H} - 1} U_{0,t}((m_{\epsilon}+2)^{-6},B-x) - \frac{\mathfrak{c}_H}{2} L_t(x)}.
\end{align*}
Due to the uniform continuity of $L_t(\cdot)$, the last term of \eqref{eq:U_and_L_bound} converges to $0$.
To estimate $A_{\epsilon}$, we observe the bound 
\begin{multline*}
  A_{\epsilon} \leq \sup_{x\in G_{m_{\epsilon}-1}} 
  \abs[\big]{(m_{\epsilon} - 1)^{-6(\frac{1}{H} - 1)} U_{0,t}((m_{\epsilon}-1)^{-6},B-x) - \frac{\mathfrak{c}_H}{2} L_t(x)} \\
  + \sup_{x\in G_{m_{\epsilon}-1}}
  \abs[\big]{\{(m_{\epsilon} - 1)^{-6(\frac{1}{H} - 1)} - \epsilon^{\frac{1}{H} - 1}\} U_{0,t}((m_{\epsilon}-1)^{-6},B-x)}. 
\end{multline*}
By \eqref{eq:U_and_L_over_grid}, 
\begin{align*}
  \limsup_{\epsilon \to 0} \sup_{x\in G_{m_{\epsilon}-1}} 
  \abs[\big]{(m_{\epsilon} - 1)^{-6(\frac{1}{H} - 1)} U_{0,t}((m_{\epsilon}-1)^{-6},B-x) - \frac{\mathfrak{c}_H}{2} L_t(x)} \leq  \delta.
\end{align*}
On the other hand, 
\begin{align*}
  \abs{(m_{\epsilon} - 1)^{-6(\frac{1}{H} - 1)} - \epsilon^{\frac{1}{H} - 1}}
  &\lesssim \abs{\epsilon - (m_{\epsilon} - 1)^{-6}}^{\frac{1}{H} - 1}\\
  &\lesssim  (m_{\epsilon} - 1)^{-7(\frac{1}{H} - 1)}.
\end{align*}
As \eqref{eq:U_and_L_over_grid} implies 
\begin{align*}
  \sup_{\epsilon \in (0, 1)} \sup_{x \in G_{m_{\epsilon}-1}} (m_{\epsilon} - 1)^{-6(\frac{1}{H} - 1)}  U_{0,t}((m_{\epsilon}-1)^{-6},B-x) < \infty,
\end{align*}
we obtain 
\begin{align*}
  \lim_{\epsilon \to 0} 
  \sup_{x\in G_{m_{\epsilon}-1}}
  \abs[\big]{\{(m_{\epsilon} - 1)^{-6(\frac{1}{H} - 1)} - \epsilon^{\frac{1}{H} - 1}\} U_{0,t}((m_{\epsilon}-1)^{-6},B-x)} = 0.
\end{align*}
Hence, we get $\limsup_{\epsilon \to 0} A_{\epsilon} \leq \delta$, and we get a similar estimate for $\bar{A}_{\epsilon}$. 
Recalling \eqref{eq:U_and_L_bound}, this implies 
\begin{align*}
  \lim_{\epsilon \to 0} \sup_{x \in \mathbb{R}} 
  \abs[\big]{\epsilon^{\frac{1}{H}-1} U_{0,t}(\epsilon,B-x) - \frac{\mathfrak{c}_H}{2} L_t(x)} \leq 2 \delta.
\end{align*}
Since $\delta$ is arbitrary, we conclude the proof.
\end{proof}

Recall the total number $D_{s, t}(\epsilon, w)$ of downcrossings from \eqref{eq:def_of_D} 
and the variation $V_{s, t}(\mathrm{P}, w)$ along Lebesgue partition $\mathrm{P}$ from \eqref{eq:def_vr_lebesgue}. 
Since the total number of upcrossings and that of downcrossings can differ by at most $1$, 
almost surely we have
\begin{align*}
  \lim_{\epsilon \to 0} \sup_{t \leq T} \sup_{a \in \mathbb{R}} 
  \abs[\big]{ \epsilon^{\frac{1}{H} - 1} D_{0, t}(\epsilon, B - a) - \frac{\mathfrak{c}_H}{2} L_t(a)}
  = 0 \quad \forall T \geq 0,
\end{align*}
or 
\begin{align}\label{eq:U_D_limit}
  \lim_{\epsilon \to 0} \sup_{t \leq T} \sup_{a \in \mathbb{R}} 
  \abs[\big]{ \epsilon^{\frac{1}{H} - 1} (U_{0, t}(\epsilon, B - a) + D_{0, t}(\epsilon, B - a)) - \mathfrak{c}_H L_t(a)}
  = 0 \quad \forall T \geq 0.
\end{align}
\begin{theorem}[{Uniform convergence of variation, \cite[Proposition~III.2.1]{Lemieux_1983}}]\label{thm:lemieux_variation}
  Let $H \in (0, 1/2)$ and $T \in (0, \infty)$. 
  Almost surely, we have 
  \begin{equation*}
      \lim_{\epsilon \to 0} \sup_{t \leq T} \sup_{\substack{\mathrm{P} : \text{ \emph{partition of} $\mathbb{R}$}, \\ \abs{\mathrm{P}} \leq \epsilon}}
      \abs{V_{0, t}(\mathrm{P}, B) - \mathfrak{c}_H t} = 0.
  \end{equation*}
\end{theorem}
\begin{proof}
  We have the identity 
  \begin{align*}
    V_{0, t}(\mathrm{P}, B) 
    = \int_{\mathbb{R}} \sum_{[a, b] \in \mathrm{P}} (b-a)^{\frac{1}{H} - 1} 
    \{ U_{0, t}(b-a, B-a) + D_{0, t}(b-a, B-a) \} \indic_{(a, b]}(x) \mathrm{d} x.
  \end{align*}
  Setting $I \defby [- \norm{B}_{L^{\infty}([0, T])}, \norm{B}_{L^{\infty}([0, T])}]$, 
  the occupation density formula yields 
  \begin{multline*}
    \abs{V_{0, t}(\mathrm{P}, B) - \mathfrak{c}_H t} \\
    \leq \int_I \sum_{[a, b] \in \mathrm{P}} (b-a)^{\frac{1}{H} - 1} 
    \abs{ U_{0, t}(b-a, B-a) + D_{0, t}(b-a, B-a) - \mathfrak{c}_H L_t(x) } \indic_{(a, b]}(x) \mathrm{d} x.
  \end{multline*}
  For $x \in [a, b]$, we have the bound 
  \begin{multline*}
    \abs{ U_{0, t}(b-a, B-a) + D_{0, t}(b-a, B-a) - \mathfrak{c}_H L_t(x) } \\
    \leq \abs{ U_{0, t}(b-a, B-a) + D_{0, t}(b-a, B-a) - \mathfrak{c}_H L_t(a) } \\
    + \mathfrak{c}_H \sup_{a_1, a_2 : \abs{a_1 - a_2} \leq \abs{\mathrm{P}}} \abs{L_t(a_1) - L_t(a_2)}.
  \end{multline*}
  Therefore, 
  \begin{multline*}
    \abs{V_{0, t}(\mathrm{P}, B) - \mathfrak{c}_H t} 
    \leq \abs{I} \times \Big\{ \sup_{\delta \leq \abs{\mathrm{P}}} \sup_{a \in \mathbb{R}} 
    \abs{ U_{0, t}(\delta, B-a) + D_{0, t}(\delta, B-a) - \mathfrak{c}_H L_t(a) } \\
    + 
    \mathfrak{c}_H \sup_{\abs{a_1 - a_2} \leq \abs{\mathrm{P}}} \abs{L_t(a_1) - L_t(a_2)}
    \Big\}.
  \end{multline*}
  %
  In view of \eqref{eq:U_D_limit} and the uniform continuity of $L$, the claim follows.
\end{proof}
\subsection{Horizontally rough function}
%
In a recent work \cite{das2020},  a concept of \emph{quadratic roughness} has been introduced. This pathwise quadratic roughness property ensures an invariant notion of quadratic variation,
i.e., given two (appropriate) partition sequences $\bm{\pi}$ and $\bm{\sigma}$, quadratic roughness of function $x$ implies $[x]_{\bm{\pi}} = [x]_{\bm{\sigma}}$. 
As expected, Brownian motion satisfies this quadratic roughness property. 
In fact, for any \emph{deterministic} partition sequence $\bm{\pi} = (\pi^n)$ with $|\pi^n|\log n \to 0$, we have $[B^{1/2}]_{\bm{\pi}} (t) = t$ almost surely. 
That is there exists $\Omega_{\bm{\pi}} \subset \Omega$ of full $\mathbb{P}$-measure 
such that for all $\omega\in \Omega_{\bm{\pi}},\; [\omega]_{\bm{\pi}} (t)=t$. 
On the other hand, by \cite{Dudley:1973aa} there exists for each $\omega \in \Omega$ a partition $\pi=\pi(\omega)$ such that $[\omega]_{\pi(\omega)}(t) = 0$, and therefore $\cap_{\bm{\pi}}\Omega_{\bm{\pi}} = \emptyset$. 
So even for Brownian motion quadratic roughness does not ensure an almost sure invariance of quadratic variation across all deterministic partitions (partitions purely on time variable). 
So an obvious question is: are there any notion of roughness which ensures almost sure invariance of quadratic variation across a large (uncountable) class of partition sequence? To answer this question we define the notion of \emph{horizontal roughness} (the word `horizontal' represents path dependent Lebesgue type partitions constructed from level crossings).

\begin{definition}[Horizontally rough: an invariance notion for $p$-th variation]
A function $x\in C^0([0,T],\mathbb{R})$ is called \emph{horizontally rough} if for any 
$t \in [0, T]$, $\rho\in \mathbb{R}$ and  $\epsilon=\{\epsilon_n\}$ with $\epsilon_n\downarrow 0$,
  \[\lim_{n\to \infty}\frac{K_{0, t}(\epsilon_n,x+\rho)}{K_{0, t}(\epsilon_n,x)} =1.\]
\end{definition}

\par This notion of horizontal roughness is completely pathwise and scale invariant. If a continuous function $x$ has the horizontally rough property and $x$ also has $p$-th variation along a uniform Lebesgue partition, then $x$ has $p$-th variation along all uniform Lebesgue partition and the $p$-th variation is the same across different uniform Lebesgue partitions. Unlike the notion of quadratic roughness defined in \cite{das2020}, this notion of horizontally rough ensures an invariant notion of $p$-th variation across a large class of Lebesgue type partitions  almost surely i.e. there is a common measure zero set outside which the $p$-th variation along any uniform Lebesgue type partition is the same. 
\begin{example}
  The following are examples of horizontally rough functions.
  \begin{itemize}
  \item From the definition any linear function is horizontally rough.
  \item Using results from \cite{chacon1981,Lemieux_1983}, we can show that Brownian motion and more generally continuous semimartingales are horizontally rough almost surely.
  \item  Theorem~\ref{thm:lemieux_type_result} shows that  fractional Brownian motion with Hurst index $H<1/2$  is horizontally rough almost surely. 
  \end{itemize}
\end{example}

It is interesting to construct horizontally rough functions sampled from a non-Gaussian, non-semimartingale process. We leave this and further properties of such functions as a future work.   


\appendix
\section{An estimate on log-normal distribution}
\begin{lemma}\label{lem:log_normal_moment}
  Let $Z$ be a standard normal distribution, and let $q \in [2, \infty)$. We then have 
  \begin{align*}
    \norm{e^{\lambda Z - \frac{\lambda^2}{2}} - 1}_{L^q(\mathbb{P})} \leq  \sqrt{2 \log 2} \sqrt{q-1} \lambda e^{(q-1) \lambda^2}, \quad \forall \lambda \geq 0.
  \end{align*}
\end{lemma}
\begin{proof}
  By \cite[Equation (1.1)]{nualart06}, we have 
  \begin{align*}
    e^{\lambda Z - \frac{\lambda^2}{2}} - 1 = \sum_{n=1}^{\infty} \lambda^n H_n(Z),
  \end{align*}
  where $H_n$ is $n$th Hermite polynomial. By the triangle inequality, 
  \begin{align*}
    \norm{e^{\lambda Z - \frac{\lambda^2}{2}} - 1}_{L^q(\mathbb{P})} \leq \sum_{n =1}^{\infty} \lambda^n \norm{H_n(Z)}_{L^q(\mathbb{P})}.
  \end{align*}
  The hypercontractivity \cite[Theorem~1.4.1]{nualart06} implies that 
  \begin{align*}
    \norm{H_n(Z)}_{L^q(\mathbb{P})} \leq (q-1)^{\frac{n}{2}} \norm{H_n(Z)}_{L^2(\mathbb{P})}.
  \end{align*}
  Furthermore, by \cite[Lemma~1.1.1]{nualart06}, 
  \begin{align*}
    \norm{H_n(Z)}_{L^2(\mathbb{P})} = \frac{1}{\sqrt{n!}}.
  \end{align*}
  Therefore, 
  \begin{align*}
    \norm{e^{\lambda Z - \frac{\lambda^2}{2}} - 1}_{L^q(\mathbb{P})} &\leq \sqrt{q-1} \lambda \sum_{n=0}^{\infty} \frac{\lambda^n (q-1)^{\frac{n}{2}}}{\sqrt{n!}} \\
    &\leq \sqrt{q-1} \lambda \Big(  \sum_{n=0}^{\infty} \frac{2^{-n}}{n+1} \Big)^{\frac{1}{2}} \Big(  \sum_{n=0}^{\infty} \frac{2^{n} \lambda^{2n} (q-1)^n }{n!}\Big)^{\frac{1}{2}} \\
    &\leq \sqrt{q-1} \lambda \sqrt{2 \log 2} e^{(q-1) \lambda^2}.  \qedhere
  \end{align*}
\end{proof}

\printbibliography[heading=bibintoc]

@article{Harang:2022aa,
	abstract = {We investigate the space-time regularity of the local time associated with Volterra–Lévy processes, including Volterra processes driven by {\$}{\$}{$\backslash$}alpha {\$}{\$}-stable processes for {\$}{\$}{$\backslash$}alpha {$\backslash$}in (0,2{$]$}{\$}{\$}. We show that the spatial regularity of the local time for Volterra–Lévy process is {\$}{\$}{\{}{$\backslash$}mathbb {\{}P{\}}{\}}{\$}{\$}-a.s. inverse proportional to the singularity of the associated Volterra kernel. We apply our results to the investigation of path-wise regularizing effects obtained by perturbation of ordinary differential equations by a Volterra–Lévy process which has sufficiently regular local time. Following along the lines of Harang and Perkowski (2020), we show existence, uniqueness and differentiability of the flow associated with such equations.},
	author = {Harang, Fabian A. and Ling, Chengcheng},
	date = {2022-09-01},
	date-added = {2023-07-03 17:31:47 +0200},
	date-modified = {2023-07-03 17:31:47 +0200},
	doi = {10.1007/s10959-021-01114-4},
	id = {Harang2022},
	%isbn = {1572-9230},
	journal = {Journal of Theoretical Probability},
	number = {3},
	pages = {1706--1735},
	title = {Regularity of Local Times Associated with Volterra–Lévy Processes and Path-Wise Regularization of Stochastic Differential Equations},
	url = {https://doi.org/10.1007/s10959-021-01114-4},
	volume = {35},
	year = {2022},
	bdsk-url-1 = {https://doi.org/10.1007/s10959-021-01114-4}}

@article{perkins81,
	author = {Edwin Perkins},
	issn = {00911798},
	journal = {The Annals of Probability},
	number = {5},
	pages = {800--817},
	publisher = {Institute of Mathematical Statistics},
	title = {A Global Intrinsic Characterization of Brownian Local Time},
	url = {http://www.jstor.org/stable/2243739},
	urldate = {2023-06-29},
	volume = {9},
	year = {1981}}

@article{perkowski15,
	author = {Perkowski, Nicolas and Pr\"{o}mel, David J.},
	doi = {10.1214/EJP.v20-3534},
	fjournal = {Electronic Journal of Probability},
	issn = {1083-6489},
	journal = {Electron. J. Probab.},
	mrclass = {60H05 (60J60 91G80)},
	mrnumber = {3339866},
	mrreviewer = {Volkert Paulsen},
	pages = {no. 46, 15},
	title = {Local times for typical price paths and pathwise {T}anaka formulas},
	url = {https://doi.org/10.1214/EJP.v20-3534},
	volume = {20},
	year = {2015}}

@article{das2020,
	author = {Rama Cont and Purba Das},
	journal = {Bernoulli},
	number = {1},
	pages = {496 -- 522},
	publisher = {Bernoulli Society for Mathematical Statistics and Probability},
	title = {{Quadratic variation and quadratic roughness}},
	volume = {29},
	year = {2023}}

@article{berman1973,
	author = {Berman, Simeon M and Getoor, Ronald},
	journal = {Indiana University Mathematics Journal},
	number = {1},
	pages = {69--94},
	publisher = {JSTOR},
	title = {Local nondeterminism and local times of Gaussian processes},
	volume = {23},
	year = {1973}}

@book{nualart06,
	author = {Nualart, David},
	edition = {Second},
	isbn = {978-3-540-28328-7},
	mrclass = {60-02 (60H07 60H30)},
	mrnumber = {2200233},
	mrreviewer = {Daniel Ocone},
	pages = {xiv+382},
	publisher = {Springer-Verlag, Berlin},
	series = {Probability and its Applications (New York)},
	title = {The {M}alliavin calculus and related topics},
	year = {2006}}

@misc{butkovsky2023stochastic,
	archiveprefix = {arXiv},
	author = {Oleg Butkovsky and Khoa Lê and Leonid Mytnik},
	eprint = {2302.11937},
	primaryclass = {math.PR},
	title = {Stochastic equations with singular drift driven by fractional Brownian motion},
	year = {2023}}

@article{Harang:2021aa,
	author = {Harang, Fabian Andsem and Perkowski, Nicolas},
	journal = {Stochastics and Dynamics},
	number = {08},
	pages = {2140010},
	publisher = {World Scientific},
	title = {{$C^\infty$}- regularization of ODEs perturbed by noise},
	volume = {21},
	year = {2021}}

@article{lochowski2017,
	author = {Rafał M. Łochowski},
	doi = {10.4064/cm6583-3-2017},
	journal = {Colloquium Mathematicum},
	number = {2},
	pages = {301--313},
	publisher = {Institute of Mathematics, Polish Academy of Sciences},
	title = {On a generalisation of the Banach Indicatrix Theorem},
	url = {https://doi.org/10.4064%2Fcm6583-3-2017},
	volume = {148},
	year = 2017}

@article{lochowski2008,
	author = {Rafał Łochowski},
	journal = {Bulletin of the Polish Academy of Sciences. Mathematics},
	language = {eng},
	number = {3},
	pages = {267-281},
	title = {On Truncated Variation of Brownian Motion with Drift},
	url = {http://eudml.org/doc/281146},
	volume = {56},
	year = {2008}}

@book{freedman,
	author = {Freedman, David},
	edition = {Second},
	isbn = {0-387-90805-6},
	mrclass = {60J65 (60J60)},
	mrnumber = {686607},
	pages = {xii+231},
	publisher = {Springer-Verlag, New York-Berlin},
	title = {Brownian motion and diffusion},
	year = {1983}}

@article{Dudley:1973aa,
	author = {R. M. Dudley},
	date = {1973-02-01},
	doi = {10.1214/aop/1176997026},
	journal = {The Annals of Probability},
	journal1 = {The Annals of Probability},
	journal2 = {The Annals of Probability},
	month = {2},
	number = {1},
	pages = {66--103},
	title = {Sample Functions of the Gaussian Process},
	url = {https://doi.org/10.1214/aop/1176997026},
	volume = {1},
	year = {1973}}

@article{feuerverger,
	author = {Andrey Feuerverger and Peter Hall and Andrew T. A. Wood},
	doi = {10.1111/j.1467-9892.1994.},
	journal = {Journal of Time Series Analysis},
	month = {11},
	number = {6},
	pages = {587-606},
	title = {{Estimation Of Fractal Index And Fractal Dimension Of A Gaussian Process By Counting The Number Of Level Crossings}},
	url = {https://ideas.repec.org/a/bla/jtsera/v15y1994i6p587-606.html},
	volume = {15},
	year = 1994}

@article{ananova20,
	author = {Ananova, Anna and Cont, Rama and Xu, Renyuan},
	doi = {http://dx.doi.org/10.2139/ssrn.3723980},
	title = {Excursion Risk},
	url = {http://dx.doi.org/10.2139/ssrn.3723980},
	year = {2020}}

@article{Kim:2022aa,
	author = {Kim, Donghan},
	date = {2022-12-01},
	doi = {10.1007/s10959-022-01159-z},
	id = {Kim2022},
	isbn = {1572--9230},
	journal = {Journal of Theoretical Probability},
	number = {4},
	pages = {2540--2568},
	title = {Local Times for Continuous Paths of Arbitrary Regularity},
	url = {https://doi.org/10.1007/s10959-022-01159-z},
	volume = {35},
	year = {2022}}

@article{Azais:1996aa,
	author = {Jean-Marc Azaïs and Mario Wschebor},
	date = {1996-09-01},
	doi = {10.3150/bj/1178291722},
	journal = {Bernoulli},
	journal1 = {Bernoulli},
	journal2 = {Bernoulli},
	month = {9},
	number = {3},
	pages = {257--270},
	title = {Almost sure oscillation of certain random processes},
	url = {https://doi.org/10.3150/bj/1178291722},
	volume = {2},
	year = {1996}}

@article{azais90,
	author = {Azaïs, Jean-Marc},
	journal = {{Probability and Mathematical Statistics (Pol)}},
	pages = {19-36},
	title = {{Conditions for convergence of number of crossings to the local time. Application to stable processes with independent increments and to gaussian processes}},
	url = {https://hal.inrae.fr/hal-02717652},
	volume = {11},
	year = {1990}}

@book{azais09,
	author = {Azaïs, Jean-Marcs and Wschebor, Mario},
	doi = {10.1002/9780470434642},
	isbn = {978-0-470-40933-6},
	mrclass = {60-02 (60E15 60G05 60G15 60G60 60G70)},
	mrnumber = {2478201},
	mrreviewer = {Anna Amirdjanova},
	pages = {xii+393},
	publisher = {John Wiley \& Sons, Inc., Hoboken, NJ},
	title = {Level sets and extrema of random processes and fields},
	url = {https://doi.org/10.1002/9780470434642},
	year = {2009}}

@article{Fristedt:1983aa,
	author = {Fristedt, Bert and Taylor, S. J.},
	date = {1983-03-01},
	doi = {10.1007/BF00532164},
	id = {Fristedt1983},
	isbn = {1432-2064},
	journal = {Zeitschrift für Wahrscheinlichkeitstheorie und Verwandte Gebiete},
	number = {1},
	pages = {73--112},
	title = {Constructions of local time for a Markov process},
	url = {https://doi.org/10.1007/BF00532164},
	volume = {62},
	year = {1983}}

@article{Kac:1943aa,
	author = {M. Kac},
	date = {1943-04-01},
	journal = {Bulletin of the American Mathematical Society},
	journal1 = {Bulletin of the American Mathematical Society},
	journal2 = {Bulletin of the American Mathematical Society},
	month = {4},
	number = {4},
	pages = {314--320},
	title = {On the average number of real roots of a random algebraic equation},
	volume = {49},
	year = {1943}}

@article{Rice:1945aa,
	author = {S. O. Rice},
	doi = {10.1002/j.1538-7305.1945.tb00453.x},
	journal = {The Bell System Technical Journal},
	journal1 = {The Bell System Technical Journal},
	journal2 = {The Bell System Technical Journal},
	number = {1},
	pages = {46--156},
	title = {Mathematical analysis of random noise},
	vo = {24},
	volume = {24},
	year = {1945},
	year1 = {Jan. 1945}}

@article{cont19,
	author = {Cont, Rama and Perkowski, Nicolas},
	doi = {10.1090/btran/34},
	fjournal = {Transactions of the American Mathematical Society. Series B},
	journal = {Trans. Amer. Math. Soc. Ser. B},
	mrclass = {60H05 (60H99)},
	mrnumber = {3937343},
	mrreviewer = {Torstein K. Nilssen},
	pages = {161--186},
	title = {Pathwise integration and change of variable formulas for continuous paths with arbitrary regularity},
	url = {https://doi.org/10.1090/btran/34},
	volume = {6},
	year = {2019}}

@article{Gerencser:2022aa,
	author = {Gerencsér, Máté},
	date = {2022-04-03},
	doi = {10.1007/s40072-022-00242-0},
	id = {Gerencsér2022},
	isbn = {2194-041},
	journal = {Stochastics and Partial Differential Equations: Analysis and Computations},
	title = {Regularisation by regular noise},
	url = {https://doi.org/10.1007/s40072-022-00242-0},
	year = {2022}}

@book{ito_mckean,
	author = {It\^{o}, Kiyosi and McKean, Jr., Henry P.},
	mrclass = {60J60 (60G17 60J70)},
	mrnumber = {0345224},
	note = {Second printing, corrected},
	pages = {xv+321},
	publisher = {Springer-Verlag, Berlin-New York},
	series = {Die Grundlehren der mathematischen Wissenschaften, Band 125},
	title = {Diffusion processes and their sample paths},
	year = {1974}}

@inbook{El-Karoui:1978aa,
	author = {El Karoui, Nicole},
	booktitle = {Temps locaux},
	date = {1978},
	id = {AST{\_}1978{\_}{\_}52-53{\_}{\_}63{\_}0},
	la = {fr},
	number = {52-53},
	publisher = {Sociétémathématique de France},
	series = {Astérisque},
	title = {Sur les montées des semi-martingales},
	url = {http://www.numdam.org/item/AST_1978__52-53__63_0/},
	year = {1978}}

@book{revuz_yor,
	author = {Revuz, Daniel and Yor, Marc},
	doi = {10.1007/978-3-662-06400-9},
	edition = {Third},
	isbn = {3-540-64325-7},
	mrclass = {60G44 (60G07 60H05)},
	mrnumber = {1725357},
	pages = {xiv+602},
	publisher = {Springer-Verlag, Berlin},
	series = {Grundlehren der mathematischen Wissenschaften [Fundamental Principles of Mathematical Sciences]},
	title = {Continuous martingales and {B}rownian motion},
	url = {https://doi.org/10.1007/978-3-662-06400-9},
	volume = {293},
	year = {1999}}

@incollection{Krug98,
	author = {Krug, J.},
	fjournal = {Markov Processes and Related Fields},
	issn = {1024-2953},
	journal = {Markov Process. Related Fields},
	mrclass = {60G17 (60G18 82C41)},
	mrnumber = {1677056},
	mrreviewer = {Ralph P. Russo},
	note = {I Brazilian School in Probability (Rio de Janeiro, 1997)},
	number = {4},
	pages = {509--516},
	title = {Persistence of non-{M}arkovian processes related to fractional {B}rownian motion},
	volume = {4},
	year = {1998}}

@article{Geman:1980ve,
	author = {Donald Geman and Joseph Horowitz},
	doi = {10.1214/aop/1176994824},
	journal = {The Annals of Probability},
	journal1 = {The Annals of Probability},
	journal2 = {The Annals of Probability},
	month = {2},
	number = {1},
	pages = {1--67},
	title = {Occupation Densities},
	url = {https://doi.org/10.1214/aop/1176994824},
	volume = {8},
	year = {1980}}

@book{Morters:2010vm,
	address = {Cambridge},
	author = {Mörters, Peter and Peres, Yuval},
	db = {Cambridge Core},
	doi = {DOI: 10.1017/CBO9780511750489},
	dp = {Cambridge University Press},
	isbn = {9780521760188},
	publisher = {Cambridge University Press},
	series = {Cambridge Series in Statistical and Probabilistic Mathematics},
	title = {Brownian Motion},
	url = {https://www.cambridge.org/core/books/brownian-motion/F639B9A8403BD465F896F3E18A9C3382},
	year = {2010}}

@article{Kratz:2006tb,
	author = {Marie F. Kratz},
	doi = {10.1214/154957806000000087},
	journal = {Probability Surveys},
	journal1 = {Probability Surveys},
	journal2 = {Probability Surveys},
	month = {1},
	number = {none},
	pages = {230--288},
	title = {Level crossings and other level functionals of stationary Gaussian processes},
	url = {https://doi.org/10.1214/154957806000000087},
	volume = {3},
	year = {2006}}

@article{le20,
	author = {Khoa Lê},
	doi = {10.1214/20-EJP442},
	journal = {Electronic Journal of Probability},
	journal1 = {Electronic Journal of Probability},
	journal2 = {Electronic Journal of Probability},
	month = {1},
	number = {none},
	pages = {1--55},
	title = {A stochastic sewing lemma and applications},
	url = {https://doi.org/10.1214/20-EJP442},
	volume = {25},
	year = {2020}}

@article{mandelbrot68,
	author = {Mandelbrot, Benoit B. and Van Ness, John W.},
	doi = {10.1137/1010093},
	journal = {SIAM Review},
	number = {4},
	pages = {422-437},
	title = {Fractional Brownian Motions, Fractional Noises and Applications},
	url = {https://doi.org/10.1137/1010093},
	volume = {10},
	year = {1968}}

@misc{matsuda22,
	author = {Matsuda, Toyomu and Perkowski, Nicolas},
	copyright = {Creative Commons Attribution 4.0 International},
	doi = {10.48550/ARXIV.2206.01686},
	publisher = {arXiv},
	title = {An extension of the stochastic sewing lemma and applications to fractional stochastic calculus},
	url = {https://arxiv.org/abs/2206.01686},
	year = {2022}}

@book{Durrett2019,
	address = {Cambridge},
	author = {Durrett, Rick},
	db = {Cambridge Core},
	doi = {DOI: 10.1017/9781108591034},
	dp = {Cambridge University Press},
	edition = {5},
	isbn = {9781108473682},
	publisher = {Cambridge University Press},
	series = {Cambridge Series in Statistical and Probabilistic Mathematics},
	title = {Probability: Theory and Examples},
	url = {https://www.cambridge.org/core/books/probability/DD9A1907F810BB14CCFF022CDFC5677A},
	year = {2019}}

@article{picard08,
	author = {Picard, Jean},
	doi = {10.1214/07-aop388},
	issn = {0091-1798},
	journal = {The Annals of Probability},
	number = {6},
	pages = {2235-2279},
	title = {A tree approach to p-variation and to integration},
	type = {Journal Article},
	url = {https://dx.doi.org/10.1214/07-aop388},
	volume = {36},
	year = {2008}}

@inbook{picard11,
	address = {Berlin, Heidelberg},
	author = {Picard, Jean},
	booktitle = {Séminaire de Probabilités XLIII},
	doi = {10.1007/978-3-642-15217-7_1},
	editor = {Donati-Martin, Catherine and Lejay, Antoine and Rouault, Alain},
	isbn = {978-3-642-15217-7},
	pages = {3--70},
	publisher = {Springer Berlin Heidelberg},
	title = {Representation Formulae for the Fractional Brownian Motion},
	url = {https://doi.org/10.1007/978-3-642-15217-7_1},
	year = {2011}}

@mastersthesis{Lemieux_1983,
	author = {Lemieux, Marc},
	collection = {Retrospective Theses and Dissertations, 1919-2007},
	doi = {http://dx.doi.org/10.14288/1.0080232},
	school = {University of British Columbia},
	series = {Retrospective Theses and Dissertations, 1919-2007},
	title = {On the quadratic variation of semi-martingales},
	url = {https://open.library.ubc.ca/collections/ubctheses/831/items/1.0080232},
	year = {1983}}

@article{davis2018,
	author = {Davis, Mark and Obłój, Jan and Siorpaes, Pietro},
	doi = {10.1214/16-AIHP792},
	fjournal = {Annales de l'Institut Henri Poincaré, Probabilités et Statistiques},
	journal = {Ann. Inst. H. Poincaré Probab. Statist.},
	month = {02},
	number = {1},
	pages = {1--21},
	publisher = {Institut Henri Poincaré},
	title = {Pathwise stochastic calculus with local times},
	url = {https://doi.org/10.1214/16-AIHP792},
	volume = {54},
	year = {2018}}

@book{levy1965,
	author = {Lévy, Paul},
	mrclass = {60.00 (60.40)},
	mrnumber = {0190953},
	mrreviewer = {H. P. McKean, Jr.},
	pages = {vi+438},
	publisher = {Gauthier-Villars \& Cie, Paris},
	title = {Processus stochastiques et mouvement brownien},
	year = {1948}}

@article{levy1940,
	author = {L\'evy, Paul},
	journal = {American Journal of Mathematics},
	pages = {487-550},
	title = {Le mouvement brownien plan},
	volume = {62},
	year = {1940}}

@article{RL2021,
	author = {Rafał M. Łochowski and Jan Obłój and David J. Prömel and Pietro Siorpaes},
	doi = {10.1214/21-EJP638},
	journal = {Electronic Journal of Probability},
	number = {none},
	pages = {1 -- 29},
	publisher = {Institute of Mathematical Statistics and Bernoulli Society},
	title = {{Local times and Tanaka--Meyer formulae for càdlàg paths}},
	url = {https://doi.org/10.1214/21-EJP638},
	volume = {26},
	year = {2021}}

@article{chacon1981,
	author = {Chacon, Rafael V and Le Jan, Yves and Perkins, E and Taylor, SJ},
	journal = {Zeitschrift für Wahrscheinlichkeitstheorie und Verwandte Gebiete},
	pages = {197--211},
	publisher = {Springer},
	title = {Generalised arc length for Brownian motion and Lévy processes},
	volume = {57},
	year = {1981}}
\end{document}